\documentclass[a4paper,reqno,11pt]{amsart}

\usepackage{amsmath,amsfonts,amssymb,amsthm,amscd}

\usepackage{graphicx,epsfig}
\usepackage{psfrag}
\usepackage{perpage}
\usepackage{url}
\usepackage{color}
\usepackage[english]{babel}
\usepackage{bbm}
\usepackage{epstopdf}
\usepackage[utf8]{inputenc}
\usepackage[T1]{fontenc}
\usepackage{microtype}
\usepackage{hyperref}
\usepackage{mathabx}

\usepackage[a4paper,scale={0.72,0.74},marginratio={1:1},footskip=7mm,headsep=10mm]{geometry}

\numberwithin{equation}{section}

\newtheorem{theorem}{Theorem}[section]
\newtheorem{lemma}[theorem]{Lemma}
\newtheorem{proposition}[theorem]{Proposition}
\newtheorem{corollary}[theorem]{Corollary}
\newtheorem{remark}[theorem]{Remark}

\usepackage{tikz}
\usetikzlibrary{arrows,automata,positioning,calc,shapes,decorations.pathreplacing,
decorations.markings,shapes.misc,petri,topaths}
\usepackage{pgfplots}
\pgfplotsset{compat=newest}
\usetikzlibrary{plotmarks}
\usepackage{grffile}
\newlength\figureheight
\newlength\figurewidth
\pgfplotsset{
tick label style={font=\scriptsize},
label style={font=\footnotesize},
legend style={font=\footnotesize},
every axis plot/.append style={very thick}
}

\usepackage{rotating}
\usepackage{amsbsy,enumerate}
\usepackage{graphicx}
\usepackage{ccaption}
\usepackage{comment}
\usepackage{mathrsfs} 

\renewcommand{\hat}{\widehat}

\newcommand{{\paa}[1]}{p_{00,#1}}
\newcommand{{\pab}[1]}{p_{01,#1}}
\newcommand{{\pba}[1]}{p_{10,#1}}
\newcommand{{\pbb}[1]}{p_{11,#1}}

\DeclareMathOperator*{\argmin}{arg\,min}
\DeclareMathOperator*{\argmax}{arg\,max}

\newcommand{\eee}{\mathrm{e}}

\newcommand{\ddd}{\mathrm{d}}

\makeatletter
\renewcommand{\fnum@figure}[1]{\textbf{\figurename~\thefigure}. }
\renewcommand{\fnum@table}[1]{\textbf{\tablename~\thetable}. }
\makeatother

\allowdisplaybreaks

\begin{document}

\title[Sample-path large deviations for graphons]{A sample-path large deviation principle\\ 
for dynamic Erd\H{o}s-R\'enyi random graphs}

\author{Peter Braunsteins}
\address{Korteweg-de-Vries Instituut, Universiteit van Amsterdam, PO Box 94248, 1090 GE Amsterdam, The Netherlands}
\email{pbraunsteins@gmail.com}

\author{Frank den Hollander}
\address{Mathematisch Instituut, Universiteit Leiden, PO Box 9512, 2300 RA Leiden, The Netherlands}
\email{denholla@math.leidenuniv.nl}

\author{Michel Mandjes}
\address{Korteweg-de Vries Instituut, Universiteit van Amsterdam, PO Box 94248, 1090 GE Amsterdam, The Netherlands}
\email{M.H.R.Mandjes@uva.nl}

\date{\today}

\begin{abstract}
We consider a dynamic Erd\H{o}s-R\'enyi random graph (ERRG) on $n$ vertices in which each edge switches on at rate $\lambda$ and switches off at rate $\mu$, independently of other edges. The focus is on the analysis of the evolution of the associated empirical graphon in the limit as $n\to\infty$. Our main result is a large deviation principle (LDP) for the sample path of the empirical graphon observed until a fixed time horizon. The rate is $\binom{n}{2}$, the rate function is a specific action integral on the space of graphon trajectories. We apply the LDP to identify (i)~the most likely path that starting from a constant graphon creates a graphon with an atypically large density of $d$-regular subgraphs, and (ii)~the mostly likely path between two given graphons. It turns out that bifurcations may occur in the solutions of associated variational problems.  

\vspace{0.5cm}
\noindent

\\
\emph{Acknowledgment.}
The work in this paper was supported by the Netherlands Organisation for Scientific Research (NWO) through Gravitation-grant NETWORKS-024.002.003.

\end{abstract}

\maketitle

\newpage

\parindent=0.5cm


\section{Introduction and main results}
\label{S1}

Section~\ref{S1.1} provides motivation and background, Section~\ref{S1.2} introduces graphs and graphons, Section~\ref{S1.3} recalls the LDP for the inhomogeneous ERRG, Section~\ref{S1.4} defines a switching dynamics for the ERRG, Section~\ref{S1.5} states the sample-path LDP for the latter, while Section~\ref{S1.6} offers a brief discussion and announces two applications.
 

\subsection{Motivation and background}
\label{S1.1}

Graphons arise as limits of dense graphs, i.e., graphs in which the number of edges is of the order of the square of the number of vertices. The theory of graphons -- developed in \cite{LSa}, \cite{LSb}, \cite{BCLSVa}, \cite{BCLSVb} -- aims to capture the limiting behaviour of large dense graphs in terms of their subgraph densities (see \cite{L} for an overview). Both typical and atypical behaviour of random graphs and their associated graphons have been analysed, including LDPs for homogeneous and inhomogeneous Erd\H{o}s-R\'enyi random graphs  \cite{CV}, \cite{DS}. 

Most of the theory focusses on \emph{static} random graphons, although recently some attempts have been made to include \emph{dynamic} random graphons \cite{R}, \cite{Ca}, \cite{Cb}, \cite{CK}, \cite{AHR}. The goal of the present paper is to generalise the LDP in \cite{CV} to a sample-path LDP for a dynamic random graph in which the edges switch on and off in a random fashion. The equilibrium of the dynamics coincides with the setup of \cite{CV}, so that our sample-path LDP is a true dynamic version of the static LDP derived in \cite{CV}. The corresponding large deviation rate function turns out to be an action integral. We consider two applications that look at optimal paths for graphons that realise a prescribed large deviation. We find that bifurcations may occur in the solutions of the associated variational problems.     


\subsection{Graphs and graphons}
\label{S1.2}

There is a natural way to embed a simple graph on $n$ vertices in a space of functions called \emph{graphons}. Let $\mathscr{W}$ be the space of functions $h\colon\,[0,1]^2 \to [0,1]$ such that $h(x,y)=h(y,x)$ for all $(x,y) \in [0,1]^2$, formed after taking the quotient with respect to the equivalence relation of almost everywhere equality. A finite simple graph $G$ on $n$ vertices can be represented as a graphon $h^G \in \mathscr{W}$ by setting
\begin{equation}
h^G(x,y) := \begin{cases}
1 &\quad \text{if there is an edge between vertex }\lceil nx \rceil\text{ and vertex }\lceil ny \rceil, \\
0 &\quad \text{otherwise.}
\end{cases} 
\end{equation}
This object is referred to as an \emph{empirical graphon} and has a block structure (see Figure~\ref{Graphon}). The space of graphons $\mathscr{W}$ is endowed with the \emph{cut distance}
\begin{equation}
\label{cutdist}
d_\square(h_1, h_2) := \sup_{S,T \subseteq [0,1]} \left| \int_{S \times T} \ddd x\, \ddd y\, 
[h_1(x,y)- h_2(x,y)] \right|, \quad h_1,h_2 \in \mathscr{W}.
\end{equation}
The space $(\mathscr{W}, d_\square)$ is not compact.

\begin{figure}[htbp]
\includegraphics[scale=0.75]{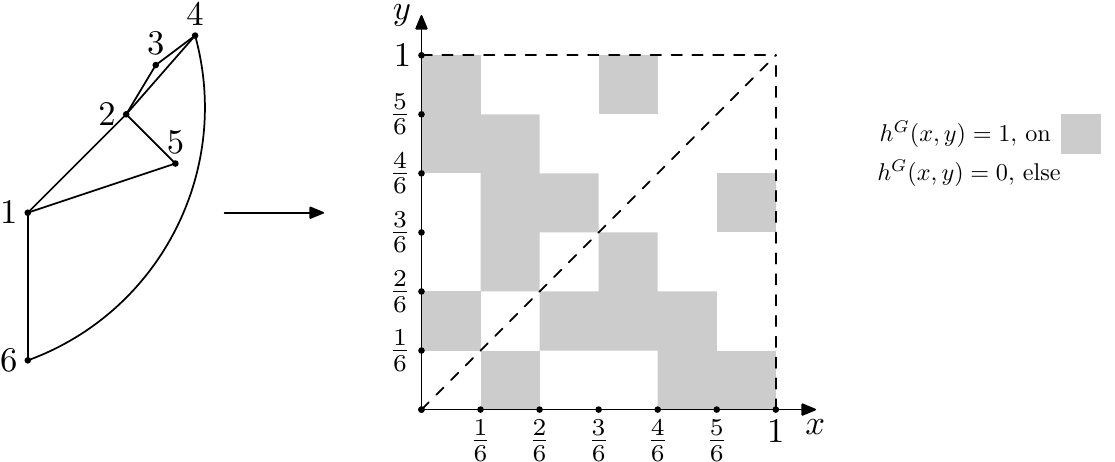}
\caption{An example of an empirical graphon.}
\label{Graphon}
\end{figure}

On $\mathscr{W}$ there is a natural equivalence relation, referred to as `$\sim$'. More precisely, with $\mathscr{M}$ denoting the set of measure-preserving bijections $\sigma\colon\, [0,1] \to [0,1]$, we write $h_1(x,y) \sim h_2(x,y)$ when there exists a $\sigma \in \mathscr{M}$ such that $h_1(x,y) = h_2(\sigma(x), \sigma(y))$ for all $(x,y) \in [0,1]^2$. This equivalence relation induces the quotient space $(\tilde{\mathscr{W}}, \delta_\square)$, where $\delta_\square$ is the \emph{cut metric} defined by 
\begin{equation}
\delta_\square(\tilde h_1, \tilde h_2) := \inf_{\sigma_1, \sigma_2 \in \mathscr{M}} d_\square (h_1^{\sigma_1}, h_2^{\sigma_2}), \quad \tilde h_1, \tilde h_2 \in \tilde{\mathscr{W}} . 
\end{equation}
The space $(\tilde{\mathscr{W}}, \delta_\square)$ {\em is} compact \cite[Lemma 8]{LSa}. 

Suppose that $H$ is a simple graph on $k$ vertices. The \emph{homomorphism density} of $H$ in $G\supseteq H$ is defined as 
\begin{equation}
t(H,G) = t(H,h^G) := \int_{[0,1]^k} \ddd x_1\, \dots\, \ddd x_k\, \prod_{\{i,j\} \in E(H) } h^G(x_i,x_j),
\end{equation}
where $E(H)$ is the set of edges of $H$ and $k=|E(H)|$. The homomorphism densities are continuous with respect to the cut metric  \cite[Proposition 3.2]{C}.


\subsection{LDP for the inhomogeneous ERRG}
\label{S1.3}

Let $r \in \mathscr{W}$ be a \emph{reference graphon} satisfying
\begin{equation}
\label{Assbasic}
\exists\,\eta>0\colon\qquad \eta \leq r(x,y) \leq 1-\eta \quad \forall\,x,y \in [0,1]^2.
\end{equation}
Fix $n \in \mathbb{N}$ and consider the random graph $G_n$ with vertex set $[n]=\{1, \dots, n \}$ where the pair of vertices $i,j \in [n]$, $i \neq j$, is connected by an edge with probability $r(\tfrac{i}{n}, \tfrac{j}{n})$, independently of other pairs of vertices. Write $\mathbb{P}_n$ to denote the law of $G_n$. Use the same symbol for the law on $\mathscr{W}$ induced by the map that associates with the graph $G_n$ its graphon $h^{G_n}$. Write $\tilde{\mathbb{P}}_n$ to denote the law of $\tilde{h}^{G_n}$, the equivalence class associated with $h^{G_n}$.

The following LDP has been proven in \cite{DS} and is an extension of the celebrated LDP for the homogeneous ERRG derived in \cite{CV}. 

\begin{theorem}{\bf [LDP for inhomogeneous ERRG]}
\label{thm:LDPinhom}
Subject to \eqref{Assbasic}, the sequence of probability measures $(\tilde{\mathbb{P}}_n)_{n\in\mathbb{N}}$ satisfies the LDP on $(\tilde{\mathscr{W}},\delta_{\square})$ with rate $\binom{n}{2}$, i.e., 
\begin{equation}
\begin{aligned}
\limsup_{n \to \infty} \frac{1}{\binom{n}{2}} \log \tilde{\mathbb{P}}_n (\mathcal{C}) 
&\leq - \inf_{\tilde h \in \mathcal{C}} J^*_r(\tilde h)
&\forall\,\mathcal{C} \subseteq \tilde{\mathscr{W}} \text{ closed},\\ 
\liminf_{n \to \infty} \frac{1}{\binom{n}{2}} \log \tilde{\mathbb{P}}_n(\mathcal{O}) 
&\geq - \inf_{\tilde h \in \mathcal{O}} J^*_r(\tilde h)
&\forall\,\mathcal{O} \subseteq \tilde{\mathscr{W}} \text{ open}.
\end{aligned}
\end{equation}
Here the rate function $J^*_r\colon\,\tilde{\mathscr{W}} \to \mathbb{R}$ is the lower semi-continuous envelope of the function $J_r$ given by
\begin{equation}
\label{Jrhdef}
J_r(\tilde h) = \inf_{\sigma \in \mathscr{M}} I_r(h^\sigma),
\end{equation}
where $h$ is any representative of $\tilde h$ and 
\begin{equation}
\label{Irhdef}
I_r(h) := \int_{[0,1]^2} \ddd x\,\ddd y\,\,\mathcal{R}\big(h(x,y) \mid r(x,y)\big), \quad h \in \mathscr{W},
\end{equation}
with
\begin{equation}
\label{Rdef}
\mathcal{R}\big(a \mid b\big) := a \log \tfrac{a}{b} + (1-a) \log \tfrac{1-a}{1-b}
\end{equation}
the relative entropy of two Bernoulli distributions with success probabilities $a \in [0,1]$, $b \in (0,1)$ $($with the convention $0 \log 0 =0$$)$.
\end{theorem}

\noindent
It is clear that $J^*_r$ is a good rate function, i.e., $J^*_r \not\equiv \infty$ and $J^*_r$ has compact level sets.  It was shown in \cite{M*} that \eqref{Assbasic} can be weakened: Theorem~\ref{thm:LDPinhom} holds when $0<r<1$ almost everywhere under the integrability conditions $\log r, \log(1-r) \in L^1([0,1]^2)$. Moreover, it was shown in \cite{M*} that $J_r$ is lower semi-continuous on $\tilde{\mathscr{W}}$, and so $J_r^*=J_r$. In \cite{BCGPS} the case where $r$ is a block graphon is considered, which is allowed to take the value $0$ or $1$ on some blocks.   


\subsection{Dynamics for the inhomogeneous ERRG}
\label{S1.4}

We now allow the edges to alternate between being active and inactive, thereby creating a dynamic version of the setup studied in \cite{CV}. Let $\mathbb{G}_n$ be the set of simple graphs with $n$ vertices. Fix a time horizon $T \in (0,\infty)$. Consider a continuous-time Markov process $\{G_n(t)\}_{t \in [0,T]}$ with state space $\mathbb{G}_n$, starting from a given graph $G_n(0)$. The edges in $G_n(t)$ update \emph{independently} after exponentially distributed times, according to the following rules:
\begin{itemize}
\item[$\circ$] an inactive edge becomes active at rate $\lambda \in (0,\infty)$; 
\item[$\circ$] an active edge becomes inactive at rate $\mu \in (0,\infty)$. 
\end{itemize}
Throughout the paper, the transition rates $\lambda,\mu \in (0,\infty)$ are held fixed. Let $\pab[t]$ ($\pbb[t]$) denote the probability that an initially inactive (active) edge is \emph{active} at time $t$. Then
\begin{equation}
\label{ptdefs}
\pab[t] = \frac{\lambda - \lambda \,\eee^{-t(\lambda + \mu)}}{\lambda + \mu}, 
\qquad  \pbb[t]=\frac{\lambda + \mu \,\eee^{-t(\lambda + \mu)}}{\lambda + \mu}.
\end{equation}

We can represent $\{G_n(t)\}_{t \in [0,T]}$ as a graphon-valued process. Abbreviate 
\begin{equation}
f_{n,t} := h^{G_n(t)}, \qquad f_n := (f_{n,t})_{t \in [0,T]}.
\end{equation} 
Let $\mathscr{W} \times [0,T]$ be the set of $\mathscr{W}$-valued paths on the time interval $[0,T]$. On the space $(\mathscr{W}, d_\square)$, we can define the Skorohod topology on $\mathscr{W}$-valued paths in the usual way, namely,
\begin{equation}
\label{Ddef} 
D=D([0,T], \mathscr{W}) = \mbox{set of c\`adl\`ag paths in } \mathscr{W},
\end{equation}
and equip $D$ with a metric that induces the Skorohod topology. Define 
\begin{equation}
\mu_n(B) := \mathbb{P}_n(f_n \in B), \qquad \tilde\mu_n(B):= \mathbb{P}_n(\tilde f_n \in B),
\end{equation} 
for $B$ in the Borel sigma-algebra induced by the metric.

Note that the initial graphon $f_{n,0}$ effectively plays the role of the reference graphon $ r$ in the static setting of an inhomogeneous ERRG treated in \cite{DS}.


\subsection{Main theorem: sample-path LDP}
\label{S1.5}

In order to state our main theorem (the sample-path LDP in Theorem~\ref{GSPLDP} below), we first state a few simpler LDPs.


\subsubsection{LDP for local edge density}

Fix $t \in [0,T]$, $(x,y) \in [0,1]^2\setminus D^*$, with $D^*$ the diagonal, and $\Delta>0$ small enough so that $[x,x+\Delta) \times [y,y+\Delta) \cap D^* = \emptyset$. Let 
\begin{equation}
\begin{aligned}
\widebar u_{t,n} = \widebar u_{t,n}(x,y) &= \frac{1}{\Delta^2} \int_{[x,x + \Delta) \times [y, y + \Delta)} 
\ddd \widebar x\,\ddd \widebar y\, f_{n,t}(\widebar x,\widebar y)\\
&= \frac{1}{(n\Delta)^2} \sum_{(i,j) \in n [x,x+\Delta) \times n [y,y +\Delta)} 
1_{\{i \text{ and } j \text{ are connected in } G_n(t)\}}
\end{aligned}
\end{equation}
denote the proportion of active edges in $[x,x+\Delta) \times [y,y+\Delta)$ at time $t$ (for simplicity we pretend that $nx,ny,n\Delta$ are integer). Fixing an initial proportion of active edges $\widebar u_{0,n} = \widebar u$ that is independent of $n$, we see that the moment generating function of $\widebar u_{t,n}$, defined by $M_{\widebar u_{t,n}}(s) := \mathbb{E}_n[\eee^{s\widebar u_{t,n}}]$, $s \in \mathbb{R}$, equals
\begin{equation}
M_{\widebar u_{t,n}}(s) = [(1-\pbb[t]) + \eee^{sN^{-1}} \pbb[t]]^{N\widebar u} 
[(1- \pab[t]) + \eee^{sN^{-1}} \pab[t]]^{N (1-\widebar u)} 
\end{equation}
with $N:=(n\Delta)^2$, the total number of edges in $[x,x+\Delta) \times [y,y+\Delta)$. Here, $N^{-1}$ is the contribution to $\widebar u_{t,n}$ from a single active edge. Hence
\begin{equation}
\label{cumlim}
\lim_{n\to\infty} N^{-1} \log M_{\widebar u_{t,n}}(vN) = J_{t,v}(\widebar u), \qquad v\in \mathbb{R},
\end{equation}
with
\begin{equation}
\label{ee}
J_{t,v}(\widebar u) := \widebar u \log [(1-\pbb[t]) + \eee^v \pbb[t]] 
+ (1-\widebar u) \log[(1- \pab[t]) + \eee^v \pab[t]].
\end{equation}
Then, by the G\"artner-Ellis theorem \cite[Chapter V]{dH}, the sequence $(\widebar u_{t,n})_{n\in\mathbb{N}}$ satisfies the LDP on $\mathbb{R}$ with rate $N$ and with good rate function
\begin{equation}
\label{ef}
I_{1,t}(\widebar u, w) := \sup_{v \in \mathbb{R}} [vw-J_{t,v}(\widebar u)], \qquad w \in \mathbb{R}, 
\end{equation}
which is the Legendre transform of \eqref{ee}. We use the indices $1,t$ to indicate that \eqref{ef} is the rate function for $1$ time lapse of length $t$.  
For completeness we remark that the supremum in \eqref{ef} allows a closed-form solution. Locally abbreviating $p_{i}:=p_{i1,t}$ and $\bar p_i:=1-p_i$, for $i=0,1$, the optimizing $v$ equals, with $a:= p_{0}p_{1}(1-w)$, $b:=p_0\bar p_1(1-\bar u-w)+\bar p_0p_1(\bar u-w)$, and $c:=-w\bar p_0\bar p_1$,
the familiar $\log((2a)^{-1}(-b\pm \sqrt{b^2-4ac}))$. 
Here the positive root should be chosen if $w> \bar u p_1+(1-\bar u)p_0$ (`exponential tilting in the upward direction': target value is larger than the mean)
and the negative root otherwise (`exponential tilting in the downward direction': target value is smaller than the mean).


\subsubsection{Two-point LDP}

If we extend the domain of $I_{1,t}(\widebar u,w)$ in \eqref{ef} to $\mathscr{W}^2$ by putting
\begin{equation}
\label{efg} 
I_{1,t}(u,h) := \int_{[0,1]^2} \ddd x\, \ddd y\,I_{1,t}(u(x,y),h(x,y)),
\end{equation}
then we obtain a candidate rate function for a \emph{two-point} LDP. However, $I_{1,t}$ is not necessarily well defined on $\tilde{\mathscr{W}}^2$ because for $u_1 \sim u_2$ and $h_1 \sim h_2$ it may be that $I_{1,t}(u_1,h_1) \neq I_{1,t}(u_2, h_2)$. To define a valid candidate rate function, we put
\begin{equation}
\label{MCP2}
\tilde I_{1,t}(\tilde u, \tilde h) := \inf_{\sigma_1, \sigma_2 \in \mathscr{M}} I_{1,t}(u^{\sigma_1}, h^{\sigma_2}) 
=  \inf_{\sigma_2 \in \mathscr{M}} I_{1,t}(u, h^{\sigma_2}) = \inf_{\sigma_1 \in \mathscr{M}} I_{1,t}(u^{\sigma_1}, h),
\end{equation}
noting that $I_{1,t}(u^{\sigma_1},h^{\sigma_2})=I_{1,t}(u,h^{\sigma_2 \,\circ\, \sigma_1^{-1}}) = I_{1,t}(u^{\sigma_1 \,\circ\, \sigma_2^{-1}},h)$ and $\sigma_2 \circ \sigma_1^{-1}, \sigma_1 \circ \sigma_2^{-1} \in \mathscr{M}$.

Define
\begin{equation}
\mu_{n,t}(B) := \mathbb{P}_n(f_{n,t} \in B), \qquad \tilde{\mu}_{n,t}(B) := \mathbb{P}_n(\tilde{f}_{n,t} \in B)
\end{equation}
for $B$ in the Borel sigma-algebra. 

\begin{theorem}{\bf [Two-point LDP]}
\label{tTwoLDP}
Suppose that $\lim_{n\to\infty} \delta_\square(\tilde f_{n,0}, \tilde u) = 0$ for some $\tilde u \in \tilde{\mathscr{W}}$. Then the sequence of probability measures $(\tilde{\mu}_{n,T})_{n\in\mathbb{N}}$ satisfies the LDP on $\tilde{\mathscr{W}}$ with rate ${n \choose 2}$ and with good rate function $\tilde I_{1,T}(\tilde u, \tilde h)$.
\end{theorem}


\subsubsection{Multi-point LDP}

The \emph{multi-point} candidate rate function follows from the two-point candidate rate function by iteration. Let ${\mathscr J}$ denote the collection of all ordered finite subsets of $[0,T]$, i.e., $j \in {\mathscr J}$ if $j=(t_0,t_1,\dots,t_k)$ with $0=t_0 < t_1 < \dots < t_k=T$ for some $k = |j| \in \mathbb{N}$. For $\tilde g \in \tilde{\mathscr{W}}\times [0,T]$ and $j\in {\mathscr J}$, let \begin{equation}
p_j(\tilde g) = (\tilde g_{t_0}, \tilde g_{t_1}, \dots, \tilde g_{t_{|j|}}) \in \tilde{\mathscr{W}}^{|j|+1}.
\end{equation} 

\begin{theorem}{\bf [Multi-point LDP]}
\label{tMultiLDP}
Suppose that $\lim_{n\to\infty} \delta_\square(\tilde f_{n,0}, \tilde u) = 0$ for some $\tilde u \in \tilde{\mathscr{W}}$. Then, for every $j\in {\mathscr J}$, the sequence of probability measures $(\tilde{\mu}_{n}\circ p_j^{-1})_{n\in\mathbb{N}}$ satisfies the LDP on $\tilde{\mathscr{W}}^{|j|+1}$ with rate ${n \choose 2}$ and with good rate function 
\begin{equation}
\tilde I_j\left((\tilde h_i)_{i=0}^{|j|}\right) := \sum^{|j|}_{i=1} \tilde I_{1, t_{i}-t_{i-1}}(\tilde h_{i-1}, \tilde h_{i}) 
\end{equation}
with $\tilde h_0 = \tilde u$.
\end{theorem}


\subsubsection{Sample-path LDP}

Let $\mathcal{AC}$ denote the set of functions $h \in \mathscr{W} \times [0,T]$ such that $t \mapsto h_t(x,y)$ is absolutely continuous for almost all $(x,y) \in [0,1]^2$. For $h \in \mathcal{AC}$, put 
\begin{equation}
\left.h'_t(x,y) = \frac{\partial h_s(x,y)}{\partial s}\right|_{s=t}.
\end{equation} 
To write down a candidate rate function for the sample-path LDP, fix $\Delta t>0$ such that $T/\Delta t \in \mathbb{N}$. We will show that
\begin{equation}
I_j\left((h_{i\Delta})_{i=0}^{T/\Delta t}\right) = \sum^{T/\Delta t}_{i=1} I_{1, \Delta t} (h_{(i-1)\Delta t}, h_{i \Delta t})  
\to I(h), \qquad \Delta t \downarrow 0, 
\end{equation}
with
\begin{equation}
\label{IDef}
I(h) := \left\{\begin{array}{ll}
\frac12 \int_0^T \ddd t \int_{[0,1]^2} \ddd x\,\ddd y\,\mathcal{L}(h_t(x,y),h'_t(x,y)), &h \in \mathcal{AC},\\
\infty,  &h \notin \mathcal{AC},
\end{array}
\right.
\end{equation}
where
\begin{equation}
\label{LDef}
\mathcal{L}(a,b) = \sup_{v \in \mathbb{R}} 
\big[v b - \lambda (e^v-1)(1-a) - \mu(\eee^{-v} -1) a\big], \qquad a \in [0,1],\,b \in \mathbb{R}.
\end{equation}
As before, $I$ in \eqref{IDef} is not necessarily well defined on $\tilde{\mathscr{W}} \times [0,T]$, and therefore is not a valid candidate function. For this reason we extend the equivalence relation $\sim$ on $\mathscr{W}$ to the equivalence relation $\sim$ on $\mathscr{W} \times [0,T]$ obtained by defining, for every $h_1,h_2 \in \mathscr{W} \times [0,T]$,
\begin{equation}
\label{PathEDef}
h_1 \sim h_2 \quad \text{ if and only if } \quad (h_1)_t \sim (h_2)_t\quad \,\forall\, t \in [0,T],
\end{equation}
and writing $\widetilde{h}$ to denote the equivalence class of $h \in \mathscr{W} \times [0,T]$. 

\begin{theorem}{\bf [Sample-path LDP]}
\label{GSPLDP}
Suppose that $\lim_{n\to\infty} \delta_\square(\tilde f_{n,0}, \tilde u) = 0$ for some $\tilde u \in \tilde{\mathscr{W}}$. Then the sequence $(\tilde\mu_n)_{n\in\mathbb{N}}$ satisfies the LDP on $\tilde{\mathscr{W}} \times [0,T]$ with rate ${n \choose 2}$ and with good rate function 
\begin{equation}
\label{GLDPR}
\tilde I(\tilde h):= \inf_{\substack{h \in \mathscr{W} \times [0,T]: \\ h \sim \tilde h}} I(h), 
\qquad \tilde h \in \tilde{\mathscr{W}} \times [0,T],
\end{equation}
with $\tilde h_0 = \tilde u$.
\end{theorem}


\subsection{Discussion}
\label{S1.6}

Theorems \ref{tMultiLDP} and \ref{GSPLDP} are LDPs for the dynamic inhomogeneous Erd\H{o}s-R\'enyi random graph with independent edge switches. The fact that the rate is $\binom{n}{2}$, the total number of edges, is natural because a positive fraction of the states of the edges must switch in order to produce a change in the graphon. The fact that the rate function in the sample-path LDP is an \emph{action integral} is also natural, because what matters is both the value of the graphon and the gradient of the graphon integrated along the sample path (due to the exponentiality of the underlying switching mechanism). 

Even though the shape of the rate function in \eqref{GLDPR} can be guessed through standard large deviations arguments, the proof of the LDP requires various non-standard steps. Specifically, we are facing the following challenges:
\begin{itemize}
\item[$\circ$]
In Theorem \ref{thm:LDPinhom} the edge probabilities are determined by a single (typically smooth) reference graphon, whereas in Theorem \ref{tMultiLDP} the edge probabilities at time $T$ are determined by a sequence of (inherently rough) empirical graphons. This adds a layer of complexity to the proof, and requires a series of approximations that are technically demanding.
\item[$\circ$]
Theorem \ref{thm:LDPinhom} is an LDP on the quotient space $\tilde{\mathscr{W}}$, whereas Theorem \ref{GSPLDP} is an LDP on the quotient space of {\it paths} $\tilde{\mathscr{W}} \times [0,T]$. This leads to various complications in the proofs, as is also evident from the variational problems that arise when we apply the LDP. While the LDP for the static inhomogeneous ERRG is covered by \cite{C}, \cite{CV} and the sample-path LDP for collections of switching processes is studied in e.g.\  \cite{SW}, the dynamic inhomogeneous ERRG considered in the present paper faces the hurdles encountered in both these works.
\end{itemize}

Several extensions may be thought of. In order to achieve a space-inhomogeneous dynamics, we may replace $\lambda,\mu$ by graphons $\lambda(x,y),\mu(x,y)$, for $(x,y) \in [0,1]^2$, that are bounded away from $0$ and $1$, and let the edge between $i$ and $j$ switch on at rate $\lambda(\frac{i}{n},\frac{j}{n})$ and switch off at rate $\mu(\frac{i}{n},\frac{j}{n})$. In addition, these graphons may vary over time, in order to capture a time-inhomogeneous dynamics. Both extensions are straightforward and are therefore not addressed in the present paper. A challenging extension would be to consider dynamics where the switches of the edges are dependent (cf.\ the setup analysed in \cite{AHR}). 

The two applications to be described in Section~\ref{S2} show that the dynamics is a source of \emph{new phenomena}. The fact that dynamics brings extra richness is no surprise: the area of \emph{interacting particle systems} is a playground with a long history \cite{Li}.  

\subsection{Outline}
Section~\ref{S2} describes two applications of Theorems \ref{tMultiLDP} and \ref{GSPLDP}, formulated in Theorems \ref{LZlemtv} and \ref{SPCond} below. Section~\ref{S3} contains the proof of Theorem \ref{tTwoLDP}, Section~\ref{S4} the proof of Theorems~\ref{tMultiLDP} and \ref{GSPLDP}, and Section~\ref{S5} the proof of Theorems \ref{LZlemtv} and \ref{SPCond}. The applications show that the dynamics introduces interesting bifurcation phenomena.


\section{Applications}
\label{S2}

Section~\ref{S2.1} identifies the most likely path the process takes if it starts from a constant graphon and ends as a graphon with an atypically large density of $d$-regular graphs. Section~\ref{S2.2} identifies the mostly likely path between two given graphons. In these applications, the LDPs presented in Section~\ref{S1} come to life.


\subsection{Application 1}
\label{S2.1}

Suppose that the initial graphon $u$ is constant, i.e., $u \equiv c$ for some $c \in [0,1]$. Condition on the event that at time $T>0$ the density of $d$-regular graphs in $G_n(T)$ is at least $r^d$, where $r$ corresponds to an atypically large edge density compared to $u$, i.e.,
\begin{equation}\label{ULT}
r > c\,\pbb[T] + (1-c) \,\pab[T].
\end{equation}
A natural question is the following. Is the graph $G_n(T)$ conditional on this event close in the cut distance to a typical outcome of an ERRG with edge probability $r$? Phrased differently, are the additional $d$-regular graphs formed by extra edges  (i) sprinkled uniformly, or (ii) arranged in some special structure? 


\subsubsection{Phase transition}

The next theorem, which can be thought of as the dynamic equivalent of \cite[Thm.\ 1.1]{LZ}, answers the above questions when the initial graphon is constant. 

\begin{theorem}{\bf [Phase transition]}
\label{LZlemtv}
Fix a constant initial graphon $u$. Let $H$ be a $d$-regular graph for some $d \in \mathbb{N}\setminus \{1\}$, and $e(H)$ the number of edges in the graph $H$. Suppose that $\delta_\square(\tilde f_{n,0}, \tilde u) \to 0$ and that $r$ satisfies \eqref{ULT}.
\begin{itemize}
\item[\rm (i)] 
If the point $(r^d, I_{1,T}(u,r))$ lies on the convex minorant of $x \mapsto I_{1,T}(u,x^{1/d})$, then
\begin{equation}
\lim_{n \to \infty} \frac{1}{{n \choose 2}} \log \mathbb{P}\left(t(H,f_{n,T}) \geq r^{e(H)}\right) = -I_{1,T}(u,r),
\end{equation}
and for every $\varepsilon >0$ there exists a $C>0$ such that
\begin{equation}
\label{dist1}
\mathbb{P}\left(\delta_\square(f_{n,T},r) < \varepsilon ~\Big|~ t(H,f_{n,T}) \geq r^{e(H)} \right) \geq 1- \eee^{-Cn^2},
\quad n \in \mathbb{N}.
\end{equation}
\item[\rm (ii)] 
If the point $(r^d, I_{1,T}(u,r))$ does not lie on the convex minorant of $x \mapsto I_{1,T}(u,x^{1/d})$, then
\begin{equation}
\lim_{n \to \infty} \frac{1}{{n \choose 2}} \log \mathbb{P}\left(t(H,f_{n,T}) \geq r^{e(H)}\right) > -I_{1,T}(u,r),
\end{equation}
and there exist $\varepsilon, C>0$ such that
\begin{equation}
\label{dist2}
\mathbb{P}\left(\inf_{s \in [0,1]} \delta_\square(f_{n,T},s)
> \varepsilon ~\Big|~ t(H,f_{n,T}) \geq r^{e(H)} \right) \geq 1- \eee^{-Cn^2}, \quad n \in \mathbb{N}.
\end{equation}
\end{itemize}
\end{theorem}

\noindent
In \eqref{dist1} and \eqref{dist2} the $\delta_\square$-distance is towards the constant graphons $r$ and $s$, respectively.  We say that $G_n(T)$ is in the
\begin{itemize} 
\item
\emph{symmetric phase} (S) when the condition of Theorem \ref{LZlemtv}(i) holds, 
\item
\emph{symmetry breaking phase} (SB) when the condition of Theorem \ref{LZlemtv}(ii) hols. 
\end{itemize}

We next explore some consequences of Theorem \ref{LZlemtv}. To avoid redundancy we set $\mu=1$ and put
\begin{equation}
p^* = \lambda/(1+\lambda),
\end{equation}
so that $\lambda=p^*/(1-p^*)$. Note that $p^*$ is the stationary probability that an edge is active. The following two propositions provide a partial phase classification.  

\begin{proposition}{\bf [Short-time SB]}
\label{Pts}
If $u < r$, then for $T$ sufficiently small $G_n(T)$ is SB.
\end{proposition}

\begin{proposition}{\bf [Monotonicity]}
\label{Pht}
\begin{itemize}
\item[{\rm (i)}] If $u = 0$ and $G_n(T)$ is S, then $G_n(T')$ is S for all $T'>T$.
\item[{\rm (ii)}] If $u = 1$ and $G_n(T)$ is SB, then $G_n(T')$ is SB for all $T'>T$.
\end{itemize}
\end{proposition}


\subsubsection{Numerics}

A natural choice of the constant initial graphon is $u = p^*$, i.e., the dynamics starts at a typical outcome of its stationary state. Figure~\ref{Fill} illustrates the consequences of varying $T$ for the case where $d=2$ (i.e., triangles). For large $T$ and $p^*=\frac{1}{7}, \frac{1}{8}$, $G_n(T)$ is S for all $r \in [0,1]$, while for large $T$ and $p^*=\frac{1}{9}, \frac{1}{10}$ there exists $r$ such that $G_n(T)$ is SB. To understand why, observe that, for large $T$, $G_n(T)$ behaves like an Erd\H{o}s-R\'enyi random graph with edge probability $p^*$. According to \cite[Theorem 1.1]{LZ}, if $\mathrm{ERRG}_n(p^*)$ is an  Erd\H{o}s-R\'enyi random graph with edge probability $p^*$, then it is S for all $r \in [0,1]$ if and only if
\begin{equation}
p^* \geq (e^2+1)^{-1}.
\end{equation} 

\begin{figure}[!htb]
\includegraphics[scale=0.2]{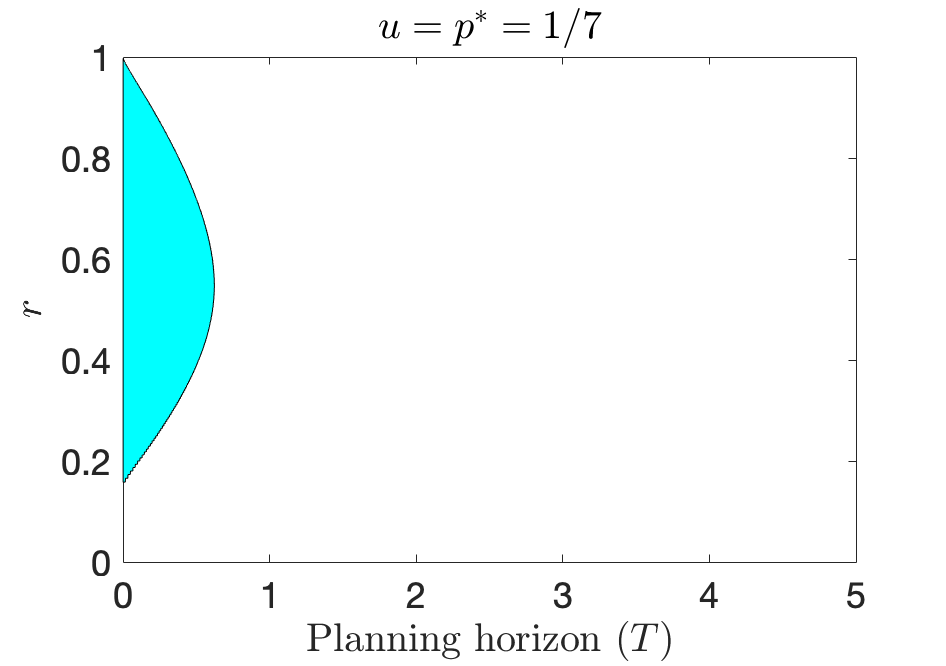}
\includegraphics[scale=0.2]{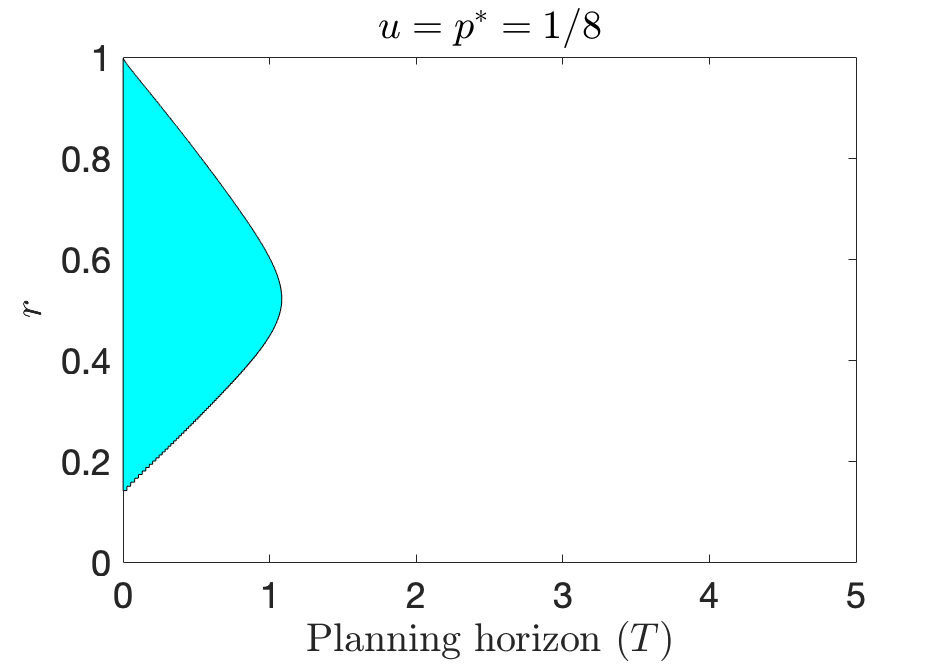}
\includegraphics[scale=0.2]{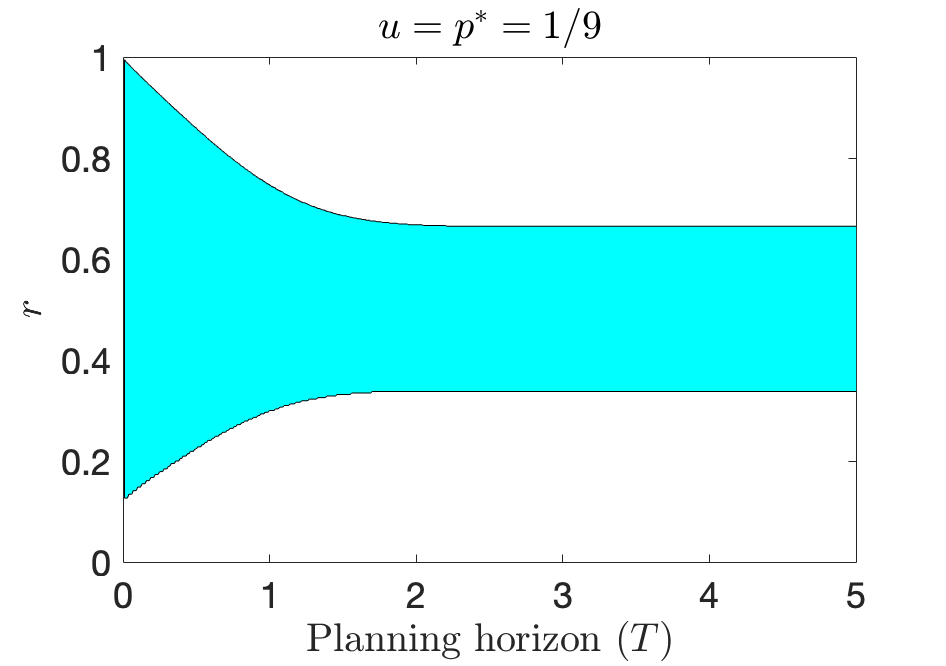}
\includegraphics[scale=0.2]{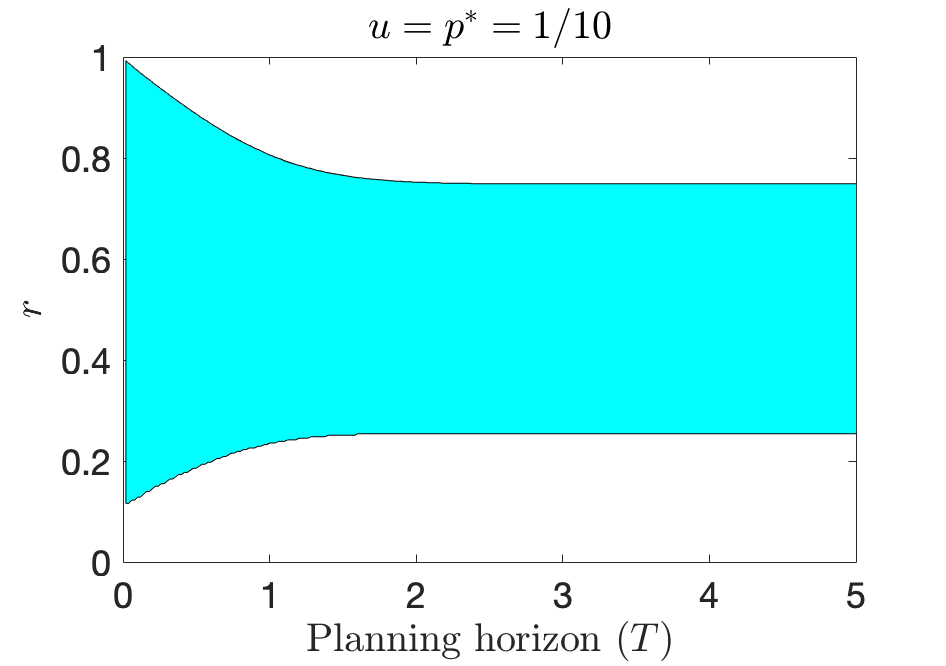}
\caption{Phase diagrams in $(T,r)$ for $d=2$ and $u \equiv p^*$ with $p^* = \frac{1}{7},\frac{1}{8},\frac{1}{9},\frac{1}{10}$. The shaded region corresponds to SB, the unshaded region to S. Observe that SB prevails for small $T$ and $r>u$.}
\label{Fill}
\end{figure}

A visual inspection of Figure \ref{Fill} indicates that, as the planning horizon $T$ increases, $G_n(T)$ can transition from SB to S. An informal explanation is the following. For small $T$ it is more costly to add extra edges than for large $T$. Hence, for small $T$ we expect to see graphs where the extra triangles are formed through the addition of a small number of extra edges arranged in a special structure (corresponding to SB), rather than through the addition of a large number of extra edges sprinkled uniformly (corresponding to S).

\begin{figure}[!htb]
\includegraphics[scale=0.26]{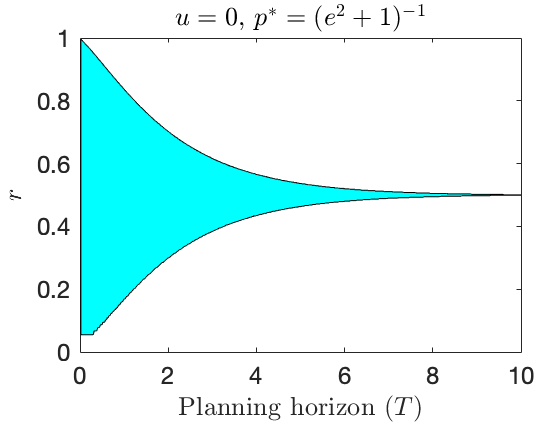}
\includegraphics[scale=0.26]{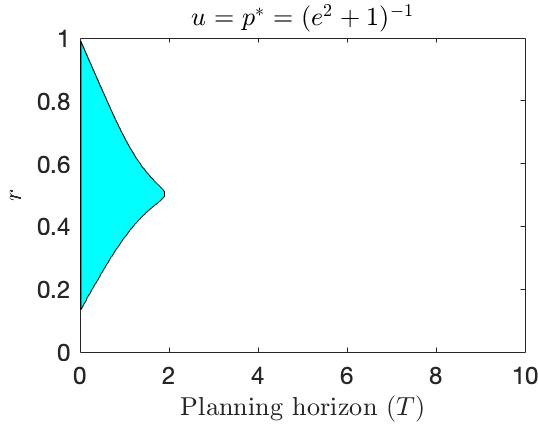}
\includegraphics[scale=0.26]{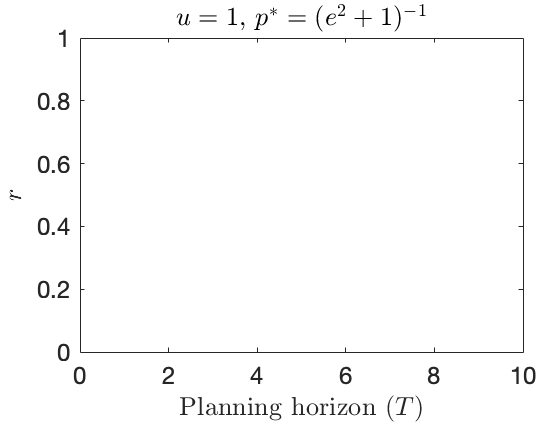}
\includegraphics[scale=0.26]{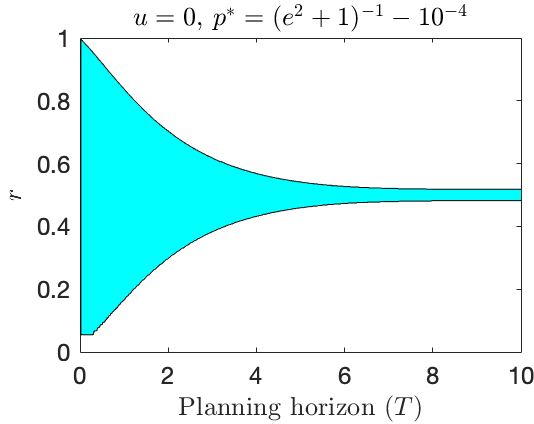}
\includegraphics[scale=0.26]{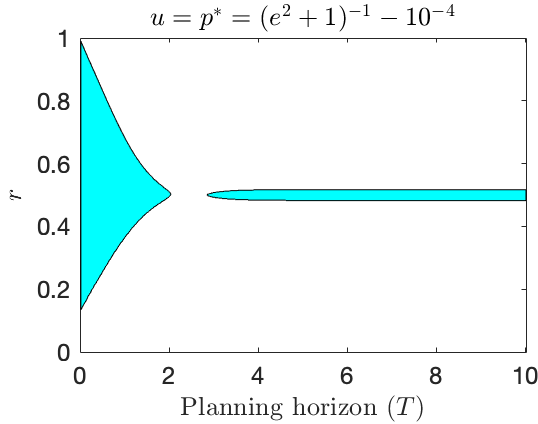}
\includegraphics[scale=0.26]{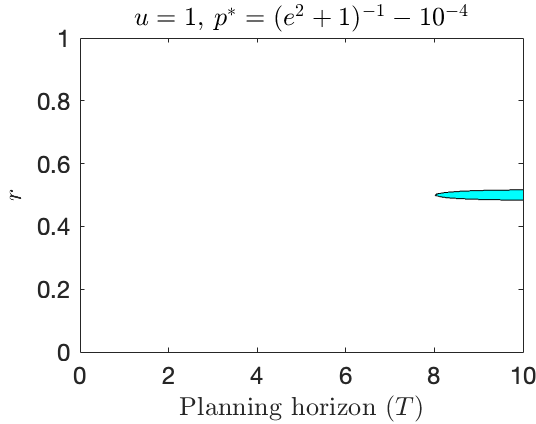}
\includegraphics[scale=0.26]{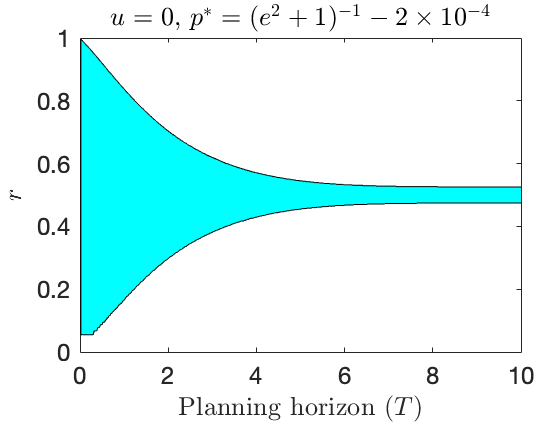}
\includegraphics[scale=0.26]{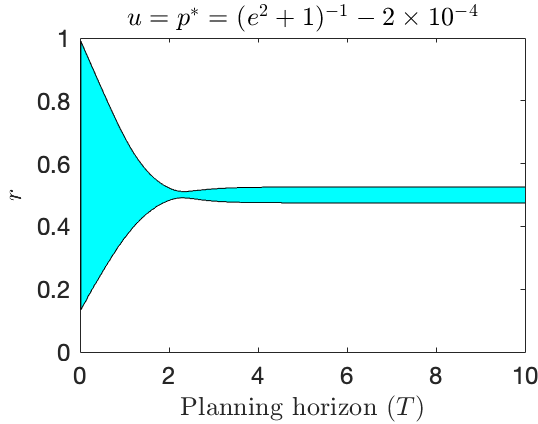}
\includegraphics[scale=0.26]{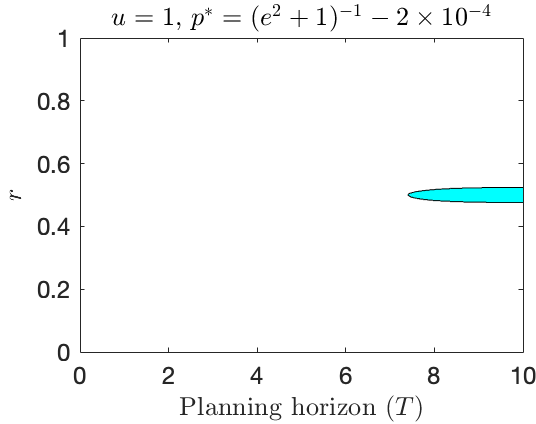}
\caption{Phase diagrams in $(T,r)$ for $u=0,p^*,1$ with $p^*$ close to $(e^2+1)^{-1}$. The shaded region corresponds to SB, the unshaded region to S.}
\label{Filleps}
\end{figure}

Because of the lack of structural results, we numerically consider additional values of $p^*$, namely, near the critical value $(e^2+1)^{-1}$. In Figure \ref{Filleps} we pick $u=0$ (left column), $u=p^*$ (center column), and $u=1$ (right column), and $p^*=(e^2+1)^{-1}$ (top row), $p^* = (e^2+1)^{-1} -10^{-4}$ (middle row), and $p^*=(e^2+1)^{-1}-2 \times 10^{-4}$ (bottom row). Note that, in line with Proposition \ref{Pht}, for $u=0$ or $u=1$ we observe at most one  phase transition in the planning horizon $T$: from SB to S when $u=0$ and from S to SB when $u=1$. However, this is not so when $u=p^*$: when $u=p^*=(e^2+1)^{-1}-10^{-4}$ and $u=p^*=(e^2+1)^{-1}-2 \times 10^{-4}$, there are values of $r$ such that, as $T$ increases, $G_n(T)$ transitions from SB to S and back from S to SB. In other words, two phase transitions occur in the planning horizon $T$, i.e., a \emph{re-entrant phase transition} is observed. 

The re-entrant phase transition in $T$ is quite distinct from the re-entrant phase transition in $r$ (which was first observed in \cite{CV} and is evident from Figures \ref{Fill} and \ref{Filleps}). It is difficult to find a probabilistic explanation for the re-entrant phase transition in $T$. However, once we observe that, when $u=0$, $G_n(T)$ can transition from SB to S and, when $u=1$, $G_n(T)$ can transition from S to SB, then it is plausible that both are possible when we consider the intermediate value $u=p^*$. Moreover, in the light of Proposition \ref{Pts}, when $u=p^*$, $G_n(T)$ can only transition from S to SB \emph{after} it has transitioned from SB to S.

      
\subsection{Application 2} 
\label{S2.2}  

Suppose that the graphon valued process starts near a graphon $\tilde u$ at time $0$, and is conditioned to end near another graphon $\tilde r$ at time $T$. A natural question is the following. Is the most likely path necessarily unique, or is it possible that there are multiple most likely paths? 


\subsubsection{Optimal paths}

The next theorem shows that we can answer this question by studying the set $\tilde H^* \subseteq \tilde{\mathscr{W}}\times [0,T]$ of paths that minimise $\tilde I$ subject to the condition that the path starts at $\tilde u$ and ends at $\tilde r$.
For $\eta>0$, put 
\begin{equation}
\tilde H_\eta := \{ \tilde h \in \tilde{\mathscr{W}} \times [0,T]\colon\, \delta_\square(\tilde h_T, \tilde r) \leq \eta \},
\end{equation}
and for $\tilde h, \tilde h' \in \tilde{\mathscr{W}} \times [0,T]$ define
\begin{equation}
\delta^\infty_\square(\tilde h, \tilde h') := \sup_{t \in [0,T]} \delta_\square(\tilde h_t,\tilde h'_t).
\end{equation} 

\begin{theorem}{\bf [Optimal paths]}
\label{SPCond}
Let $\tilde H \subseteq \tilde{\mathscr{W}} \times [0,T]$ be the set of paths starting at $\tilde u$ and ending at $\tilde r$. Let $\tilde H^* \subseteq \tilde H$ be the set of minimisers of $\tilde I$ in $\tilde H$. Then $\tilde H^*$ is non-empty and compact. In addition, if $\lim_{n\to\infty} \delta_\square(\tilde f_{n,0} , \tilde u) = 0$, then
\begin{equation}
\lim_{\eta \downarrow 0} \limsup_{n \to \infty} \frac{1}{{n \choose 2}} \log 
\mathbb{P}\big(\delta^\infty_\square(\tilde f_n, \tilde H^*\big) \geq \varepsilon \mid  \tilde f_n \in \tilde H_{\eta}) \leq - C,
\end{equation}
where $C>0$ is a constant that depends on $\tilde u$, $\tilde r$, $T$ and $\varepsilon$.
\end{theorem}


\subsubsection{Variational problems}

To better understand the set $\tilde H^*$, we next solve two related variational problems, each with its own probabilistic interpretation.

\begin{lemma}{\bf [Identification of minimiser]}
\label{SPQOpt}
Pick $u \in [0,r]$. Let 
\begin{equation}
f^*_{u \to r}(t) := \argmin_{s \in [0,1]} [I_{1,t}(u,s) + I_{1,T-t}(s,r) ], \qquad t \in [0,T].
\end{equation}
Then $f^*_{u \to r}$ is the unique minimiser of 
\begin{equation}
I(f) = \int_0^T \ddd t\, \mathcal{L}( f(t), f'(t)),
\end{equation}
subject to the condition $f(0)=u$ and $f(T)=r$, where $\mathcal{L}$ is defined in \eqref{LDef}. In addition, $I(f^*_{u \to r})=I_{1,T}(u,r)$.
\end{lemma}

\begin{remark}
{\rm Let $\{X_i(t) \}_{t \geq 0}$, $i\in\mathbb{N}$, be independent processes switching between active and inactive, with $\lambda$ the rate of becoming active and  $\mu$ the rate of becoming inactive, as before. Define
\begin{equation}
L_n(t)= \frac{1}{n} \sum_{i=1}^n X_i(t), \qquad \lim_{n\to\infty} L_n(0) = u.
\end{equation}
Then, informally, we can interpret $f^*_{u \to r}$ as the most likely path that $(L_n(t))_{t \in [0,T]}$ takes from $u$ to $r$ when $n$ is large. Such paths $f^*_{u \to r}$ can be efficiently computed (see Lemma \ref{MM1C} below). 
An illustration is given in Figure \ref{QueueSPLDP}: in the right-hand-side the time horizon $T$ is relatively large, so that the least costly way to reach $r$ is by first falling back towards the equilibrium value $\tfrac23$ and afterwards moving towards $r$ in a relatively short time interval before $T$, whereas in the left-hand-side the the time horizon $T$ is relatively small, so that the least costly way to reach $r$ is by immediately moving towards it.}
\end{remark}

\begin{figure}[!htb]
\includegraphics[scale=0.175]{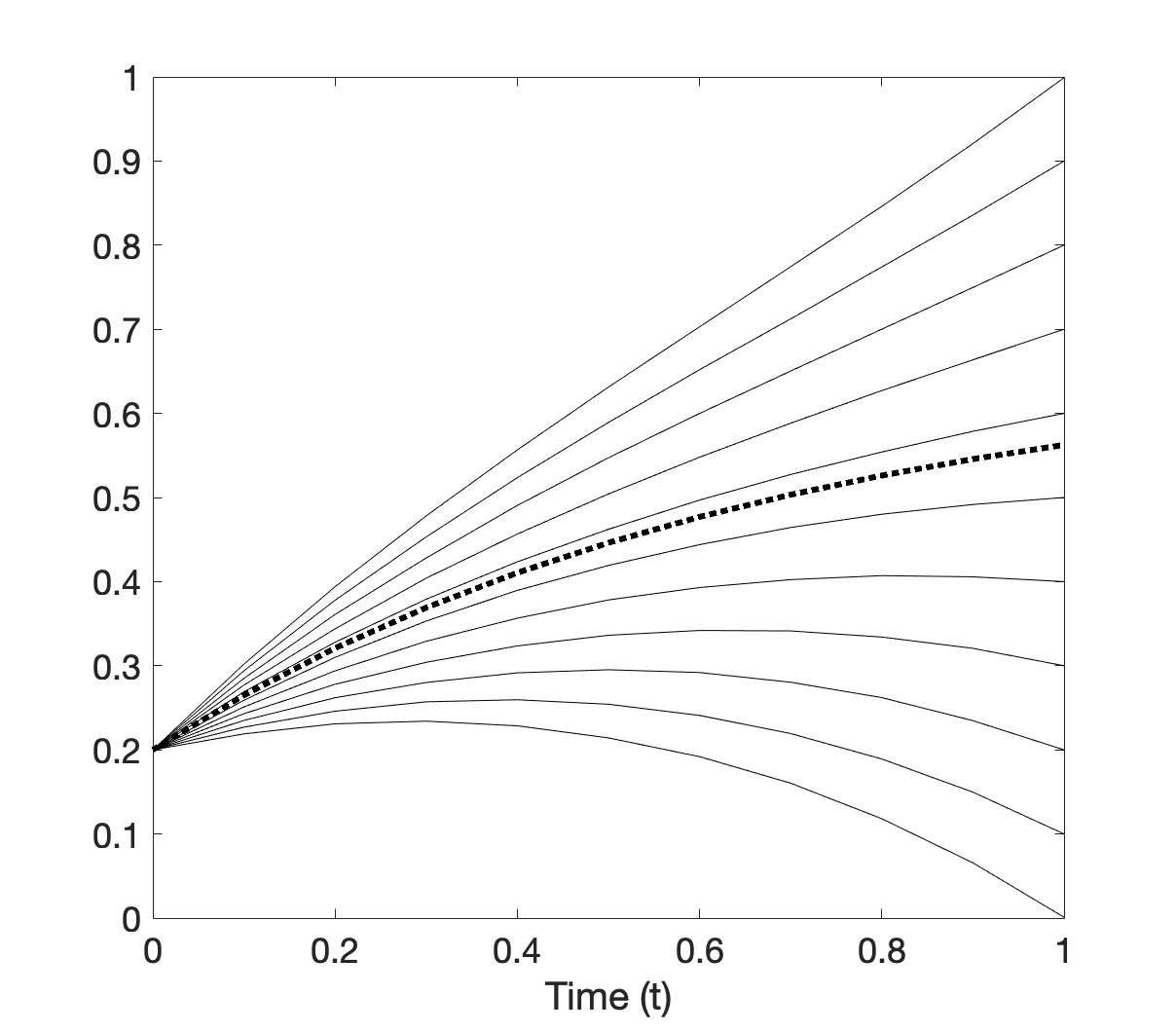}
\includegraphics[scale=0.175]{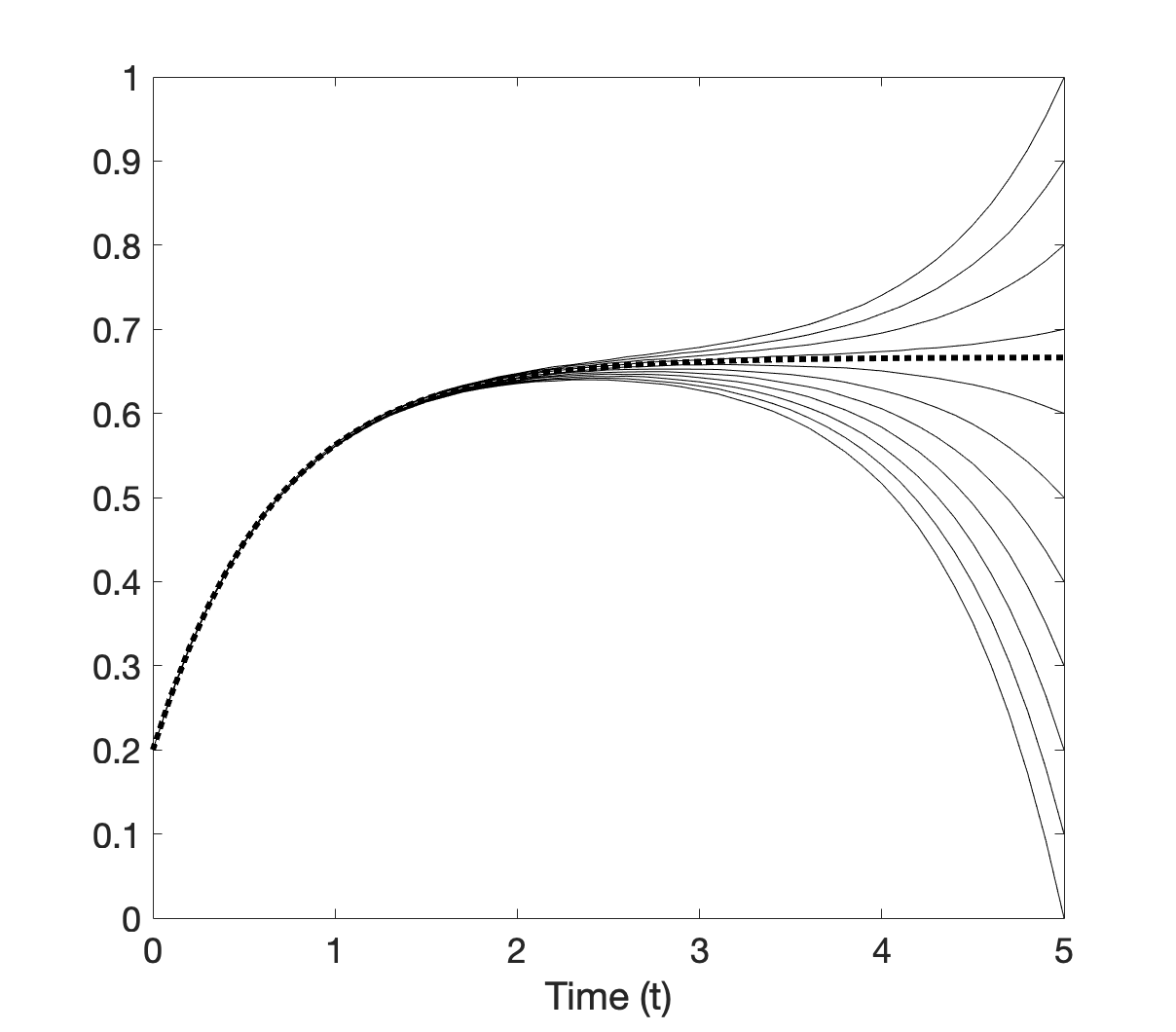}
\caption{\label{QueueSPLDP} Solid curves: the paths $t \mapsto f^*_{u\to r}(t)$ for $\lambda=1$, $\mu=\tfrac12$, $u=\tfrac15$ and $r=0, \tfrac{1}{10}, \tfrac{2}{10}, \dots, 1$ when $T=1$ (left) and $T=5$ (right). Dashed curve: most likely path without the terminal condition.}
\end{figure}

We need to define what we mean when we say that two paths $h,g \in \mathscr{W} \times [0,T]$ are equal. Define the equivalence relation `$\equiv$' by writing $h \equiv g$ if and only if
\begin{equation}
\label{PathQuot}
{\rm Leb} \left\{ (x,y) \in [0,1]^2\colon\, \text{there exists a } t \in [0,T] \text{ such that } h_t(x,y) \neq g_t(x,y) \right\} = 0.
\end{equation} 
Below when we write $\mathscr{W} \times [0,T]$ we assume that this is the quotient space formed by the equivalence relation `$\equiv$'.

\begin{lemma}{\bf [Identification of minimiser]}
\label{WTpath}
Set $u,r \in \mathscr{W}$. Let 
\begin{equation}
h^*_{u\to r}(x,y,t) = f^*_{u(x,y) \to r(x,y)}(t), \qquad t \in [0,T], \, (x,y)\in [0,1]^2.
\end{equation}
Then $h^*_{u \to T}$ is the unique minimiser of 
\begin{equation}
I(h) = \int_{0}^T \ddd t \int_{[0,1]^2} \ddd x \,\ddd y\, \mathcal{L}\big(h_t (x,y) , h'_t(x,y)\big),
\end{equation}
subject to the condition that $h_0 =u$ and $h_T=r$, where $\mathcal{L}$ is defined in \eqref{LDef}. In addition, $I(h^*_{u \to r}) = I_{1,T}(u,r)$.
\end{lemma}

We next turn our attention to the original variational problem on $\tilde{\mathscr{W}}\times [0,T]$. If $\tilde h \in \tilde H^*$, then, armed with Lemma \ref{WTpath} and the specific form of $\tilde I$, we may expect that there exists a representative $h$ of $\tilde h$ such that 
\begin{equation}
I(h) = I_{1,T}(u,r^\sigma), \qquad h_T=r^\sigma,
\end{equation}
for some $\sigma \in \mathscr{M}$. By Lemma \ref{WTpath}, the only such paths are of the form $h^*_{u \to r}$. Theorem \ref{Ap2Thm} below, which applies when $u$ and $r$ are \emph{block graphons}, shows that we may restrict our attention to the equivalence classes of these paths, i.e., the set $\{ \tilde h^*_{u \to r^\sigma} \}_{\sigma \in \mathscr{M}}$, and implies that we can replace the variational problem on $\tilde{\mathscr{W}} \times [0,T]$ by a significantly simpler one, in terms of permutations of the target graphon $r$. 

For $I \in \mathbb{N}$, let $\mathscr{W}^{(I)}$ denote the space of block graphons with $I^2$ blocks, so that for any $g \in \mathscr{W}^{(I)}$ there exist \emph{block endpoints} $0 = x_0 < x_1 < \dots < x_I=1$ such that 
\begin{equation}
g(x,y) = g_{ij} \qquad \forall\,x,y \in [x_{i-1}, x_i) \times [x_{j-1}, x_j). 
\end{equation}
For $u \in \mathscr{W}^{(I)}$ and $r \in \mathscr{W}^{(J)}$ with block endpoints $0=a_0 < a_1 < \dots < a_I=1$ and $0=b_0 < b_1< \dots< b_J=1$, respectively, let $\alpha\colon\,\mathscr{M} \mapsto [0,1]^{I \times J}$ be defined by
\begin{equation}
\label{Aldeff}
\alpha(\sigma)_{ij} := {\rm Leb} \left\{x \in [0,1]\colon\, x \in [a_{i-1}, a_i), \, 
\sigma(x) \in [b_{i-1}, b_i) \right\}, \qquad \sigma \in \mathscr{M}.
\end{equation}
Note that $\alpha$ can map to any value in the compact set 
\begin{equation}
\label{Vdeff}
V= \left\{ \boldsymbol v \in [0,1]^{I \times J}\colon\, \sum_{j \in J} v_{ij} = a_i - a_{i-1}\,\,\forall\,i \in I, \,\,
\sum_{i \in I} v_{ij} = b_{j} - b_{j-1}\,\,\forall\,j \in J \right\}.
\end{equation}
It is important to point out that, for any $\sigma_1,\sigma_2 \in \mathscr{M}$ with $\alpha(\sigma_1)=\alpha(\sigma_2)$,
\begin{equation}
\tilde h^*_{u \to r^{\sigma_1}} = \tilde h^*_{u \to r^{\sigma_2}}, \qquad I_{1,T}(u, r^{\sigma_1})= I_{1,T}(u,r^{\sigma_2}).
\end{equation}
For $v \in [0,1]^{I \times J}$, let $\sigma_{\boldsymbol v}$ be any element of $\mathscr{M}$ such that $\alpha(\sigma) = \boldsymbol v$.

\begin{theorem}{\bf [Optimal paths]}
\label{Ap2Thm}
Suppose that $u \in \mathscr{W}^{(I)}$ and $r \in \mathscr{W}^{(J)}$ for some $I,J \in \mathbb{N}$. Then
\begin{equation}
\label{HtSFt}
\tilde H^* \subseteq  \tilde F^*:=\{ \tilde h^*_{u \to r^\sigma} \}_{\sigma \in \mathscr{M}} = \{ \tilde h^*_{u \to r^{\sigma_{\boldsymbol v}}} \}_{\boldsymbol v \in V}.
\end{equation}
Moreover, if $V^*$ is the set of $\boldsymbol v \in V$ that minimise $I_{1,T}(u, r^{\sigma_{\boldsymbol v}})$, then $V^*$ is non-empty and 
\begin{equation}
\label{HtSHst}
\tilde H^* = \{ \tilde h^*_{u \to r^{\sigma_{\boldsymbol v}}} \}_{\boldsymbol v \in V^*}.
\end{equation} 
\end{theorem}

\noindent
The requirement that $u$ and $r$ be block graphons is harmless, as block graphons can be used to approximate any graphon arbitrarily closely in ${L}^2$  \cite[Proposition 2.6]{C}. The following corollary is immediate because the constant graphon is invariant under permutation.

\begin{corollary}{\bf [Uniqueness]}
\label{OPG}
If either $u$ or $r$ is a constant graphon, then $\tilde H^*=\tilde F^*=\{\tilde h^*_{u \to r} \}$, i.e., both sets contain a single element.
\end{corollary}


\subsubsection{Multiplicity}

We next explore whether $\tilde H^*$ can contain multiple paths. Theorem~\ref{Ap2Thm} provides a concrete criterion. In particular, $\tilde H^*$ contains multiple paths if and only if there exist $\boldsymbol v_1, \boldsymbol v_2 \in V$ such that
\begin{equation}
\label{MConds} 
\tilde I_{1,T}( \tilde u, \tilde r) = I_{1,T}(u, r^{\sigma_{\boldsymbol v_1}})
= I_{1,T}(u, r^{\sigma_{\boldsymbol v_2}}), \qquad  
\tilde h^*_{u \to  r^{\sigma_{\boldsymbol v_1}}}  \neq \tilde h^*_{u \to  r^{\sigma_{\boldsymbol v_2}}}.
\end{equation}
We thus need to determine whether  these two conditions can be satisfied simultaneously. 

We begin by focussing on the latter condition: Can $\tilde F^*$ (defined in Theorem \ref{Ap2Thm}) contain multiple paths? The answer is yes: even though for any $\tilde h, \tilde g \in \tilde F^*$ we have $\tilde u = \tilde h_{0} = \tilde g_0$ and $\tilde r = \tilde h_{T} = \tilde g_T$, this does not necessarily imply that $\tilde h_t = \tilde g_t$ for $t \in (0, T)$. To see why, we consider the following simple example. Let  $u, r \in \mathscr{W}$ be such that 
\begin{equation}
\label{huDef}
u(x,y)=r(x,y) =
\begin{cases}
1 \qquad &\text{if } x \leq \tfrac12 , \, y \leq \tfrac12 , \\
0 \qquad &\text{if } x \geq \tfrac12 , \, y \geq \tfrac12 , \\
\tfrac12 \quad  & \text{otherwise}, 
\end{cases}
\end{equation}
and $\sigma \in \mathscr{M}$ be such that 
\begin{equation}
\label{siDef}
\sigma(x) = 
\begin{cases}
x+\tfrac12  \qquad &\text{if } x \leq \tfrac12 , \\
x-\tfrac12  \qquad &\text{if } x > \tfrac12 . \\
\end{cases}
\end{equation}
Recalling the definition of $f^*_{a \to b}$ from Lemma \ref{SPQOpt}, we have
\begin{equation}
h^*_{ u \to r}(x,y,t) = 
\begin{cases}
f^*_{1 \to 1}(t)  \quad &\text{if } x \leq \tfrac12 , \, y \leq \tfrac12 , \\
f^*_{0 \to 0}(t) \quad &\text{if } x \geq \tfrac12 , \, y \geq \tfrac12 , \\
f^*_{\frac{1}{2} \to \frac{1}{2}}(t) \quad  & \text{otherwise},
\end{cases} 
\end{equation}
and
\begin{equation}
h^*_{ u \to r^\sigma}(x,y,t) = 
\begin{cases}
f^*_{1 \to 0}(t)  \quad &\text{if } x \leq \tfrac12 , \, y \leq \tfrac12 , \\
f^*_{0 \to 1}(t) \quad &\text{if } x \geq \tfrac12 , \, y \geq \tfrac12 , \\
f^*_{\frac{1}{2} \to \frac{1}{2}}(t) \quad  & \text{otherwise}.
\end{cases}
\end{equation}
The three $(x,y)$-coordinate paths of $h^*_{u \to r}$ (solid) and $h^*_{u \to r^\sigma}$ (dashed) are illustrated in Figure~\ref{UniPath1} for $\lambda=\mu=1$ and $T=3$. It is easiest to see that $\tilde h^*_{u \to r} \neq \tilde h^*_{u \to r^\sigma}$ by looking at their values when $t=\tfrac32$: $h^*_{u \to r^\sigma}(\cdot,\cdot,\tfrac32)$ is a constant graphon while $h^*_{u \to r}(\cdot,\cdot,\tfrac32)$ is not. 

\begin{figure}[!htb]
\includegraphics[scale=0.40]{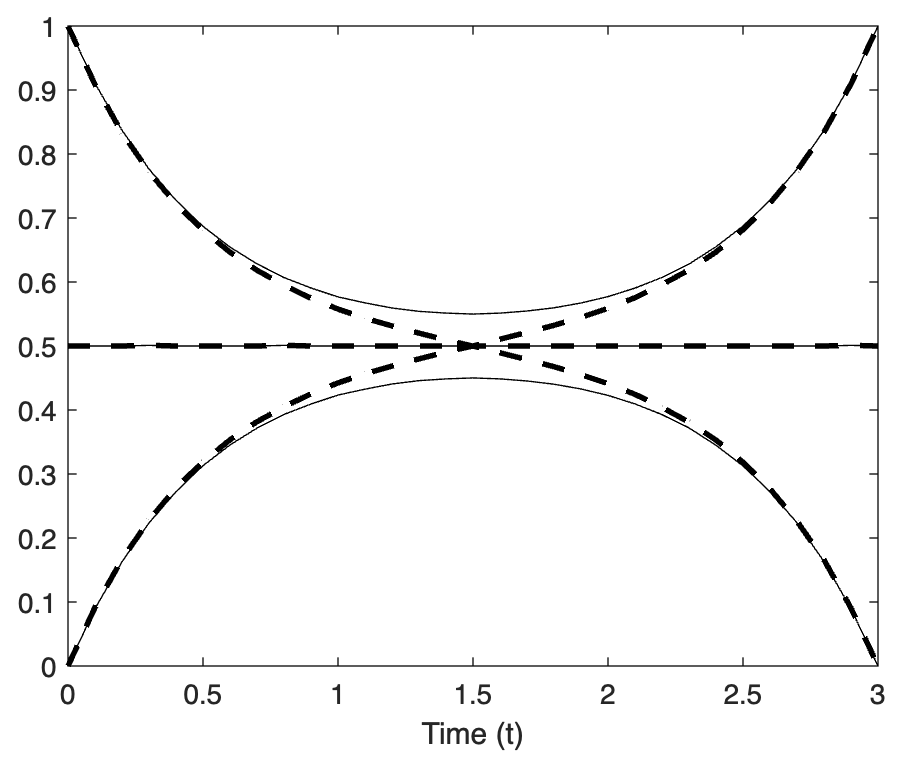}
\caption{\label{UniPath1} The $(x,y)$-coordinate paths of $h^*_{u \to r}$ (solid) and $h^*_{u \to r^\sigma}$ (dashed) for $\lambda=\mu=1$, $T=3$, with $u$, $r$ and $\sigma$ defined in \eqref{huDef} and \eqref{siDef}.}
\end{figure}

While the above example demonstrates that $\tilde F^*$ can contain multiple elements, it does not imply that $\tilde H^*$ can contain multiple elements. Indeed, the next proposition implies that in the above example $\tilde H^*$ contains a single element: $\tilde H^*= \{ \tilde h^*_{u \to r} \}$.

\begin{proposition}{\bf [Uniqueness]}
\label{TwoBlocks}
If there exist $\sigma_1,\sigma_2 \in \mathscr{M}$ such that, for any $a < b$ and $x \in [0,1]$,
\begin{equation}
\label{HUC}
u^{\sigma_2}(a,x) \leq u^{\sigma_2}(b,x), \qquad r^{\sigma_1}(a,x) \leq r^{\sigma_1}(b,x),
\end{equation}
then $\tilde H^* = \{ \tilde h^*_{u^{\sigma_2} \to r^{\sigma_1}} \}$.
\end{proposition}

\begin{figure}
\centering
\begin{tikzpicture}[scale=0.35]
\usetikzlibrary{arrows}
\tikzset{vertex/.style = {shape=circle,draw,minimum size=1.3em}}
\tikzset{edge/.style = {->,> = latex'}}
\draw (0,0) rectangle (10,10); 
\draw (2,0) -- (2,10);
\draw (6,0) -- (6,10);
\draw (0,4) -- (10,4);
\draw (0,8) -- (10,8);
\node at (1,9) {$u_{11}$};
\node at (4,9) {$b$};
\node at (8,9) {$a$};
\node at (1,6) {$b$};
\node at (4,6) {$a$};
\node at (8,6) {$u_{23}$};
\node at (1,2) {$a$};
\node at (4,2) {$u_{23}$};
\node at (8,2) {$b$};

\draw (15,0) rectangle (25,10);
\draw (17,0) -- (17,10);
\draw (21,0) -- (21,10);
\draw (15,4) -- (25,4);
\draw (15,8) -- (25,8);
\node at (16,9) {$r_{11}$};
\node at (19,9) {$c$};
\node at (23,9) {$d$};
\node at (16,6) {$c$};
\node at (19,6) {$c$};
\node at (23,6) {$r_{23}$};
\node at (16,2) {$d$};
\node at (19,2) {$r_{23}$};
\node at (23,2) {$d$};

\draw (30,0) rectangle (40,10);
\draw (32,0) -- (32,10);
\draw (36,0) -- (36,10);
\draw (30,4) -- (40,4);
\draw (30,8) -- (40,8);
\node at (31,9) {$r_{11}$};
\node at (34,9) {$d$};
\node at (38,9) {$c$};
\node at (31,6) {$d$};
\node at (34,6) {$d$};
\node at (38,6) {$r_{23}$};
\node at (31,2) {$c$};
\node at (34,2) {$r_{23}$};
\node at (38,2) {$c$};

\node at (-1,2) {\footnotesize $\frac25$} ;
\node at (-1,6) {\footnotesize $\frac25$} ;
\node at (-1,9) { \footnotesize $\frac15$} ;

\node at (1,11) {\footnotesize $\frac15$} ;
\node at (4,11) {\footnotesize $\frac25$} ;
\node at (8,11) {\footnotesize $\frac25$} ;

\node at (14,2) {\footnotesize $\frac25$} ;
\node at (14,6) {\footnotesize $\frac25$} ;
\node at (14,9) {\footnotesize $\frac15$} ;

\node at (16,11) {\footnotesize $\frac15$} ;
\node at (19,11) {\footnotesize $\frac25$} ;
\node at (23,11) {\footnotesize $\frac25$} ;

\node at (29,2) {\footnotesize $\frac25$} ;
\node at (29,6) {\footnotesize $\frac25$} ;
\node at (29,9) {\footnotesize $\frac15$} ;

\node at (31,11) {\footnotesize $\frac15$} ;
\node at (34,11) {\footnotesize $\frac25$} ;
\node at (38,11) {\footnotesize $\frac25$} ;

\node at (5,13) {\large $u$} ;
\node at (20,13) {\large $r$} ;
\node at (35,13) {\large $r^\sigma$} ;

\end{tikzpicture} 
\caption{\label{PathEx} An example of $u$, $r$ and $\sigma$ for which $\tilde h^*_{u \to r} \neq \tilde h^*_{u \to r^\sigma}$ and $\tilde I(\tilde h^*_{u \to r}) = \tilde I (\tilde h^*_{u \to r^\sigma})$.}
\end{figure}

Is there always a single optimising path, i.e., does $\tilde H^*$ always contain a single element? To answer this question we consider the graphons $u$, $r$ and $r^\sigma$ illustrated in Figure \ref{PathEx}. Let $A_1 = [0, \frac{1}{5})$, $A_2 = [\frac{1}{5},\frac{3}{5})$ and $A_3 =[\frac{3}{5},1]$. For constants $a,b,c,d, u_{11},u_{23},r_{11},r_{23} \in [0,1]$, let 
\begin{equation}
\label{PathExu}
u(x,y) = \begin{cases}
u_{11} \quad &\text{if } (x,y) \in A^2_1,  \\
a \quad &\text{if } (x,y) \in A_2^2 \cup A_1 \times A_3 \cup A_3  \times A_1, \\
b \quad &\text{if } (x,y) \in A_3^2 \cup A_1 \times A_2 \cup A_2  \times A_1, \\
u_{23} &\text{otherwise, }
\end{cases}
\end{equation}
and 
\begin{equation}
\label{PathExr}
r(x,y) = \begin{cases}
r_{11} \quad &\text{if } (x,y) \in A_1^2, \\
c \quad &\text{if } (x,y) \in A_2^2 \cup A_1 \times A_3 \cup A_3  \times  A_1, \\
d \quad &\text{if } (x,y) \in A_3^2 \cup A_1 \times A_2 \cup A_2  \times  A_1, \\
r_{23} &\text{otherwise, }
\end{cases}
\end{equation}
and let $\sigma \in \mathscr{M}$ be such that 
\begin{equation}
\label{PathExs}
\sigma(x) = 
\begin{cases}
x \quad  &\text{if } x \in A_1, \\
x+\frac{2}{5} \quad &\text{if } x \in A_2, \\
x-\frac{2}{5} &\text{if } x \in A_3.
\end{cases}
\end{equation}
Note that $I (h^*_{u \to r} ) =I_{1,T} (u, r)  = I_{1,T}(u, r^\sigma)=  I(h^*_{u \to r^\sigma})$, regardless of the values of $a,b,c,d, u_{11},u_{23},r_{11},r_{23}$.  However, to ensure that both conditions in \eqref{MConds} are satisfied, we need to select these parameters carefully. The next proposition tells us how we can do this. 

\begin{proposition}{\bf [Non-uniqueness]}
\label{ThreeBlocks}
Suppose that $u$, $r$, $\sigma$ are given by \eqref{PathExu}, \eqref{PathExr}, \eqref{PathExs}, and set 
\begin{equation}
a=c=0, \; \; \; b=d=\varepsilon, \; \; \;  u_{11}=u_{23}=r_{11}=r_{23}=1.
\end{equation}
Then, for $\varepsilon,T>0$ sufficiently small, $\tilde H^* = \{ \tilde h^*_{u \to r}, \tilde h^*_{u \to r^\sigma} \}$ and $\tilde h^*_{u \to r} \neq \tilde h^*_{u \to r^\sigma}$.
\end{proposition}

\noindent
Through this example we are led to conclude that if the process begins near $\tilde u$ at time $0$ and is conditioned to end at $\tilde r$ at time $T$, then it may take one of two equally likely paths. Note that, in view of the counting and inverse counting lemmas \cite[Lemmas 10.23 and 10.32]{L}, by specifying graphons $\tilde u$ and $\tilde r$ at times $t=0$ and $t=T$ we are in effect specifying the subgraph density $t(H,\tilde f_{n,t})$ at times $t=0$ and $t=T$ for every simple graph $H$. By these same lemmas, Proposition \ref{ThreeBlocks} shows that, for some subgraphs $H$, the subgraph density $t \mapsto t(H, \tilde f_{n,t})$ may take one of two equally likely paths from $t(H,\tilde u)$ to $t(H, \tilde r)$.


\section{Proof of the two-point LDP}
\label{S3}

In this section we prove Theorem~\ref{tTwoLDP}. We settle the lower semi-continuity of the rate function $I_{1,t}(\tilde u,\cdot)$ in Section \ref{subsec:LSC}, the upper bound of the LDP in Section \ref{subsec:UB}, and the lower bound of the LDP in Section \ref{subsec:LB}.

Abbreviate $\tilde{\mathbb{P}}^{\tilde u}_{n,t}(\cdot) := \mathbb{P}(\tilde f_{n,t} \in \cdot \,|\, \tilde f_{n,0} = \tilde u)$ and $\mathbb{P}^{u}_{n,t}(\cdot): = \mathbb{P} ( f_{n,t} \in \cdot \,|\, f_{n,0} = u)$. As before, let $\mathscr{W}^{(I)} \subseteq \mathscr{W}$ denote the space of block graphons with $I^2\in{\mathbb N}$ blocks. Also define the $\varepsilon$-balls
\begin{equation}
\begin{aligned}
\tilde{\mathbb{B}}_\square(\tilde h, \varepsilon) 
&:= \{\tilde g \in \tilde{\mathscr{W}}\colon\, \delta_\square(\tilde h, \tilde g)\leq \varepsilon\}, \\
{\mathbb{B}}_\square(\tilde h, \varepsilon) 
&:= \{g \in {\mathscr{W}}\colon\, \delta_\square(\tilde h, \tilde g)\leq \varepsilon\}, \\
{\mathbb{B}}_\square(h, \varepsilon) 
&:= \{g \in {\mathscr{W}}\colon\, d_\square(h, g) \leq \varepsilon\}.
\end{aligned}
\end{equation}
Write $\tilde{\mathscr{W}}_n$ to denote the set of empirical graphons with $n$ vertices. 

We first state two properties of $I_{1,t}$, uniformity and convexity, that are needed along the way and are straightforward to verify.

\begin{lemma}
\label{UnifCont} 
For every $t >0$ and $u,h \in [0,1]$, 
\begin{equation}
I_{1,t}(u,h) \leq \max \left\{ -\log \pab[t], - \log (1-\pbb[t]) \right\} < \infty.
\end{equation}
Moreover, for $\eta,\varepsilon>0$, let
\begin{equation}
\label{LSC0}
\Delta_{I}(\eta,\varepsilon) := \max_{u,h \in [0,1], u' \in [u-\eta,u+\eta], h' \in [h-\varepsilon, h+\varepsilon]} 
| I_{1,t}(u,h)-I_{1,t}(u',h')|.
\end{equation}
Then $\lim_{\eta,\varepsilon\downarrow 0} \Delta_{I}(\eta,\varepsilon) = 0$. 
\end{lemma}

\begin{lemma}
\label{Convexity}
For every $t>0$ and $u,h \in [0,1]$, $u \mapsto I_{1,t}(u,h)$ is convex and $h \mapsto I_{1,t}(u,h)$ is strictly convex. Moreover, for every $u,h \in \mathscr{W}$ and $A, B \subseteq [0,1]$, 
\begin{equation}
\label{LSC1} 
\frac{1}{{\rm Leb}(A \times B)} \int_{A \times B} \ddd x\, \ddd y\,I_{1,t}(u(x,y),h(x,y))
\geq  I_{1,t}\big(\widebar u(A \times B), \widebar h(A \times B)\big),
\end{equation}
where 
\begin{equation}
\widebar u(A \times B) := \frac{1}{{\rm Leb}(A \times B)}\int_{A \times B} \ddd x\, \ddd y\,u(x,y),
\end{equation}
and $\widebar h(A \times B)$ is defined similarly.
\end{lemma}


\subsection{Lower semi-continuity}
\label{subsec:LSC}

We first establish that $I_{1,t}(\tilde u, \cdot)$ is a good rate function, i.e., $I_{1,t}(\tilde u, \cdot) \not\equiv \infty$ and $\tilde x \mapsto \tilde I_{1,t}(\tilde u, \tilde x)$ has compact level sets. Because $\tilde{\mathscr{W}}$ is compact, it suffices to show that $\tilde x \mapsto \tilde I_{1,t}(\tilde u, \tilde x)$ is lower semi-continuous. We will in fact show that $(\tilde u, \tilde x) \mapsto \tilde I_{1,t}(\tilde u, \tilde x)$ is lower semi-continuous, because this stronger property is needed below.

\begin{lemma}
\label{LwrSC}
Suppose that $\lim_{n\to\infty} \delta_\square(\tilde u_n,\tilde u) = 0$ and $\lim_{n\to\infty} \delta_\square(\tilde h_n,\tilde h) = 0$. Then 
\begin{equation}
\liminf_{n \to \infty} \tilde{I}_{1,t}(\tilde u_n, \tilde h_n) \geq \tilde I_{1,t}(\tilde u, \tilde h).
\end{equation}
\end{lemma}

\begin{proof}
Let $\mathcal{A}_k$ be the set of all $k$-set partitions of $[0,1]$. In particular, $\mathcal{A}_k$ includes any $\{ A_i\}_{i=1}^k$ such that $A_i \subseteq [0,1]$, $A_{i} \cap A_j = \emptyset$ for $i,j\in[k]$ with $i\not=j$,  and $\cup_{i=1}^k A_i = [0,1]$. Let $(\sigma_n)_{n\in\mathbb{N}}$ be any sequence of elements in $\mathscr{M}$. Recalling \eqref{efg}--\eqref{MCP2} and applying Lemmas~\ref{UnifCont}--\ref{Convexity}, we obtain, for any $k \in \mathbb{N}$,
\begin{equation}
\begin{aligned}
&\liminf_{n \to \infty} I_{1,t}(u_n, h_n^{\sigma_n})\\
&\geq \liminf_{n \to \infty} \sup_{\{A_i\}_{i=1}^k \in \mathcal{A}_k} 
\sum_{i,j=1}^k {\rm Leb}(A_i \times A_j) I_{1,t}\big(\widebar u_n(A_i \times A_j) , \widebar{ h_n^{\sigma_n}}(A_i \times A_j)\big) \\
&\geq \liminf_{n \to \infty} \sup_{\{A_i\}_{i=1}^k \in \mathcal{A}_k} \sum_{i,j=1}^k {\rm Leb}(A_i \times A_j)\\
&\qquad \times \big[ I_{1,t}\big(\widebar u(A_i \times A_j) , \widebar{ h^{\sigma_n}}(A_i \times A_j)\big) 
- \Delta_I\big(d_\square(u_n,u), d_\square(h_n,h)\big) \big]\\
&\geq \inf_{\sigma \in \mathscr{M}}  \sup_{\{A_i\}_{i=1}^k \in \mathcal{A}_k} 
\sum_{i,j=1}^k {\rm Leb}(A_i \times A_j) I_{1,t}\big(\widebar u(A_i \times A_j) , \widebar{ h^{\sigma}}(A_i \times A_j)\big).
\end{aligned}
\end{equation}
The proof is complete once we show that 
\begin{equation}
\label{LSC2}
\lim_{k \to \infty} \inf_{\sigma \in \mathscr{M}}  \sup_{\{A_i\}_{i=1}^k \in \mathcal{A}_k}
\sum_{i,j=1}^k {\rm Leb}(A_i \times A_j) I_{1,t}\big(\widebar u(A_i \times A_j),\widebar{ h^{\sigma}}(A_i \times A_j)\big) 
= \inf_{\sigma \in \mathscr{M}} I_{1,t} (u, h^\sigma).
\end{equation}

\medskip\noindent
{\bf 1.} 
First we establish \eqref{LSC2} when $u$ and $h$ are \emph{block graphons}, i.e., $u \in \mathscr{W}^{(I)}$ and $h \in \mathscr{W}^{(J)}$. Let $0=a_0<a_1< \dots < a_I=1$ and $0=b_0<b_1< \dots < b_J=1$ denote their block endpoints, $A_i=[a_{i-1}, a_i)$ and $B_i=[b_{i-1}, b_i)$ their block intervals, and $\{u_{ij}\}_{1 \leq i,j \leq I}$ and $\{h_{\ell m}\}_{1 \leq \ell,m \leq J}$ their block values. (For example, if $x \in A_i$ and $y \in A_j$, then $u(x,y)=u_{ij}$.) For $ i \in[I]$ and $ j \in[J]$, let 
\begin{equation}
\label{CJijdef}
C_{J(i-1)+j} := \left\{x \in [0,1]\colon\, x \in A_i, \sigma^{-1}(x) \in B_j \right\}, 
\end{equation}
Suppose that $k=IJ$. Then \eqref{CJijdef} defines a sequence of sets $C_1,\ldots,C_k$. We have 
\begin{equation}
\label{Csums}
\begin{aligned}
\sum_{i,j=1}^k{\rm Leb}(C_i \times C_j) 
&\,I_{1,t}\big(\widebar u(C_i \times C_j),\widebar{ h^{\sigma}}(C_i \times C_j)\big)\\
&= \sum_{ i,j \in[ I], \:\ell, m \in[ J]} {\rm Leb}( C_{J(i-1)+\ell} \times C_{J(j-1) +m})\,  I_{1,t}(u_{ij},h_{\ell m})\\
&= \int_{[0,1]^2} \ddd x\,\ddd y \, I_{1,t}\big(u(x,y), h^\sigma(x,y)\big) = I_{1,t}(u,h^\sigma),
\end{aligned}
\end{equation}
where the first and second equality are obtained by observing that $u$ and $h^\sigma$ are constant on $C_i \times C_j$. Since this holds for any $\sigma \in \mathscr{M}$ and $k \geq IJ$ (for $k > I J$ simply take $C_i=\emptyset$ for all $i > I J$), we have established \eqref{LSC2}.

To explain the above in a bit more detail, suppose that each point $x \in [0,1]$ is a vertex. Because $u$ and $h$ are block graphons with $I$ and $J$ blocks, respectively, we can think of vertices as being of type $1, \dots, I$ at time 0 and of type $1, \dots, J$ at time $t$. We would like $C_{J(i-1)+j}$ to contain all vertices of type $i$ at time 0 and of type $j$ at time $t$. It is clear that this means that $x \in A_i$. However, because we have applied an arbitrary permutation $\sigma$ to $h$ to get $h^\sigma$, the types of all the vertices at time $t$ have been mixed up. Nonetheless, we know that vertex $x$ is of type $j$ at time $t$ if it maps to block $j$ in $h$ when the permutation $j$ is undone (i.e., $\sigma^{-1}(x) \in B_j$). Now, $C_{J(i-1)+j}$ contains all vertices that are of type $i$ at time 0 and of type $j$ at time $t$. Hence, to arrive at \eqref{Csums}, simply note that the density of edges between the vertices in $C_{J(i-1)+j}$ and $C_{J(i'-1)+j'}$ at time 0 is $u_{i i'}$, and that the density of edges between the vertices in $C_{J(i-1)+j}$ and $C_{J(i'-1)+j'}$ at time $t$ is $h_{j j'}$.

\medskip\noindent
{\bf 2.}
Next we establish \eqref{LSC2} when $u$ and $h$ are not block graphons by relying on a limiting argument. For $\ell \in \mathbb{N}$ and $g \in \mathscr{W}$, let $\hat g^{(\ell)} \in \mathscr{W}^{(\ell)}$ be the block graphon such that if $i,j \in [\ell]$ and $(x,y) \in [\frac{i-1}{\ell}, \frac{i}{\ell}) \times [\frac{j-1}{\ell}, \frac{j}{\ell}) =: B_{ij}^{(\ell)}$, then
\begin{equation}
\hat g^{(\ell)}(x,y) = \ell^2 \int_{B_{i,j}^{(\ell)}} \ddd x'\, \ddd y'\, g(x',y').
\end{equation}
Applying Lemma \ref{UnifCont}, we have, for any $\sigma \in \mathscr{M}$,
\begin{equation}
\label{LSCEq3}
\begin{aligned}
&|I_{1,t}(u,h^\sigma) - I_{1,t}(\hat u^{(\ell)},(\hat h^{(\ell)})^\sigma) | \\
&\leq \int_{[0,1]^2} \ddd x\, \ddd y\, \Delta_{I}\big(u(x,y)-\hat u^{(\ell)}(x,y), h(x,y) - (\hat h^{(\ell)})^\sigma(x,y)\big) \\
&\leq \Delta_I(\varepsilon, \eta) +  2\max \left\{ -\log \pab[t], - \log (1-\pbb[t]) \right\} \\
&\quad\times {\rm Leb}\Big\{(x,y) \in [0,1]^2\colon\, |u(x,y)-\hat u^{(\ell)}(x,y)| 
\geq \varepsilon \text{ or }  |h(x,y)-\hat h^{(\ell)}(x,y)| \geq \eta \Big\}.
\end{aligned}
\end{equation}
Because $\hat u^{(\ell)} \to u$ and $\hat h^{(\ell)} \to h$ in $L^2$ as $\ell \to \infty$ (\cite[Proposition 2.6]{C}), for any $\varepsilon, \eta >0$ the second term in the right-hand side of \eqref{LSCEq3} tends to 0 as $\ell \to \infty$. Letting $\varepsilon, \eta \downarrow 0$ and applying Lemma \ref{UnifCont} once more, we obtain 
\begin{equation}
\lim_{\ell \to \infty} \inf_{\sigma \in \mathscr{M}} I_{1,t} (\hat u^{(\ell)}, (\hat h^{(\ell)})^\sigma)  
= \inf_{\sigma \in \mathscr{M}} I_{1,t}(u,h^\sigma).
\end{equation}
Noting that convergence in $L^2$ implies convergence in the cut distance, and using the fact that we have already established \eqref{LSC2} for block graphons, we find
\begin{equation}
\begin{aligned}
\lim_{k \to \infty} &\inf_{\sigma \in \mathscr{M}}  \sup_{\{A_i\}_{i=1}^k \in \mathcal{A}_k}
\sum_{i,j=1}^k {\rm Leb}(A_i \times A_j) I_{1,t}\big(\widebar u(A_i \times A_j),\widebar{ h^{\sigma}}(A_i \times A_j)\big) \\
&=\lim_{\ell \to \infty}\lim_{k \to \infty} \inf_{\sigma \in \mathscr{M}}  \sup_{\{A_i\}_{i=1}^k \in \mathcal{A}_k}
\sum_{i,j=1}^k {\rm Leb}(A_i \times A_j) I_{1,t}\Big(\overline{\hat{u}^{(\ell)}}(A_i \times A_j)
\overline{( \hat{h}^{(\ell)})^{\sigma}}(A_i \times A_j)\Big)  \\
&=\lim_{\ell \to \infty} \inf_{\sigma \in \mathscr{M}} I_{1,t}\big(\hat u^{(\ell)}, (\hat h^{(\ell)})^\sigma\big)  
= \inf_{\sigma \in \mathscr{M}} I_{1,t}(u,h^\sigma),
\end{aligned}
\end{equation}
which completes the proof of \eqref{LSC2}.
\end{proof}


\subsection{Upper bound}
\label{subsec:UB}

We start by observing that
\begin{equation}
\label{MCP}
\mathbb{P}(\tilde f_{n,t} \in \, \cdot  \mid \tilde f_{n,0}=\tilde u_n)
= \mathbb{P}(\tilde f_{n,t} \in \, \cdot \mid f_{n,0}=u_n).
\end{equation}
Indeed, due to the fact that the dynamics is homogeneous, the outcome of $\tilde f_{n,t}$ is independent of the specific representative of $\tilde u_n$. We first establish the upper bound when $u \in \mathscr{W}^{(I)}$ for some $I\in{\mathbb N}$, i.e., the limiting initial graphon has a block structure. Afterwards we can use a limiting argument to obtain the upper bound for $u \in \mathscr{W}$, which we will not spell out. 

\begin{lemma}
Suppose that $u \in \mathscr{W}^{(I)}$ for some $I\in{\mathbb N}$, and $u_n^\eta \in \mathbb{B}_\square(u,\eta)$ for all $\eta>0$ and $n$ large enough. Then 
\begin{equation}
\lim_{\eta \downarrow 0} \limsup_{n \to \infty} \frac{1}{{n \choose 2}} \log 
\mathbb{P}(\tilde f_{n,t}  \in \tilde C \mid f_{n,t}=u_n^\eta) \leq -\inf_{\tilde x \in \tilde C} \tilde I_{1,t}(\tilde u,\tilde x)
\end{equation}
for any closed set $\tilde C \subseteq \tilde{\mathscr{W}}$.
\end{lemma}

\begin{proof}
By \cite[Lemma 4.1]{C}, it suffices to prove that, for all $\tilde h \in \tilde{\mathscr{W}}$,
\begin{equation}
\lim_{\varepsilon \downarrow 0} \lim_{\eta \downarrow 0} \limsup_{n \to \infty} \frac{1}{{n \choose 2}}\log 
\mathbb{P}( \tilde f_{n,t} \in \tilde{\mathbb{B}}_\square(\tilde h, \varepsilon) \mid \tilde f_{n,0} 
= \tilde u_n^\eta) \leq - \tilde I_{1,t}(\tilde u,\tilde h),
\end{equation}
which is equivalent to
\begin{equation}
\label{tltgnq}
\lim_{\varepsilon \downarrow 0} \lim_{\eta \downarrow 0} \limsup_{n \to \infty} \frac{1}{{n \choose 2}}
\log \mathbb{P}( f_{n,t} \in {\mathbb{B}}_\square(\tilde h, \varepsilon) \mid f_{n,0} 
= u_n^\eta) \leq - \tilde I_{1,t}(\tilde u,\tilde h),
\end{equation}
by which we have transferred the problem from $\tilde{\mathscr{W}}$ to $\mathscr{W}$. Note that to get \eqref{tltgnq} we have applied \eqref{MCP} to replace $\tilde f_{n,0} = \tilde u^\eta_n$ by $f_{n,0} = u^\eta_n$ in the condition. Since \eqref{MCP} holds for any $u^\eta_n$ in the equivalence class $\tilde u_n^\eta$, we may assume that there exists a $u$ in the equivalence class $\tilde u$ such that $d_\square(u^\eta_n, u)<\eta$ for all $n$ large enough. The proof consists of 6 steps.

\medskip\noindent
{\bf 1.}
In contrast to $\tilde{\mathbb{B}}_{\square}(\tilde h, \varepsilon)$, whose elements cling tightly to $\tilde h$, the elements of ${\mathbb{B}}_{\square}(\tilde h, \varepsilon)$ are scattered throughout $\mathscr{W}$. We therefore need a systematic method of collecting these elements. To this end we recall a version of Szemer\'edi's regularity lemma \cite[Theorem 3.1]{C}, which states that for any $\varepsilon >0$ there exist a constant $C(\varepsilon)<\infty$ and a set $\mathscr{W}(\varepsilon) \subseteq \mathscr{W}$ with $| \mathscr{W}(\varepsilon)| \leq C(\varepsilon)$ such that for any $f \in \mathscr{W}$ there exist $\phi \in \mathscr{M}$ and $g \in \mathscr{W}(\varepsilon)$ satisfying $d_\square(f^\phi, g)<\varepsilon$, and that for any $g \in \mathscr{W}(\varepsilon)$ there exists a $J\in {\mathbb N}$ such that $g \in \mathscr{W}^{(J)}$. Thus, if we let 
\begin{equation}
\mathbb{B}_\square( \mathscr{W}(\varepsilon),\varepsilon) 
= \left\{ f \in \mathscr{W}\colon\, \min_{g \in \mathscr{W}(\varepsilon)} d_\square(g,f) \leq \varepsilon \right\},
\end{equation}
then
\begin{equation}
\begin{aligned}
\{ f_{n,t} \in \mathbb{B}_\square (\tilde h, \varepsilon)\} 
&\subseteq \{ f_{n,t} \in \mathbb{B}_\square(\tilde h, \varepsilon) \} \cap 
\Big( \bigcup_{\sigma_n \in \mathscr{M}_n} \{ f^{\sigma_n}_{n,t} \in 
\mathbb{B}_{\square}(\mathscr{W}(\varepsilon),\varepsilon) \} \Big) \\
&= \bigcup_{g \in \mathscr{W}(\varepsilon)} \bigcup_{\sigma_n \in \mathscr{M}_n} 
\{ f_{n,t} \in \mathbb{B}_\square(\tilde h, \varepsilon) \} \cap
 \{ f_{n,t}^{\sigma_n} \in \mathbb{B}_\square(g,\varepsilon) \},
\end{aligned}
\end{equation}
where $\mathscr{M}_n$ is the set of permutations of the $n$ intervals of length $1/n$ in $[0,1]$. Because $\mathscr{W}(\varepsilon)$ is finite, it is enough to show that 
\begin{equation}
\begin{aligned}
\label{TSG}
\lim_{\eta \downarrow 0}\limsup_{n \to \infty}
& \frac{1}{{n \choose 2}}\log \mathbb{P}^{u^\eta_n}\Big(\bigcup_{\sigma_n \in \mathscr{M}_n} 
\{ f_{n,t} \in \mathbb{B}_\square(\tilde h, \varepsilon) \} \cap 
\{ f_{n,t}^{\sigma_n} \in \mathbb{B}_\square(g,\varepsilon) \} \Big) \\
&\leq - \tilde I_{1,t}(u,\tilde h) + E(\varepsilon) \qquad \forall\,g \in \mathscr{W}(\varepsilon),
\end{aligned}
\end{equation}
where $E(\varepsilon)$ must vanish as $\varepsilon \downarrow 0$. Note that the event in \eqref{TSG} is empty when $\delta_\square(\tilde g, \tilde h)> 2 \varepsilon$. We thus only need to establish \eqref{TSG} when $\delta_\square(\tilde g, \tilde h) \leq 2 \varepsilon$. Observe that the left-hand side of \eqref{TSG} is at most 
\begin{equation}
\begin{aligned}
\lim_{\eta \downarrow 0} \limsup_{n \to \infty} \frac{1}{{n \choose 2}} \log \mathbb{P}^{u^\eta_n} 
&\Big( \bigcup_{\sigma_n \in \mathscr{M}_n} \{ f_{n,t}^{\sigma_n} \in 
\mathbb{B}_\square(g,\varepsilon) \} \Big)\\
&\leq \lim_{\eta \downarrow 0} \limsup_{n \to \infty} \frac{1}{{n\choose 2}} \max_{\sigma_n \in \mathscr{M}_n} 
\log \mathbb{P}^{u^\eta_n} \left( f_{n,t}^{\sigma_n} \in \mathbb{B}_\square(g,\varepsilon) \right) 
\label{TSG1}
\end{aligned}
\end{equation}
because $\log(n!) = o(\binom{n}{2})$. To bound the right-hand side of \eqref{TSG1}, we show that we can replace $\mathscr{M}_n$ by a finite set $\mathcal{T}=\mathcal{T}(I,J,\eta)$ (whose cardinality does not depend on $n$) without incurring a significant error.

\medskip\noindent
{\bf 2.} We construct the set $\mathcal{T}=\mathcal{T}(I,J,\eta)$ as in the proof of \cite[Lemma 3.3]{DS}. Recall that $u \in \mathscr{W}^{(I)}$ and $g \in \mathscr{W}^{(J)}$, and write $0=a_0<a_1<\dots<a_I=1$ and $0=b_0 < b_1<\dots < b_J=1$ to denote their block endpoints. Define the intervals $A_i=[a_{i-1},a_i)$ and $B_{j}=[b_{j-1},b_j)$. Let
\begin{equation}
\label{Vdef}
V:= \left\{ (v_{i,j})_{i \in [I], j \in [J]}\colon\, v_{ij} \in (0,1), \sum_{j \in [J]} v_{ij} = {\rm Leb}(A_i),  
\sum_{i \in [I]} v_{ij} = {\rm Leb}(B_j) \right\},
\end{equation}
and for ${\boldsymbol v}  \in V$ define 
\begin{equation}
A_{ij} := \left[ a_i + \sum_{k=1}^{j-1} v_{i,k}, a_i + \sum_{k=1}^j v_{i,k} \right), 
\qquad B_{ji} := \left[ b_j + \sum_{k=1}^{i-1} v_{k,j}, b_j + \sum_{k=1}^i v_{k,j} \right).
\end{equation}
Pick $\tau_{\boldsymbol v} \in \mathscr{M}$ satisfying (see Figure~\ref{Ex2}) 
\begin{equation}
\label{tauC}
\tau_{\boldsymbol v}(A_{ij}) = B_{ij}, \quad \forall\, i \in [I], j\in [J].
\end{equation}
Concretely, this means that we choose $\tau_{\boldsymbol v}$ such that if $x \in A_{ij}$, then
\begin{equation}
\tau_{\boldsymbol v}(x) = \left( x - a_i + \sum_{k=1}^{j-1} v_{ik} \right) + \left(b_j + \sum_{k=1}^{i-1} v_{kj} \right).
\end{equation}
The map $\tau_{\boldsymbol v}$ can be understood as follows. For a vertex $v \in [n]$ and a set $A \subseteq [0,1]$, write $v \rightsquigarrow A$ when $[\frac{v-1}{n},\frac{v}{n}) \subseteq A$. Refer to $v$ such that $v \rightsquigarrow A_i$ as a type-$i$ vertex. The interval $A_{ij}$ contains roughly $n v_{ij}$ type-$i$ vertices, which are the only type-$i$ vertices that get mapped onto the interval $B_j$. Thus, under the map $\tau_{\boldsymbol v}$, $B_j$ contains roughly $n v_{ij}$ type-$i$ vertices. Note also that, after $\tau_{\boldsymbol v}$ has been applied, the labels of type-$i$ vertices inside each block are sorted in increasing order. 

\begin{figure}
\centering
\begin{tikzpicture}[scale=0.65]
\tikzset{vertex/.style = {shape=circle,draw,minimum size=1.3em}}
\tikzset{edge/.style = {->,> = latex'}}

\draw (-4,8) rectangle (1,13);
\draw[fill=gray] (-4,10.5) rectangle (-1.5,13);
\draw[fill=black] (-0.25,8) rectangle (1,9.25);
\filldraw [black] (-4,8) circle (2pt)node[anchor=east] {$a_3$};
\filldraw [black] (-4,9.25) circle (2pt)node[anchor=east] {$a_2$};
\filldraw [black] (-4,10.5) circle (2pt)node[anchor=east] {$a_1$};
\filldraw [black] (-4,13) circle (2pt)node[anchor=east] {$a_0$};
\filldraw [black] (1,13) circle (2pt)node[anchor=south] {$a_3$};
\filldraw [black] (-0.25,13) circle (2pt)node[anchor=south] {$a_2$};
\filldraw [black] (-1.5,13) circle (2pt)node[anchor=south] {$a_1$};
\filldraw [black] (1,10.92) circle (2pt);
\filldraw [black] (1,10.5) circle (2pt);
\filldraw [black] (1,9.25) circle (2pt);
\filldraw [black] (1,8.83) circle (2pt);
\filldraw [black] (1,8.42) circle (2pt);
\filldraw [black] (1,8) circle (2pt);
\draw[->,dashed] (1,8.21) -- (4,12.79);
\draw[->,dashed] (1,8.63) -- (4,10.29);
\draw[->,dashed] (1,9.05) -- (4,11.95);
\draw[->,dashed] (1,9.87) -- (4,11.12);
\draw[->,dashed] (1,10.71) -- (4,12.37);
\draw[->,dashed] (1,11.96) -- (4,9.04);
\node at (2.5,13) {\large $\sigma$};
\node at (-1.5,14.25) {\large $u$};

\draw (4,8) rectangle (9,13);
\draw[fill=black] (4,12.58) rectangle (4.42,13);
\draw[fill=black] (4.83,12.58) rectangle (5.25,13);
\draw[fill=black] (4,11.75) rectangle (4.42,12.17);
\draw[fill=black] (4.83,11.75) rectangle (5.25,12.17);
\draw[fill=gray] (4.42,12.17) rectangle (4.83,12.58);
\draw[fill=black] (4,10.08) rectangle (4.42,10.5);
\draw[fill=black] (4.83,10.08) rectangle (5.25,10.5);
\draw[fill=black] (6.5,10.08) rectangle (6.92,10.5);
\draw[fill=black] (6.5,11.75) rectangle (6.92,12.17);
\draw[fill=black] (6.5,12.58) rectangle (6.92,13);
\draw[fill=gray] (6.92,8) rectangle (9,10.08);
\draw[fill=gray] (6.92,12.17) rectangle (9,12.58);
\draw[fill=gray] (4.42,8) rectangle (4.83,10.08);
\filldraw [black] (4,8) circle (2pt)node[anchor=east] {$b_3$};
\filldraw [black] (4,11.75) circle (2pt)node[anchor=east] {$b_1$};
\filldraw [black] (4,10.5) circle (2pt)node[anchor=east] {$b_2$};
\filldraw [black] (4,13) circle (2pt)node[anchor=east] {$b_0$};
\filldraw [black] (9,13) circle (2pt)node[anchor=south] {$b_3$};
\filldraw [black] (5.25,13) circle (2pt)node[anchor=south] {$b_1$};
\filldraw [black] (6.5,13) circle (2pt)node[anchor=south] {$b_2$};
\node at (6.5,14.25) {\large $u^\sigma$};

\draw (12,8) rectangle (17,13);
\draw[fill=black] (12,11.75) rectangle (13.25,13);
\draw[fill=gray] (13.25,11.75) rectangle (14.5,13);
\draw[fill=gray] (12,10.5) rectangle (13.25,11.75);
\draw[fill=gray] (14.5,8) rectangle (17,10.5);
\filldraw [black] (12,8) circle (2pt)node[anchor=east] {$b_3$};
\filldraw [black] (12,11.75) circle (2pt)node[anchor=east] {$b_1$};
\filldraw [black] (12,10.5) circle (2pt)node[anchor=east] {$b_2$};
\filldraw [black] (12,13) circle (2pt)node[anchor=east] {$b_0$};
\filldraw [black] (17,13) circle (2pt)node[anchor=south] {$b_3$};
\filldraw [black] (13.25,13) circle (2pt)node[anchor=south] {$b_1$};
\filldraw [black] (14.5,13) circle (2pt)node[anchor=south] {$b_2$};
\node at (14.5,14.25) {\large $h$};

\draw (-4,16) rectangle (1,21);
\draw[fill=gray] (-4,18.5) rectangle (-1.5,21);
\draw[fill=black] (-0.25,16) rectangle (1,17.25);
\filldraw [black] (-4,16) circle (2pt)node[anchor=east] {$a_3$};
\filldraw [black] (-4,17.25) circle (2pt)node[anchor=east] {$a_2$};
\filldraw [black] (-4,18.5) circle (2pt)node[anchor=east] {$a_1$};
\filldraw [black] (-4,21) circle (2pt)node[anchor=east] {$a_0$};
\filldraw [black] (1,21) circle (2pt)node[anchor=south] {$a_3$};
\filldraw [black] (-0.255,21) circle (2pt)node[anchor=south] {$a_2$};
\filldraw [black] (-1.5,21) circle (2pt)node[anchor=south] {$a_1$};
\node at (-1.5,22.25) {\large $u$};

\draw (4,16) rectangle (9,21);
\draw[fill=gray] (4,20.58) rectangle (4.42,21);
\draw[fill=black] (4.42,19.75) rectangle (5.25,20.58);
\draw[fill=black] (8.58,19.75) rectangle (9,20.58);
\draw[fill=black] (4.42,16) rectangle (5.25,16.42);
\draw[fill=black] (8.58,16) rectangle (9,16.42);
\draw[fill=gray] (6.5,16.42) rectangle (8.58,18.5);
\draw[fill=gray] (4,16.42) rectangle (4.42,18.5);
\draw[fill=gray] (6.5,20.58) rectangle (8.58,21);
\filldraw [black] (4,16) circle (2pt)node[anchor=east] {$b_3$};
\filldraw [black] (4,19.75) circle (2pt)node[anchor=east] {$b_1$};
\filldraw [black] (4,18.5) circle (2pt)node[anchor=east] {$b_2$};
\filldraw [black] (4,21) circle (2pt)node[anchor=east] {$b_0$};
\filldraw [black] (9,21) circle (2pt)node[anchor=south] {$b_3$};
\filldraw [black] (5.25,21) circle (2pt)node[anchor=south] {$b_1$};
\filldraw [black] (6.5,21) circle (2pt)node[anchor=south] {$b_2$};
\filldraw [black] (1,20.58) circle (2pt);
\filldraw [black] (1,18.5) circle (2pt);
\filldraw [black] (1,17.25) circle (2pt);
\filldraw [black] (1,16.42) circle (2pt);
\filldraw [black] (1,16) circle (2pt);
\draw[->,dashed] (1,20.79) -- (4,20.79);
\draw[->,dashed] (1,19.54) -- (4,17.46);
\draw[->,dashed] (1,16.21) -- (4,16.21);
\draw[->,dashed] (1,16.83) -- (4,20.17);
\draw[->,dashed] (1,17.88) -- (4,19.12);
\node at (2.5,21) {\large $\tau$};
\node at (6.5,22.25) {\large $u^\tau$};

\draw (12,16) rectangle (17,21);
\draw[fill=black] (12,19.75) rectangle (13.25,21);
\draw[fill=gray] (13.25,19.75) rectangle (14.5,21);
\draw[fill=gray] (12,18.5) rectangle (13.25,19.75);
\draw[fill=gray] (14.5,16) rectangle (17,18.5);
\filldraw [black] (12,16) circle (2pt)node[anchor=east] {$b_3$};
\filldraw [black] (12,19.75) circle (2pt)node[anchor=east] {$b_1$};
\filldraw [black] (12,18.5) circle (2pt)node[anchor=east] {$b_2$};
\filldraw [black] (12,21) circle (2pt)node[anchor=east] {$b_0$};
\filldraw [black] (17,21) circle (2pt)node[anchor=south] {$b_3$};
\filldraw [black] (13.25,21) circle (2pt)node[anchor=south] {$b_1$};
\filldraw [black] (14.5,21) circle (2pt)node[anchor=south] {$b_2$};
\node at (14.5,22.25) {\large $h$};
\end{tikzpicture}
\caption{Illustration of the map $\tau$ in \eqref{tauC}.}
\label{Ex2}
\end{figure}
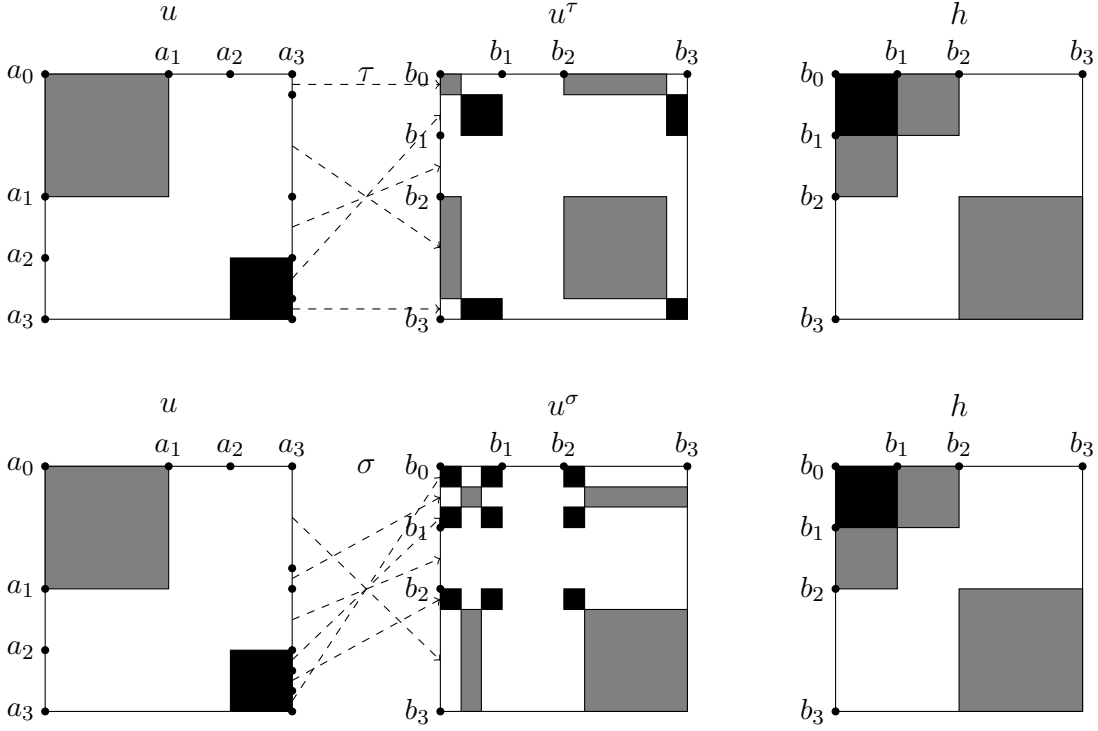

\medskip\noindent
{\bf 3.} 
We have now introduced all the objects that are needed to construct the set $\mathcal{T}$. We have the set $V$ and a mapping $\tau$ that relates elements of $V$ to permutations. As ${\mathcal T}$ must be {\it finite} while $V$ is uncountably infinite, we cannot let $\mathcal{T}$ be simply the image of $V$ under $\tau$. Instead, we construct a finite subset $\widebar V$ of $V$ such that any element of $V$ is close to an element of $\widebar V$ (exploiting the compactness of $V$), after which we let $\mathcal T$ be the image of $\widebar V$ under $\tau$. Concretely, we let $\widebar V \subseteq V$ be a finite set such that for any ${\boldsymbol u}  \in V$ there exists a ${\boldsymbol v} \in \widebar V$ with 
\begin{equation}
\label{TSG2}
\lVert {\boldsymbol u} -{\boldsymbol v}  \rVert_\infty<\frac{\eta}{2IJ}.
\end{equation}
After that we put $\mathcal{T} := \{ \tau_{\boldsymbol v}\colon\, \boldsymbol v \in \bar V \}$. It should be noted that if $\sigma \in \mathscr{M}$ and $C_{ij}(\sigma):= \{v \in [n] \colon\,v \rightsquigarrow A_i, \sigma_n(v) \rightsquigarrow B_j \}$, then for any $\sigma_n \in \mathscr{M}_n$ there exists $\tau \in \mathcal{T}$ such that   
\begin{equation}
\frac{1}{n}\sum_{i,j} |C_{ij}(\sigma_n) - C_{ij}(\tau) | < \eta,
\end{equation} 
provided $n$ is sufficiently large. In other words, for any $\sigma_n \in \mathscr{M}_n$ there exists a permutation $\tau \in \mathcal{T}$ that maps approximately the same proportion of type-$i$ vertices to the interval $B_j$ (see Figure \ref{Ex2} for an illustration). `Note that we require $n$ to be large to account for boundary effects, i.e., for $v \in [n]$ such that $[\frac{v-1}{n}, \frac{v}{n})$ is not contained in a single $A_{i}$ or such that $\tau([\frac{v-1}{n}, \frac{v}{n}))$ is not contained in a single $B_j$.

\medskip\noindent
{\bf 4.}
Let $\sigma_n^0 \in \mathscr{M}_n$ be a permutation that permutes the blocks $B_j$ only, and sorts the different vertices within $B_j$ in ascending order of their original label. Formally this means that $\sigma^0_n$ satisfies the following properties:
\begin{itemize}
\item[$\circ$]
If $\sigma_n(u) \rightsquigarrow B_j$, then $\sigma^0 \circ \sigma_n(u) \rightsquigarrow B_j$.
\item[$\circ$]
If $\sigma_n(u), \sigma_n(v) \rightsquigarrow B_j$ and $u \rightsquigarrow A_{i_1}$, $v \rightsquigarrow A_{i_2}$ with $i_1<i_2$, then $(\sigma_n^0 \circ \sigma_n)(u) < (\sigma_n^0 \circ \sigma_n)(v)$.
\item[$\circ$]
If $\sigma_n(u), \sigma_n(v) \rightsquigarrow B_j$ and $u,v \rightsquigarrow A_i$ with $u < v$, then $(\sigma^0_n \circ \sigma_n)(u) < (\sigma_n^0 \circ \sigma_n)(v)$.
\end{itemize}
Observe that, because $g = g^{\sigma_n^0}$, 
\begin{equation}
\mathbb{P}^{u_n^\eta}( f_{n,t}^{\sigma_n} \in \mathbb{B}_\square(g, \varepsilon) ) 
= \mathbb{P}^{u_n^\eta} \Big(f_{n,t}^{\sigma_n\, \circ\, \sigma^0_n} \in \mathbb{B}_\square(g, \varepsilon) \Big) 
= \mathbb{P}^{(u_n^\eta)^{\sigma_n\, \circ\, \sigma_n^0}}( f_{n,t} \in \mathbb{B}_\square(g,\varepsilon)),
\end{equation}
where the last equality follows from the fact that applying $\sigma_n \circ \sigma_0$ at time $0$ is equivalent to applying it at time $t$. Now, by \eqref{TSG2}, for any $\sigma_n \in \mathscr{M}_n$ there exists a $\tau \in \mathcal{T}$ such that 
\begin{equation}
d_\square\big( (u_n^\eta)^{\sigma_n\, \circ\, \sigma_n^0}, u^\tau\big) 
\leq d_\square\big( (u_n^\eta)^{\sigma_n\, \circ \,\sigma_n^0}, u^{\sigma_n\, \circ\, \sigma_n^0}\big) 
+ d_\square\big(u^{\sigma_n \,\circ\, \sigma_n^0}, u^\tau\big) \leq 2 \eta.
\end{equation}
Consequently, we have derived the upper bound
\begin{equation}
\begin{aligned}
\lim_{\eta \downarrow 0} \limsup_{n \to \infty} \frac{1}{{n\choose 2}}& \max_{\sigma_n \in \mathscr{M}_n} \log 
\mathbb{P}^{u^\eta_n} \left( \{ f_{n,t}^{\sigma_n} \in \mathbb{B}_\square(g,\varepsilon) \} \right)\\
&\leq \lim_{\eta \downarrow 0} \limsup_{n \to \infty} \frac{1}{{n\choose 2}} \max_{\tau \in \mathcal{T}} 
\sup_{\widebar u^\eta_n \in \mathbb{B}_\square(u^\tau, 2\eta) } 
\mathbb{P}^{\widebar u^\eta_n}(f_{n,t} \in \mathbb{B}_\square(g,\varepsilon))
\label{TSG3}
\end{aligned}
\end{equation}
for every $\varepsilon>0$ and $g \in \mathscr{W}$.

\medskip\noindent
{\bf 5.}
To further bound the right-hand side of \eqref{TSG3}, let $\{u_{ij} \}_{ i,j \in[ I]}$ and $\{ g_{\ell m}\}_{  \ell,m\in[ J]}$ be the block values of $u$ and $g$, respectively (for example, if $x \in A_i$ and $y \in A_j$, then $u(x,y)=u_{ij}$). We assume without loss of generality that $v_{ij} >0$ for all $i,j$ (recall \eqref{Vdef}; the $i,j$ with $v_{ij}=0$ can be ignored). Abbreviate
\begin{equation}
\eta_{i\ell jm} := \frac{2\eta}{v_{i\ell}v_{jm}}, \qquad \varepsilon_{i\ell jm} := \frac{\varepsilon}{v_{i\ell}v_{jm}}.
\end{equation}
Observe that, for each $ i,j \in[I]$ and $ k,\ell \in[J]$, if $d_\square(\widebar u^\eta_n, u^\tau) \leq 2\eta$, then 
\begin{equation}
\int_{A_{ik} \times A_{j \ell}} \ddd x\, \ddd y\, \widebar u_n(x,y) \in \left[u_{ij} 
- \eta_{i\ell jm}, u_{ij} +\eta_{i\ell jm}\right],
\end{equation} 
while if $f_{n,t} \in \mathbb{B}_\square(g, \varepsilon)$, then 
\begin{equation}
\int_{A_{ik} \times A_{j \ell}} \ddd x\, \ddd y\, f_{n,t} (x,y)  \in \left[ g_{k\ell} 
- \varepsilon_{i\ell jm} , g_{k\ell} + \varepsilon_{i\ell jm} \right].
\end{equation}
Because the rectangle $A_{i\ell} \times A_{jm}$ represents $n^2 v_{i\ell} v_{jm}$ independently evolving edges, we have 
\begin{equation}
\begin{aligned}
\limsup_{n \to \infty} \frac{1}{n^2 v_{i\ell} v_{jm}} \log 
&\,\mathbb{P} \Bigg( \left. \int_{A_{i\ell}\times A_{jm}} \ddd x\, \ddd y\,f_{n,t}(x,y) 
\in  \left[ g_{\ell m} - \varepsilon_{i\ell jm} , g_{\ell m} + \varepsilon_{i\ell jm} \right] ~\right| \\ 
&\qquad\qquad \int_{A_{i\ell} \times A_{jm}} \ddd x\, \ddd y\,f_{n,0}(x,y) \in  \left[u_{ij} 
- \eta_{i\ell jm}, u_{ij} + \eta_{i\ell jm}\right] \Bigg)\\
&\leq - \inf_{\widebar u \in \left[u_{ij} - \eta_{i\ell jm}, u_{ij} + \eta_{i\ell jm}\right], \, 
x \in \left[ g_{\ell m} - \varepsilon_{i\ell jm} , g_{\ell m} 
+ \varepsilon_{i\ell jm} \right]} I_{1,t}(\widebar u,x),
\end{aligned}
\end{equation}
while this upper bound must be multiplied by $\tfrac{1}{2}$ when $(i,\ell) = (j,m)$ (to avoid double counting). This leads to 
\begin{equation}
\begin{aligned}
\limsup_{n \to \infty} \frac{1}{{n \choose 2}}
& \log \mathbb{P}^{\widebar u_n^\eta}(f_{n,t} 
\in \mathbb{B}_\square(g,\varepsilon)) \\
&\leq - \sum_{i\ell jm} v_{i\ell}v_{jm} \inf_{\widebar u \in \left[u_{ij} - \eta_{i\ell jm}, u_{ij} 
+ \eta_{i\ell jm}\right], \, x \in \left[ g_{k\ell} - \varepsilon_{ikj\ell} , g_{k\ell} 
+ \varepsilon_{ikj\ell} \right]} I_{1,t}(\widebar u,x) \\
&\leq -I_{1,t}(u^\tau, g) + \sum_{i\ell jm} v_{i\ell}v_{jm}\, \Delta_I\left(\eta_{i\ell jm} 
\wedge 1, \varepsilon_{i\ell jm} \wedge 1\right),
\end{aligned}
\end{equation}
where $\Delta_I(\eta,\varepsilon)$ is defined in \eqref{LSC0}. Regarding the last inequality, note that because we are dealing with block graphons the integral in the definition of $I_{1,t}$ can be expressed as a sum with weights given by $v_{ij}$. Set $\gamma>0$. If $v_{i\ell}v_{jm}>\gamma$, then 
\begin{equation}
v_{i\ell}v_{jm}\, \Delta_I\left(\eta_{i\ell jm} \wedge 1, \varepsilon_{i\ell jm} \wedge 1\right) 
\leq \Delta_I \left(\tfrac{2\eta}{\gamma}, \tfrac{\varepsilon}{\gamma}\right),
\end{equation}
whereas if $v_{i\ell}v_{jm}\leq\gamma$ then, by Lemma \ref{UnifCont},
\begin{equation}
v_{i\ell}v_{jm} \Delta_I\left(\eta_{i\ell jm},\tfrac{\varepsilon}{\gamma}\right) \leq C \gamma
\end{equation}
with $C:=\max\{-\log p_{01,t}, - \log(1-p_{11,t}\}$. Consequently, for any $\boldsymbol v \in V$ and $\gamma>0$, we have 
\begin{equation}
\sum_{i\ell jm} v_{i\ell}v_{jm}\, \Delta_I\left(\eta_{i\ell jm} \wedge 1, 
\varepsilon_{i\ell jm} \wedge 1\right) \leq I^2 J^2\left(C \gamma 
+ \Delta_I \left(\tfrac{2 \eta}{\gamma}, \tfrac{\varepsilon}{\gamma}\right)\right)
=:E(\gamma,\eta,\varepsilon).
\end{equation}
Combining the above formulas, we arrive at
\begin{equation}
\label{Fclaim2}
\begin{aligned}
 \lim_{\eta \downarrow 0} \limsup_{n \to \infty} \frac{1}{{n\choose 2}} \max_{\tau \in \mathcal{T}} &
 \sup_{\widebar u^\eta_n \in \mathbb{B}_\square(u^\tau, 2\eta) } 
 \mathbb{P}^{\widebar u^\eta_n}(f_{n,t} \in \mathbb{B}_\square(g,\varepsilon)) \\
& \leq -\lim_{\eta \downarrow 0} \min_{\tau \in \mathcal{T}} \big[I_{1,t}(u^\tau,g) + E(\gamma,\eta,\varepsilon)\big] \\
 &\leq -\inf_{\phi \in \mathscr{M}} I_{1,t}(u^\phi,g) +E(\gamma,0,\varepsilon) 
 = - \tilde I_{1,t}(\tilde u, \tilde g) + E(\gamma,0,\varepsilon).
\end{aligned}
\end{equation}
Picking $\gamma=\varepsilon^{1/2}$, we can apply Lemma \ref{UnifCont}, to obtain 
\begin{equation}
\label{Fclaim3}
E(\varepsilon) := E(\varepsilon^{1/2},0,\varepsilon) \downarrow 0, \qquad  \varepsilon \downarrow  0.
\end{equation}

\medskip\noindent
{\bf 6.}
We can now finally prove \eqref{TSG}. Recall that we only need to consider $g\in \mathscr{W}(\varepsilon)$ such that $\tilde \delta_\square(\tilde g, \tilde h) \leq 2\varepsilon$. In view of \eqref{TSG3}, \eqref{Fclaim2} and \eqref{Fclaim3}, it is enough to show that 
\begin{equation}
\label{Fclaim}
- \tilde I_{1,t}(\tilde u, \tilde g) \leq  -\tilde I_{1,t}(\tilde u, \tilde h) + E(2\varepsilon)
\end{equation}
for all $\tilde g \in \mathscr{W}$ such that $\tilde \delta_\square(\tilde g, \tilde h) \leq 2\varepsilon$. Without loss of generality we may assume that $d_\square(g,h) \leq 2\varepsilon$. Following a similar line of reasoning as above, we get 
\begin{equation}
-I_{1,t}(u^\sigma,g) \leq -I_{1,t}(u^\sigma,h) + E(2\varepsilon) \qquad \forall\, \sigma \in \mathscr{M},
\end{equation}
which implies \eqref{Fclaim}.
\end{proof}


\subsection{Lower bound}
\label{subsec:LB}

To establish the lower bound, it suffices to prove that 
\begin{equation}
\label{UG0}
\lim_{\varepsilon \downarrow 0} \lim_{\eta \downarrow 0} \liminf_{n \to \infty} \frac{1}{{n \choose 2}} \log 
\mathbb{P}\big(\tilde f_{n,t} \in \tilde{\mathbb{B}}_\square(\tilde h, \varepsilon) \mid \tilde f_{n,0} = \tilde u_n^\eta\big) 
\geq - \tilde I_{1,t}(\tilde u, \tilde h) \qquad \forall\, \tilde h \in \tilde{\mathscr{W}}. 
\end{equation}
For any $\beta>0$ there exists a $\phi(\beta) \in \mathscr{M}$ such that 
\begin{equation}
\tilde I_{1,t}(\tilde u, \tilde h) \geq I_{1,t}(u, h^\phi) - \beta.
\end{equation}
Because $\mathbb{B}_\square(h^\phi, \varepsilon) \subseteq \mathbb{B}_\square(\tilde h, \varepsilon)$ for any $\phi \in \mathscr{M}$, picking $h = h^{\phi(\beta)}$ and letting $\beta \downarrow 0$, we see that \eqref{UG0} follows once we show that
\begin{equation}
\label{UG}
\lim_{\varepsilon \downarrow 0} \lim_{\eta \downarrow 0} \liminf_{n \to \infty} \frac{1}{{n \choose 2}} \log 
\mathbb{P}\big(f_{n,t} \in \mathbb{B}_\square(h, \varepsilon) \mid f_{n,0} = u_n^\eta\big) \geq - I_{1,t}(u,h)
\qquad \forall\,u, h \in \mathscr{W}.
\end{equation}
The proof comes in 5 steps and is constructed around a series of technical lemmas (Lemmas~\ref{A2}--\ref{LLN} below).

\medskip\noindent
{\bf 1.}
To prove \eqref{UG}, we first introduce some notation. As before, we work with block graphons. For $k \in \mathbb{N}$ and $i,j \in [k]$, let  $B_{i,j}^{(k)} := [\tfrac{i-1}{k}, \tfrac{i}{k}) \times [\tfrac{j-1}{k}, \tfrac{j}{k})$. For $g \in \mathscr{W}$, we let $\hat g_k\in \mathscr{W}^{(k)}$ be defined at the bottom-left corner points of $B_{i,j}^{(k)}$ by
\begin{equation}
\hat g_k\Big(\tfrac{i-1}{k},\tfrac{j-1}{k}\Big) := {k^2} \int_{B_{i,j}^{(k)}} \ddd x\, \ddd y\,g(x,y),
\end{equation}
and for $(x,y) \in B_{i,j}^{(k)}$ as
\begin{equation}
\hat g_k(x,y) := \hat g_k\Big(\tfrac{i-1}{k},\tfrac{j-1}{k}\Big).
\end{equation}
We settle \eqref{UG} by using a Cram\'er-transform-type argument, i.e., we rely on a particular change of measure. Concretely, for $z,x \in [0,1]$, let 
\begin{equation} 
\label{as}
\tau_t(z,x)
:= \argmax_{v\in\mathbb{R}} \big[vx - J_{t,v}(z)\big],
\end{equation}
where $J_{t,v}(z)$ is the function defined in \eqref{ee}. The idea is to use $\tau_t$ to describe the probability that particular edges are active at time $t$ when $f_{n,t}$ is conditioned to be close to $h$. To that end, abbreviate $\theta_{k,t}(x,y):=\exp({\tau_t(\hat u_k(x,y),\hat h_k(x,y))})$, let
\begin{equation}
\begin{aligned}
\alpha_{k,t}(x,y) &:= \frac{\pbb[t]\,\theta_{k,t}(x,y)}
{1-\pbb[t] + \pbb[t]\,\theta_{k,t}(x,y)},\\ 
\beta_{k,t}(x,y) &:= \frac{\pab[t]\,\theta_{k,t}(x,y)}
{1 - \pab[t] + \pab[t]\,\theta_{k,t}(x,y)},
\end{aligned}
\end{equation}
and for $\eta>0$ put
\begin{equation}
\label{qkn}
q^\eta_{k,n}(x,y) 
:= u^\eta_n(x,y) \alpha_{k,t}(x,y) + [1- u^\eta_n(x,y)] \beta_{k,t}(x,y).
\end{equation}
For $i,j \in [k]$, let 
\begin{equation}
\label{LR1}
q^\eta_k(i,j,n) := \frac{1}{{\rm Leb}(B_{i,j}^{(n)})} \int_{B_{i,j}^{(n)}} \ddd x\, \ddd y\, q^\eta_{n,k}(x,y).
\end{equation}
We can informally interpret $q^\eta_k(i,j,n)$ as an estimate of the probability that edge $(i,j)$ is active at time $t$ given that $f_{n,t}$ is close to $h$. Note that $q^\eta_{k}(i,j,n)$ depends not only on whether edge $(i,j)$ is initially active (dependence on $n$), but also on the proportion of other `nearby' edges that are initially active (dependence on $k$). 

\medskip\noindent
{\bf 2.}
In the following lemma we show that $q^\eta_{k,n} \in \mathscr{W}$ converges to $h$ in an appropropriate limit.

\begin{lemma}
\label{A2}
For every $t>0$,
\begin{equation}
\label{dl}
\lim_{k \to \infty} \lim_{\eta \downarrow 0} \lim_{n \to \infty} d_{\square}(h,q^\eta_{k,n}) = 0.
\end{equation}
\end{lemma}

\begin{proof} 
First note that
\begin{equation}
\label{dl1}
d_\square( q^\eta_{k,n}, h )  \leq d_\square( q^\eta_{k,n}, \hat h_k ) + d_{\square}( h, \hat h_k).
\end{equation}
We next analyse $d_\square( q^\eta_{k,n}, \hat h_k )$. From \eqref{as}, we find the first-order condition
\begin{equation}
\label{HF}
\hat{h}_k(x,y) = \hat u_k(x,y)\, \alpha_{k,t}(x,y) + [1- \hat u_k(x,y)]\, \beta_{k,t}(x,y) \qquad \forall\,x,y \in [0,1].
\end{equation}
Using \eqref{HF}, we obtain (rewrite the supremum over subsets $S,T \subseteq [0,1]$ in \eqref{cutdist} as a supremum over functions $a,b\colon [0,1] \to [0,1]$) 
\begin{equation}
\label{dl2}
\begin{aligned}
&d_\square(q^\eta_{k,n}, \hat h_k)\\ 
&= \sup_{a,b} \int_{[0,1]^2} \ddd x\, \ddd y\,a(x)b(y)
\bigg[\alpha_{k,t}(x,y) [u^\eta_n(x,y) - \hat{u}_k(x,y)]
+ \beta_{k,t}(x,y) [\hat u_k(x,y) - u^\eta_n(x,y)] \bigg]\\
&\leq \sup_{a,b} \int_{[0,1]^2} \ddd x\, \ddd y\,a(x)b(y) 
\bigg[\alpha_{k,t}(x,y) [u^\eta_n(x,y) - {u}(x,y)]
+ \beta_{k,t}(x,y) [u(x,y) - u^\eta_n(x,y)] \bigg]\\
&\qquad\qquad\qquad + \int_{[0,1]^2} \ddd x\, \ddd y\,| u(x,y) - \hat u_k(x,y) |\\
&\leq 2 k^2 d_\square (u , u^\eta_n) + \int_{[0,1]^2} \ddd x\,\ddd y\,| u(x,y) - \hat u_k(x,y) |, 
\end{aligned}
\end{equation}
where we note that $\alpha_{k,t}(x,y) ,\beta_{k,t}(x,y)\in[0,1]$, and use the triangle inequality in combination with the fact that the pair $(\hat u_k, \hat h_k)$ can take at most $k^2$ different values (corresponding to their values on the interior of $B_{i,j}^{(k)}$). The claim now follows from \eqref{dl1} and \eqref{dl2}, because $d_\square(u,u^\eta_n) \downarrow 0$, $\hat h_k \to h$ in $L^2$ and $\hat u_k \to u$ in $L^2$ (see \cite[Proposition 2.6]{C}).
\end{proof}

\medskip\noindent
{\bf 3.}
The next step is to construct a random variable $H_n$ on simple graphs with $n$ vertices by declaring that, for every $i,j\in[n]$ with $i<j$, vertices $i$ and $j$ are connected by an edge with probability $q^\eta_{n,k}(i,j)$ defined in \eqref{LR1}. Let $\mathbb{P}^\eta_{n,k,h}$ denote the law of $H_n$. Note that if $f \in \mathscr{W}_n$, then 
\begin{equation}
\label{LR2}
\begin{aligned}
\mathbb{P}^\eta_{n,k,h}( \{ f\}) &= \prod_{1 \leq i < j \leq n} q^\eta_{k}(i,j,n)^{f_{ij}}[1-q^\eta_{k}(i,j,n)]^{1-f_{ij}},\\
\mathbb{P}^\eta_{n,t}( \{ f \}) 
&= \prod_{1 \leq i < j \leq n} \pbb[t]^{u^\eta_n(i,j) f_{ij}} \pab[t]^{[1-u^\eta_n(i,j)]f_{ij}}\\ 
&\qquad\qquad \times (1- \pbb[t])^{u^\eta_n(i,j)(1-f_{ij} )} (1-\pab[t])^{[1-u^\eta_n(i,j)](1 -f_{ij})},
\end{aligned}
\end{equation}
where $f_{ij}$ denotes the value of $f(x,y)$ on the interior of $B_{i,j}^{(n)}$.

\begin{lemma}
\label{A3}
For fixed $k\in\mathbb{N}$,
\begin{equation}
\label{Pclaim}
\lim_{\eta \downarrow 0} \lim_{n \to \infty} \frac{1}{{n \choose 2}} \int_{\mathscr{W}_n}
\left(\log \frac{\ddd \mathbb{P}^\eta_{n,k,h}}{\ddd \mathbb{P}^\eta_{n,t}}\right)\, 
\ddd \mathbb{P}^\eta_{n,k,h} = I_{1,t}(\hat u_k, \hat h_k).
\end{equation}
\end{lemma} 

\begin{proof} 
Let $n_\ell=\ell k$ with $\ell\in\mathbb{N}$, and note that, for all $i,j \in [n_\ell]$,
\begin{equation}
B_{i,j}^{(n_\ell)} \subseteq B_{\lceil i/k \rceil, \lceil j/k \rceil}^{(k)}.
\end{equation} 
Rewriting \eqref{qkn} as a product, and applying \eqref{LR1} and \eqref{LR2}, we obtain that, for any $f \in \mathscr{W}_n$, \begin{equation}
\label{eq:ratio}
\begin{aligned}
\frac{\mathbb{P}^\eta_{n,k,h}(f)}{\mathbb{P}^\eta_{n,t}(f)}
&= \prod_{1 \leq i < j \leq n} \bar F_{k,i,j}(1,1)^{u_n^\eta(i,j)f_{ij} } \bar F_{k,i,j}(1,0)^{[1-u_n^\eta(i,j)]f_{ij}}\\ 
&\qquad\qquad \times \bar F_{k,i,j}(0,1)^{u_n^\eta(i,j)(1-f_{ij})} \bar F_{k,i,j}(0,0)^{[1-u_n^\eta(i,j)](1-f_{ij})},
\end{aligned}
\end{equation}
where, for $\ell, m \in \{0,1\}$,
\begin{equation}
F_{k,x,y}(\ell, m) = \frac{(\theta_{k,t}(x,y))^\ell}
{1-p_{m1,t} + p_{m1,t} \,\theta_{k,t}(x,y)},\hspace{0.21cm}
\bar F_{k,i,j}(\ell, m) = \frac{(\bar\theta_{k,t}(i,j))^\ell}
{1-p_{m1,t} + p_{m1,t} \,\bar\theta_{k,t}(i,j)},
\end{equation}
with
\begin{equation}
\bar\theta_{k,t}(i,j):=\theta_{k,t}\left(\tfrac{i-1}{k},\tfrac{j-1}{k}\right).
\end{equation}
From \eqref{eq:ratio}, elementary computations yield
\begin{equation}
\label{Feq1}
\begin{aligned}
&\lim_{\eta \downarrow 0} \lim_{n \to \infty} \frac{1}{{n \choose 2}} \int_{\mathscr{W}_n} 
\left(\log \frac{\ddd \mathbb{P}^\eta_{n,k,h} }{\ddd\mathbb{P}^\eta_{n,t} }(f)\right)\,
\ddd \mathbb{P}^\eta_{n,k,h}(f) \\
&=\lim_{\eta \downarrow 0} \lim_{n \to \infty} \frac{1}{{n \choose 2}} 
\sum_{1 \leq i < j \leq n} \int_{\mathscr{W}} \bigg[ u^\eta_n(i,j) f_{ij} \log \bar F_{k,i,j}(1,1) 
+[1-u^\eta_n(i,j)] f_{ij} \log \bar F_{k,i,j}(1,0) \\
&\qquad + u^\eta_n(i,j)(1- f_{ij}) \log \bar F_{k,i,j}(0,1)
+[1-u^\eta_n(i,j)] (1- f_{ij}) \log \bar F_{k,i,j}(0,0) \bigg] \ddd \mathbb{P}^\eta_{n,k,h}(f) \\
&= \lim_{\eta \downarrow 0} \lim_{n \to \infty} \frac{1}{{n \choose 2}} 
\sum_{1 \leq i < j \leq n} \bigg[ \tau_t\big(\hat u_k(i,j), \hat h_k(i,j)\big) 
\int_{\mathscr{W}}  f_{ij}\, \ddd \mathbb{P}^\eta_{n,k,h}(f) + u^\eta_n(i,j) \log \bar F_{k,i,j}(0,1)  \\
&\qquad +  [1-u_n^\eta(i,j)] \log \bar F_{k,i,j}(0,0) \bigg] \\
&=\lim_{\eta \downarrow 0} \lim_{n \to \infty} \frac{1}{{n \choose 2}} 
\sum_{1 \leq i < j \leq n} \bigg[ u^\eta_n(i,j)\left\{ \tau_t\big(\hat u_k(i,j), \hat h_k(i,j)\big)\, \pbb[t] \bar F_{k,i,j}(1,1) 
+ \log \bar F_{k,i,j}(0,1) \right\} \\
&\qquad + [1-u^\eta_n(i,j)] \left\{ \tau_t\big(\hat u_k(i,j), \hat h_k(i,j)\big)\, \pab[t] \bar F_{k,i,j}(1,0) 
+ \log \bar F_{k,i,j}(0,0) \right\}  \bigg].
\end{aligned}
\end{equation}
Using that $\hat u_k(x,y)$ is constant on the interior of $B_{i,j}^{(k)}$ for all $1 \leq i,j \leq k$, in combination with the fact that $\lim_{\eta \downarrow 0} \lim_{n \to \infty} d_\square (u^\eta_n, u) = 0$, we obtain that \eqref{Feq1} equals
\begin{equation}
\label{Feq2}
\begin{aligned}
&\int_{[0,1]^2} \ddd x\, \ddd y\, \bigg[ \hat u_k(x,y)\left\{ \tau_t\big(\hat u_k(x,y), 
\hat h_k(x,y)\big)\, \pbb[t] F_{k,x,y}(1,1) + \log F_{k,x,y}(0,1) \right\} \\
& + [1-\hat u_k(x,y)] \left\{ \tau_t\big(\hat u_k(x,y), \hat h_k(x,y)\big)\, \pab[t] F_{k,x,y}(1,0) 
+ \log F_{k,x,y}(0,0) \right\}  \bigg].
\end{aligned} 
\end{equation}
Applying \eqref{HF}, we see that \eqref{Feq2} equals
\begin{equation}
\begin{aligned}
& \int_{[0,1]^2} \ddd x \,\ddd y\, \bigg[ \tau_t(\hat u_k(x,y), \hat h_k(x,y) ) \hat h_k(x,y) + \hat u_k(x,y) \log F_{k,x,y}(0,1) \\
&\qquad  + [1- \hat u_k(x,y)] \log  F_{k,x,y}(0,0)\bigg] = I_{1,t}(\hat u_k, \hat h_k),
\end{aligned} 
\end{equation}
which settles the claim in \eqref{Pclaim} along subsequences $n_\ell$ of the form $\ell k$ with $\ell\in\mathbb{N}$. Straightforward reasoning gives the same along full sequences: the resulting discrepancies correspond to sets of vanishing Lebesgue measure (see the proof of \cite[Proposition 2.6]{C} for a similar argument).
\end{proof}

\medskip\noindent
{\bf 4.}
Two further lemmas are needed.

\begin{lemma}
\label{A4}
For every $t>0$,
\begin{equation}
\lim_{k \to \infty} I_{1,t}(\hat u_k, \hat h_k) = I_{1,t}(u,h).
\end{equation}
\end{lemma}

\begin{proof}
The claim follows from \eqref{LSC0} and the fact that $\hat u_k \stackrel{L^2}{\to} u$ and $\hat h_k \stackrel{L^2}{\to} h$ as $k \to \infty$.
\end{proof}

\begin{lemma}
\label{LLN}
For fixed $k\in\mathbb{N}$ and $\varepsilon>0$,
\begin{equation}
\lim_{n \to \infty} \mathbb{P}^\eta_{n,k,h}(\mathbb{B}_\square(q^\eta_{n,k} , \varepsilon)) =1.
\end{equation}
\end{lemma}

\begin{proof}
The claim follows from the same argument as in the proof of \cite[Lemmas 5.6 and 5.8-- 5.11]{C}.
\end{proof}

\medskip\noindent
{\bf 5.}
We are now ready to prove \eqref{UG}. For fixed $k \in \mathbb{N}$ and $\varepsilon>0$,
\begin{equation}
\begin{aligned}
&\mathbb{P}^\eta_{n,t} ( \mathbb{B}_\square(q^\eta_{n,k}, \varepsilon))\\ 
&= \int_{\mathbb{B}_\square(q^\eta_{n,k},\varepsilon)} \ddd \mathbb{P}^\eta_{n,t} 
= \int_{\mathbb{B}_\square(q^\eta_{n,k},\varepsilon)} \exp \left( -\log \frac{\ddd 
\mathbb{P}^\eta_{n,k,h}}{\ddd \mathbb{P}^\eta_{n,t}}\right) \ddd \mathbb{P}^\eta_{n, k,h} \\
&= \mathbb{P}^\eta_{n,k,h} (\mathbb{B}_\square(q^\eta_{n,k},\varepsilon)) 
\frac{1}{ \mathbb{P}^\eta_{n,k,h}(\mathbb{B}_\square(q^\eta_{n,k}, \varepsilon)) }
\int_{\mathbb{B}_\square(q^\eta_{n,k},\varepsilon)} \exp \left( -\log \frac{\ddd 
\mathbb{P}^\eta_{n,k,h}}{\ddd \mathbb{P}^\eta_{n,t}}\right) \ddd \mathbb{P}^\eta_{n, k, h}.
\end{aligned}
\end{equation}
Therefore, by Jensen's inequality,
\begin{equation}
\begin{aligned}
\log \mathbb{P}^\eta_{n,t} ( \mathbb{B}_\square( q^\eta_{n,k}, \varepsilon))
&\geq \log \mathbb{P}^\eta_{n,k,h}(\mathbb{B}_\square(q_{n,k}, \varepsilon))\\
&\hspace{0.2cm} -\frac{1}{\mathbb{P}^\eta_{n,k,h} (\mathbb{B}_\square(q^\eta_{n,k}, \varepsilon)) }
\int_{ \mathbb{B}_\square(q^\eta_{n,k}, \varepsilon) }  
\left(\log \frac{\ddd \mathbb{P}^\eta_{n,k,h}}{\ddd \mathbb{P}^\eta_{n,t}}\right) 
\,\ddd \mathbb{P}^\eta_{n, k, h}.
\end{aligned}
\end{equation}
By Lemma \ref{LLN}, $\mathbb{P}_{n,k,h}(\mathbb{B}_\square(q_{n,k}, \varepsilon)) \to 1$, which implies that
\begin{equation}
\label{Pprop}
\lim_{\eta \downarrow 0} \liminf_{n \to \infty} \frac{1}{{n \choose 2}} \log 
\mathbb{P}^\eta_{n,t}(\mathbb{B}_\square(q^\eta_{n,k},\varepsilon)) 
\geq - \lim_{\eta \downarrow 0} \lim_{n \to \infty} \frac{1}{{n \choose 2}} \int
\left(\log \frac{\ddd \mathbb{P}^\eta_{n,k,h}}{\ddd \mathbb{P}^\eta_{n,t}}\right)\, \ddd \mathbb{P}^\eta_{n,k,h}.
\end{equation}
According to Lemma \ref{A3}, the right-hand side equals $-I_{1,t}(\hat u_k, \hat h_k)$. Since
\begin{equation}
\begin{array}{lll}
&\lim_{k \to \infty} \lim_{\eta \downarrow 0} \lim_{n \to \infty} d_\square (q^\eta_{n,k},h)=0 &\mbox{(by Lemma \ref{A2})},\\ 
&\lim_{k \to \infty} I_{1,t}(\hat u_k , \hat h_k) = I_{1,t}(u,h) &\mbox{(by Lemma \ref{A4})}, 
\end{array}
\end{equation}
if we let $k \to \infty$, then we obtain from \eqref{Pprop} that 
\begin{equation}
\lim_{\eta \downarrow 0} \liminf_{n \to \infty} \frac{1}{{n \choose 2}} \log \mathbb{P}^\eta_{n,t}(\mathbb{B}_\square(h,\eta)) \geq - I_{1,t}(u,h),
\end{equation}
where we use that $\mathbb{B}_\square(q^\eta_{n,k},\varepsilon) \supseteq \mathbb{B}_\square(h,\eta)$ for $0 < \eta < \varepsilon$ and $n$ large enough.
\qed

\medskip
Collecting the results in Sections \ref{subsec:LSC}--\ref{subsec:LB}, we see that we have completed the proof of Theorem~\ref{tTwoLDP}.


\section{Proofs of the multi-point and sample-path LDPs} 
\label{S4}


\subsection{Proof of the multi-point LDP}
\label{PrLdp} 

The objective of this section to prove Theorem \ref{tMultiLDP} with the help of Theorem \ref{tTwoLDP}. The structure aligns with Section~\ref{S3}: lower semi-continuity, lower bound, upper bound.


\subsubsection{Lower semi-continuity}

For $\tilde h, \tilde g \in \tilde{\mathscr{W}}^{(\ell)}$, let 
\begin{equation}
\delta_\square^{\ell}(\tilde h, \tilde g) := \max_{i \in[ \ell]} \delta_\square(\tilde h_i, \tilde g_i).
\end{equation}
By Lemma \ref{LwrSC}, for any sequence $(\tilde h_n)_{n\in\mathbb{N}}$ in $\tilde{\mathscr{W}}^{|j|+1}$ such that $\lim_{n\to\infty} \delta_\square^{|j|+1} (\tilde h_n, \tilde h) = 0$,
\begin{equation}
\liminf_{n \to \infty} \tilde I_j(\tilde h_n) = \liminf_{n \to \infty} \sum_{i=1}^{|j|} 
\tilde I_{1,t_{i}-t_{i-1}}(\tilde h_{n,i-1}, \tilde h_{n,i}) \geq \sum_{i=1}^{|j|} 
\tilde I_{1,t_{i}-t_{i-1}}(\tilde h_{i-1}, \tilde h_{i})=\tilde I_j(\tilde h),
\end{equation}
which settles the lower semi-continuity of $\tilde I_j$.


\subsubsection{Lower bound}

Let $\tilde h=(\tilde h_i)_{i=0}^{|j|} \in \tilde{\mathscr{W}}^{|j|+1}$ with $\tilde h_0 \equiv \tilde u$, and let $\boldsymbol{\eta}= (\eta_i)_{i=0}^{|j|} \in (0,1)^{|j|+1}$. Define  
\begin{equation}
\tilde{\mathbb{B}}^{|j|}_\square(\tilde h, \boldsymbol{\eta})
= \big\{ \tilde g \in \tilde{\mathscr{W}}^{|j|+1}\colon\, \delta_\square(\tilde h_i,\tilde g_i) 
\leq \eta_i \,\,\quad\forall i \in[ |j|] \big\}.
\end{equation}
For $n$ large enough such that $\tilde u_n \in \tilde{\mathbb{B}}_\square (\tilde h_0, \eta_0)$, the Markov property yields
\begin{equation}
\label{MP1}
\begin{aligned}
&\frac{1}{{n\choose 2}} \log  (\tilde \mu_n \circ p_j^{-1})\big(\tilde{\mathbb{B}}_{\square}^{|j|}(h,\boldsymbol{\eta})\big)\\ 
&\qquad \geq \frac{1}{{n\choose 2}} \sum_{i=1}^{|j|} \inf_{\tilde w \in \tilde{\mathbb{B}}_{\square}(\tilde h_{i-1}, \eta_{i-1}) 
\cap \tilde{\mathscr{W}}_n} \log  \mathbb{P}\Big(\tilde f_{n,t_i} \in \tilde{\mathbb{B}}_\square(\tilde h_i, \eta_i)
~\Big|~ \tilde f_{n,t_{i-1}} = \tilde w\Big).
\end{aligned} 
\end{equation}
Applying Theorem \ref{tTwoLDP}, we obtain
\begin{equation}
\label{As1LSC}
\begin{aligned}
&\lim_{\eta_i \downarrow 0} \lim_{\eta_{i-1} \downarrow 0} \liminf_{n \to \infty} \frac{1}{{n\choose 2}}
\inf_{\tilde w \in \tilde{\mathbb{B}}_{\square}(\tilde h_{i-1}, \eta_{i-1}) \cap \tilde{\mathscr{W}}_n} 
\log \mathbb{P}\Big(\tilde f_{n,t_i} \in \tilde{\mathbb{B}}_\square(\tilde h_i, \eta_i) ~\Big|~ \tilde f_{n,t_{i-1}} = \tilde w\Big)\\
&\qquad \qquad \geq - \lim_{\eta_i \downarrow 0}  \inf_{\tilde x \in \tilde{\mathbb{B}}_\square(\tilde h_i,\eta_i)} 
I_{1,t_i-t_{i-1}}(\tilde h_{i-1}, \tilde x)=-I_{1,t_i-t_{i-1}}(\tilde h_{i-1},\tilde h_i).
\end{aligned}
\end{equation}
Combining \eqref{MP1} and \eqref{As1LSC}, we get 
\begin{equation}
\lim_{\eta_{|j|} \downarrow 0} \cdots \lim_{\eta_0 \downarrow 0} \liminf_{n \to \infty} \frac{1}{{n\choose 2}} 
\log (\tilde \mu_n \circ p_j^{-1})\big(\tilde{\mathbb{B}}^{|j|}_\square(\tilde h, \boldsymbol{\eta})\big) 
\geq -\sum_{i =1}^{|j|+1} I(\tilde h_{i-1}, \tilde h_i) = -\tilde I_j(\tilde h),
\end{equation}
from which the desired lower bound follows.


\subsubsection{Upper bound}

Following similar arguments as above, we obtain
\begin{equation}
\lim_{\eta_{|j|} \downarrow 0} \cdots \lim_{\eta_0 \downarrow 0} \limsup_{n \to \infty} \frac{1}{{n \choose 2}} 
\log (\tilde \mu_n \circ p_j^{-1})\big(\mathbb{B}^{|j|}_\square(h, \boldsymbol{\eta})\big) \leq -I_j(h).
\end{equation}
To achieve this, we need a `local-to-global transference' result. But because $|j|$ is finite and $\tilde{\mathscr{W}}$ is compact, we can use the ideas in \cite[Lemma 4.1]{C}.
\qed


\subsection{Proof of the sample-path LDP}

We are now ready to prove the sample-path LDP in Theorem \ref{GSPLDP}.

\begin{proof}
Consider a path $\tilde h \in \tilde{\mathscr{W}} \times [0,T]$. Applying Theorem \ref{tMultiLDP} and the Dawson-G\"artner projective limit LDP \cite[Theorem 4.6.1]{DZ}, we obtain a sample-path LDP in the pointwise topology (for which we use the label $\mathcal{X}$) with rate ${{n\choose 2}}$ and with rate function 
\begin{equation}
\label{PWspLDP}
\tilde I_{\mathcal{X}}(\tilde h) := \sup_{0=t_0<t_1<\dots < t_k \leq T} \sum_{i = 1}^k \tilde I_{1,t_i-t_{i-1}}(\tilde h_{t_{i-1}}, \tilde h_{t_i} ),
\end{equation}
where $\tilde h \in \tilde{\mathscr{W}} \times [0,T]$.  There are two major challenges we need to overcome to establish Theorem~\ref{GSPLDP}: 
\begin{itemize}
\item[(I)] 
Prove exponential tightness to strengthen the topology (Lemma \ref{LExT} below).
\item[(II)] 
Show that $\tilde I_{\mathcal{X}}(\tilde h) = \tilde I(\tilde h)$ (Lemma \ref{RateEq} below).
\end{itemize}

\medskip\noindent
(I) Recall the definition of $D$ in \eqref{Ddef}. We say that the sequence of probability measures $(\mu_n)_{n \in \mathbb{N}}$ on $D$ is exponentially tight when for every $\alpha < \infty$ there exists a compact set $K_\alpha \subseteq D$ such that 
\begin{equation}
\limsup_{n \to \infty} \frac{1}{{{n\choose 2}}} \log \mu_n(K_{\alpha}^c) < -\alpha.
\end{equation}

\begin{lemma}
\label{LExT}
$(\tilde \mu_n)_{n \in \mathbb{N}}$ is exponentially tight.
\end{lemma}

\begin{proof}
By \cite[Theorem 4.1]{FK} (with $\beta=1$ in the notation used in \cite{FK} and $n$ replaced by $\binom{n}{2}$), the compactness of $\tilde{\mathscr{W}}$ and the Markov property, it suffices to show that for each $\eta, \xi> 0$ there exist random variables $\gamma_n(\eta,\xi)$, satisfying
\begin{equation}
\label{TC1}
\mathbb{E}\left[ \eee^{{{n \choose 2}} \xi \delta_\square(\tilde f_{n,\eta'}, \tilde u)} \mid f_{0,n}= u_n  \right]
\leq \mathbb{E}[ \eee^{\gamma_n(\eta, \xi)}], \qquad 0 \leq \eta' \leq \eta, \,u_n \in {\mathscr{W}}_n,
\end{equation}
such that, for each $\xi>0$, 
\begin{equation}
\label{TC2}
\lim_{\eta \downarrow 0} \limsup_{n \to \infty} \frac{1}{{n \choose 2}} \log \mathbb{E}[ \eee^{\gamma_n(\eta, \xi)}]=0.
\end{equation}
To construct $\gamma_n(\eta, \xi)$, let $E$ denote the total number of edges that change (i.e., go from active to inactive or from inactive to active) somewhere in the time interval $[0,\eta]$. Then, given $f_{n,0}=u_n$, we have 
\begin{equation}
\delta_\square(\tilde f_{n,\eta'},\tilde u) \leq \int_{[0,1]^2} \ddd x\, \ddd y\,| f_{n,\eta'}(x,y)-u(x,y) | \leq E\,{n \choose 2}^{-1},
\end{equation}
where the last inequality holds for all $0 \leq \eta' \leq \eta$. Next observe that, because all edges evolve independently, for any $u_n \in \mathscr{W}_n$ the random variable $E$ given $f_{n,0}=u_n$ is stochastically dominated by $Y= {\rm Bin}({n \choose 2}, 1- \eee^{-\max \{ \lambda,\mu \} \eta})$. Thus, if we let $\gamma_n(\eta, \xi) = \xi Y$, then \eqref{TC1} holds and 
\begin{equation}
\frac{1}{{n \choose 2}} \log \mathbb{E}[\eee^{\gamma_n(\eta,\xi)} ] 
= \log\big( \eee^{-\max \{ \lambda,\mu \}\eta} + \eee^\xi  (1- \eee^{-\max \{ \lambda,\mu \}\eta} ) \big) 
\downarrow 0, \qquad \eta \downarrow 0,
\end{equation}
which finishes the proof.
\end{proof}

\medskip\noindent
(II) The following identity holds.

\begin{lemma}
\label{RateEq}
$\tilde I(\tilde h) = \tilde I_{\mathcal{X}}(\tilde h)$ for all $\tilde h \in \tilde{\mathscr{W}} \times [0,T]$.
\end{lemma}

\begin{proof} 
Recall the definition of $I(h)$ in \eqref{IDef}, and that $
\mathcal{AC}$ is the set of functions ${h}$ on $ \mathscr{W} \times [0,T]$ such that $t \mapsto h_t(x,y)$
is absolutely continuous for almost all  $(x,y) \in [0,1]^2.$

\medskip\noindent
{\bf 1.}
We first establish that 
\begin{equation}
\tilde I_{\mathcal{X}}(\tilde h) \equiv \inf_{h \sim \tilde h} \sup_{0=t_0<t_1<\dots < t_k \leq T} 
\sum_{i = 1}^k  I_{1,t_i-t_{i-1}}( h_{t_{i-1}}, h_{t_i} ) \leq \tilde I(\tilde h).
\end{equation}
It is enough to show that $I_{\mathcal{X}}(h) \leq I(h)$ for all $h \sim \tilde h$. If $h \notin \mathcal{AC}$, then $I(h)=\infty$ by definition, whereas if $h \in \mathcal{AC}$, then due the convexity of $I_{1,t}$ we obtain $I_{\mathcal{X}}\leq I(h)$ by applying Jensen's inequality.

\medskip\noindent
{\bf 2.} We now establish the reverse inequality. Similarly as above, it is enough to show that $I_{\mathcal{X}}(h) \geq I(h)$ for all $h \sim \tilde h$. 

\medskip\noindent
{\bf a.}
Suppose that $h \in \mathcal{AC}$. Then $h_t'(x,y):= \left. \frac{\partial}{\partial s} h_{s}(x,y) \right|_{s=t}$ exists for almost all $x,y \in [0,1]^2$. Letting $t_i-t_{i-1}=\Delta t$ for all $i$, we have 
\begin{equation}
\begin{aligned}
I_{\mathcal{X}}(h) 
&\geq \sum_{i = 1}^{T / \Delta t} I_{1,\Delta t}(h_{(i-1)\Delta t}, h_{i \Delta t} ) \\
&= \sum_{i = 1}^{T / \Delta t} \int_0^1 \ddd x \int_0^x \ddd y\,\\
&\qquad \sup_{v \in \mathbb{R}} \bigg[v\big( h_{(i-1)\Delta t}(x,y) + h'_{(i-1)\Delta t}(x,y)\Delta t\big)\\ 
&\qquad\qquad\qquad\qquad\qquad\qquad
- h_{(i-1)\Delta t}(x,y) \log \big[\mu \Delta t + \eee^v (1 - \mu \Delta t) \big] \\
&\qquad\qquad\qquad\qquad\qquad\qquad
- [1- h_{(i-1)\Delta t}(x,y)] \log ( 1 - \lambda + \eee^v \lambda \Delta t) + o(\Delta t) \bigg]\\
&= \sum_{i = 1}^{T / \Delta t} \int_0^1 \ddd x \int_0^x \ddd y\,
\sup_{v \in \mathbb{R}} \bigg[ v h'_{(i-1)\Delta t}(x,y) - h_{(i-1)\Delta t}(x,y) \mu(\eee^{-v} -1) \\
&\qquad\qquad\qquad\qquad\qquad\qquad
- [1-h_{(i-1)\Delta t}(x,y)] \lambda (\eee^v-1) +o(1)\bigg] \Delta t \\ 
&\to \int_0^T \ddd t \int_0^1 \ddd x \int_0^x \ddd y\,
\sup_{v \in \mathbb{R}} \bigg[ v h'_{t}(x,y) - h_{t}(x,y) \mu(\eee^{-v} -1) 
- [1-h_{t}(x,y)] \lambda (\eee^v-1) \bigg]\\ 
&=: I(h), \qquad \Delta t \downarrow 0. 
\end{aligned}
\end{equation}
This implies that if $h \in \mathcal{AC}$, then $I_{\mathcal{X}}(h) \geq I(h)$.

\medskip\noindent
{\bf b.}
It is now enough to show that if $h \sim \tilde h$ and $h \notin \mathcal{AC}$, then there exists a $\hat h \sim \tilde h$ such that 
\begin{equation}
I(\hat h)=I_{\mathcal{X}}(h)
\end{equation}
(in many cases we can take $h=\hat h$). The following argument is sketchy, but follows standard reasonings.   

Suppose that $h \notin \mathcal{AC}$. By definition this means that the set of $(x,y) \in[0,1]^2$ for which there exist $\delta_{x,y} >0$ and $\{ s^k_{1,x,y} < t^k_{1,x,y} \leq \dots \leq s^k_{\ell_k,x,y}<t^k_{\ell_k,x,y}\}$ with
\begin{equation}
\lim_{k\to\infty} \sum_{i=1}^{\ell_k} (t^k_{i,x,y}- s^k_{i,x,y})=0, \qquad 
\lim_{k \to \infty}\sum_{i=1}^{\ell_k}|h_{t^k_{{i},x,y}}(x,y)- h_{s^k_{{i},x,y}}(x,y)| \geq \delta_{x,y},
\end{equation}
has positive Lebesgue measure. We distinguish between a number of cases.

\medskip\noindent
$\bullet$ Suppose that there exist $\{ s^k_1 < t^k_1 \leq \dots \leq s^k_{\ell_k}<t^k_{\ell_k}\}$ independent of $(x,y)$ such that the set of $(x,y) \in[0,1]^2$ for which 
\begin{equation}
\label{Coord}
\lim_{k \to\infty} \sum_{i=1}^{\ell_k} (t^k_{i}- s^k_{i}) = 0, \qquad
\lim_{k \to\infty} \sum_{i=1}^{\ell_k}|h_{t^k_{i}}- h_{s^k_{\ell_i}}| \geq \delta \text{ for some } \delta>0,
\end{equation}
has positive Lebesgue measure. Then, following arguments similar to those in the proof of \cite[Lemma 5.1.6]{DZ}, we get $I_{\mathcal{X}}(h)=I(h)=\infty$. 

\medskip\noindent
$\bullet$
Suppose that no sequence satisfying \eqref{Coord} exists. Roughly speaking, the paths $t \mapsto h_t(x,y)$ that are not absolutely continuous fall into two categories: those that contain `steps' and those that contain `holes'. 

\medskip\noindent
(i) We say that a path $t \mapsto h_t(x,y)$ contains a step when there exists a $t \in [0,T]$ such that if $t_{k} \uparrow t$ and $t'_k \downarrow t$, then $h_{t_k}(x,y) \to c$ and $h_{t'_k}(x,y) \to d$ with $c \neq d$.\\
(ii) We say that a path $t \mapsto h_t(x,y) $ contains a hole when there exists a $t \in [0,T]$ such that if $t_k \to t$ (with $t_k \neq t$ for all $k$), then $h_{t_k} \to c \neq h_t$.\\ 
(If the above limits do not exists, then the arguments below can be easily adapted.) 

\medskip\noindent
$\blacktriangleright$ 
Suppose that the set of $(x,y)$ such that $t \mapsto h_t(x,y)$ contains a step of size $\gamma>0$ has positive Lebesgue measure $\beta>0$. Then
\begin{equation}
I_{\mathcal{X}}(h) \geq \sum_{i=1}^{T/\Delta t} I_{1,\Delta t}(h_{(i-1)\Delta t}, h_{i \Delta t}) 
\geq \beta \gamma \log  \frac{\gamma}{1-\eee^{-\max \{ \lambda, \mu\}\Delta t }} \to \infty, 
\qquad \Delta t \downarrow 0.
\end{equation}
where the last inequality follows from a similar reasoning as in the proof of Lemma \ref{LExT}. This implies $I_{\mathcal{X}}(h)=I(h)=\infty$.

\medskip\noindent
$\blacktriangleright$ 
Suppose that the set of $(x,y)$ such that $t \mapsto h_t(x,y)$ contains a hole has positive Lebesgue measure. 

\medskip\noindent
(i) We say that $t \mapsto h_t(x,y)$ has a hole at time $t$ if $h_{t_k}(x,y) \to c \neq h_t(x,y)$ for any $t_k \to t$ (with $t_k \neq t$ for all $k$).\\ 
(ii) We say that $t \mapsto \hat h_t(x,y)$ has this hole filled in if $\hat h_t(x,y)=c$ and $\hat h_t(x,y)=h_t(x,y)$ otherwise. 

\medskip\noindent
Construct $\hat h \in \mathscr{W} \times [0,T]$ from $h$ by filling in all the holes. Since there exists no sequence $\{s^k_{1,x,y} < t^k_{1,x,y} \leq \dots \leq s^k_{\ell_k,x,y}<t^k_{\ell_k,x,y}\}$ satisfying \eqref{Coord}, a positive Lebesgue measure of holes cannot occur simultaneously. Thus, $\hat h_t(x,y) = h_t(x,y)$ almost everywhere, which implies that $I_{\mathcal{X}}(h)= I_{\mathcal{X}}(\hat h)$. In addition, because $\mathscr{W}$ was constructed by taking the quotient with respect to almost sure equivalence, we also have $\hat h \sim \tilde h$. The fact that $I(\hat h)=I_{\mathcal{X}}(\hat h)$ now follows from the arguments above. \end{proof}

Lemmas \ref{LExT}--\ref{RateEq} in combination with Theorem \ref{tMultiLDP} complete the proof of Theorem \ref{GSPLDP}.
\end{proof}


\section{Proofs: Applications}
\label{S5}


\subsection{Application 1} 

The following lemma is the time-varying equivalent of \cite[Theorem 4.1 and Proposition 4.2]{CV} (see also \cite[Theorem 2.7]{LZ}). Let 
\begin{equation}
\label{VPeq}
\phi_T(H,r) := \inf \left\{ \tilde I_{1,T}(u,h) : h \in \mathscr{W}, t(H,h) \geq r \right\},
\end{equation}
let $F^+$ be the set of minimisers, and let $\tilde F^+$ be the image of $F^+$ in $\tilde{\mathscr{W}}$. 

\begin{lemma}
\label{VarProb}
Fix a constant graphon $u$. Let $H$ be a $d$-regular graph for some $d \in \mathbb{N}\setminus\{1\}$. Suppose that $\lim_{n\to\infty} \delta_{\square}(\tilde u_n, \tilde u) = 0$ and $u\,\pbb[T] + (1-u)\pab[T]<r<1$. Then
\begin{equation}
\lim_{n \to \infty} \frac{1}{{n \choose 2}} \log \mathbb{P}\big(t(H,\tilde f_{n,T}) \geq r \mid \tilde f_{n,0} = \tilde u_n\big) 
= \phi_T(H,r).
\end{equation}
Moreover, $\tilde F^+$ is non-empty and compact and, for each $\varepsilon >0$, there exists a positive constant $C(H,r,\lambda, \mu, T, u, \varepsilon)$ such that
\begin{equation}
\mathbb{P}\Big(\delta_\square(\tilde f_{n,T}, \tilde F^+) \geq \varepsilon ~\Big|~ 
t(H,\tilde f) \geq r, \tilde f_{n,0}=\tilde u_n\Big) \leq \eee^{-Cn^2}, \qquad n\in\mathbb{N}.
\end{equation}
\end{lemma}

\begin{proof}
Observe that $x \mapsto I_{1,T}(u,x)$ is uniformly continuous (by Lemma~\ref{UnifCont}), $h \mapsto I_{1,T}(u, h)$ is constant under measure-preserving bijections when $u$ is a constant graphon, and $h \mapsto I_{1,T}(u,h) =\tilde I_{1,T}(\tilde u, \tilde h)$ is lower semi-continuous (by Lemma \ref{LwrSC}). Therefore the claim follows via the same line of argument used to prove \cite[Theorems 6.1--6.2]{C}, where we apply Theorem \ref{tMultiLDP} in place of \cite[Theorem 5.2]{C}.
\end{proof}

The proof of Theorem \ref{LZlemtv} borrows various elements from that of \cite[Theorem 1.1]{LZ}. 

\medskip\noindent
\textbf{Proof of Theorem \ref{LZlemtv}.} We distinguish between the two cases.

\medskip\noindent 
(i) Suppose that $(r^d, I_{1,T}(u,r))$ lies on the convex minorant of $x \mapsto I_{1,T}(u,x^{1/d})$. Applying  the generalised version of H\"older's inequality derived in \cite{F}, we obtain that, for any $f \in \mathscr{W}$,
\begin{equation}
\label{Tap3}
t(H,f) = \int_{[0,1]^3} \ddd x\, \ddd y\, \ddd z\, f(x,y)f(y,z)f(z,x) 
\leq \left( \int_{[0,1]^2}  \ddd x\, \ddd y\,f^d(x,y) \right)^{e(H)/d} = \Vert f \rVert^{e(H)}_d.
\end{equation}
Abbreviate $\psi(x) := I_{1,T}(u,x^{1/d})$, and let $\hat \psi$ be the convex minorant of $\psi$. Then by Jensen's inequality we have 
\begin{align}
\label{Tap4}
\nonumber
I_{1,T}(u,f) &= \int_{[0,1]^2}\ddd x\, \ddd y\, \psi(f(x,y)^d) 
\geq \int_{[0,1]^2} \ddd x\, \ddd y\,\hat{\psi}(f(x,y)^d) \\
&\geq \hat{\psi} \left( \int_{[0,1]^2} \ddd x\, \ddd y\,f(x,y)^d \right) = \hat \psi( \lVert f \rVert_d^d).
\end{align}
Consequently, by \eqref{Tap3}, if $t(H,f) \geq r^{e(H)}$, then $\lVert f \rVert_d^d \geq r^d$, while, by \eqref{Tap4}, if $\lVert f \rVert_d^d \geq r^d$, then $I_{1,T}(u,f) \geq \hat{\psi}(r^d)=I_{1,T}(u,r)$, where the last equality follows from the condition that $(r^d,I_{1,T}(u,r))$ lies on the convex minorant of $\psi$. Because $\psi$ is strictly increasing on the interval $(u \,p_{1 1,T}+(1-u)\,p_{01,T},1]$ and $\hat \psi$ is not linear in any neighbourhood of $r^d$, equality can occur if and only if $f \equiv r$. Thus, $\tilde F^+ = \{\tilde r \}$. The claim now follows from Lemma \ref{VarProb}.

\medskip\noindent
(ii) Suppose that $(r^d, I_{1,T}(u,r))$ does not lie on the convex minorant of $x \mapsto I_{1,T}(u,x^{1/d})$. Then there necessarily exist $0 \leq r_1 < r < r_2 \leq 1$ such that the point $(r^2, I_{1,T}(u,r))$ lies strictly above the line segment joining $(r_1^d, I_{1,t}(u, r_1))$ and $(r_2^d, I_{1,T}(u, r_2))$. Following the method set out in \cite[Lemma 3.4]{LZ}, we can use this fact to construct a graphon $r_\varepsilon$ with $I_{1,T}(u, r_\varepsilon) < I_{1,T}(u,r)$. (We refer the reader to \cite{LZ} for the specific details of this construction.) Thus, again using the fact that $\psi$ is strictly increasing on the interval $[u \,p_{11,T}+ (1-u)\,p_{0  1,T},1]$, we conclude that $\tilde F^+$ contains no constant graphons. Let $\tilde C \subseteq \mathscr{W}$ be the set of constant graphons. Since $\tilde F^+$ and $\tilde C$ are disjoint and compact, we have $\delta_\square(\tilde F^+,\tilde C)>0$. The result now follows by applying Lemma \ref{VarProb} with $\varepsilon =\tfrac12 \delta_\square(\tilde F^+, \tilde C)$. 
\qed

\medskip\noindent
\textbf{Proof of Proposition \ref{Pts}.}
We recall from the introduction that there is an explicit expression for $I_{1,T}$. Indeed, from \eqref{ee}--\eqref{ef} that 
\begin{equation}
\label{eq:Jopt}
I_{1,T}(u,r) = \sup_{v \in \mathbb{R}} [vr-J_{T,v}(u)]
\end{equation}
with
\begin{equation}
\label{App0}
J_{T,v}(u) = u\log(1-p_{11, T} + \eee^vp_{1 1,T} ) + (1-u) \log (1-p_{01,T} + \eee^vp_{01,T}). 
\end{equation}
Due to the convexity of $v \mapsto vr-J_{T,v}(u)$, we obtain the maximiser in the right-hand side of \eqref{eq:Jopt} by setting the partial derivative with respect to $v$ equal to $0$. This yields
\begin{equation}
0 = r - u \frac{p_{11,T} \eee^v}{1-p_{11,T} + \eee^v p_{11,T}} 
- (1-u) \frac{ \eee^v p_{01,T} }{1-p_{0 1,T} + \eee^v p_{01,T} },
\end{equation}
which implies that
\begin{equation}
\label{App}
\begin{aligned}
0 &=  \eee^{2v} (1-r) p_{11,T} p_{0 1,T} \\
&\quad + \eee^v \big[ (u-r) p_{1  1,T} (1-p_{0  1,T} ) + (1-u-r) p_{0  1,T} (1-p_{1  1, T}) \big] \\
&\quad -r(1-p_{1  1,T})(1-p_{0 1,T}).
\end{aligned}
\end{equation}
By the discussion that followed \eqref{ef} and the fact that $u<r$, for $T$ sufficiently small 
we then obtain $I_{1,T}(u,r)$ by substituting into \eqref{App0}
\begin{equation}
\label{QuadSol}
e^v = \frac{-b+ \sqrt{b^2-4ac}}{2a},
\end{equation}
where $a$, $b$, $c$ can be read from the three lines in \eqref{App}. It is easily verified that if $r>u$ and $T \downarrow 0$, then $a=(1-r) \lambda T + O(T^2)$, $b = u-r + O(T)$, and $c=-r \mu T + O(T^2)$, so that 
\begin{equation} 
\eee^v= \frac{u-r}{(1-r) \lambda T}+ O(1).    
\end{equation} 
Setting $x^{1/d}=r$, we therefore have  
\begin{equation}
I_{1,T}(u,x^{1/d}) = (x^{1/d}-u)\log (1/T) + O(1).
\end{equation}
Because $x \mapsto x^{1/d}$ is concave, for $T$ sufficiently small the point $r^d \mapsto I_{1,T}(u,r)$ cannot lie on the convex minorant of $x \mapsto I_{1,T}(u, x^{1/d})$. The fact that, for such $T$, $G_n(T)$ is SB now follows from Theorem \ref{LZlemtv}.
\qed

\medskip\noindent
\textbf{Proof of Proposition \ref{Pht}.}
If $u=0$, then 
\begin{equation}
I_{1,T}(0,x^{1/2}) = x^{1/2} \log \frac{x^{1/2}}{p_{01,T}} + (1-x^{1/2}) \log \frac{1-x^{1/2}}{1-p_{01, T}}.
\end{equation}
Calling $F(x)$ all terms that do not depend on $T$, we compute
\begin{equation}
\begin{aligned}
I_{1,T}(0,x^{1/2}) 
&= -x^{1/2} \log p_{0 1,T} - (1-x^{1/2}) \log (1- p_{0 1,T}) + F(x), \\
\frac{\partial I_{1,T}(0,x^{1/2})}{\partial x} 
&= -\tfrac{1}{2} x^{-1/2} \log p_{01,T} + \tfrac{1}{2} x^{-1/2}\log(1- p_{0 1,T}) + F'(x), \\
\frac{\partial^2 I_{1,T}(0,x^{1/2})}{\partial x^2} 
&=\tfrac{1}{4} x^{-3/2} \log p_{01,T} - \tfrac{1}{4} x^{-3/2}\log(1- p_{0 1,T}) + F''(x).
\label{eq:2ndder}
\end{aligned}
\end{equation}
The last line \eqref{eq:2ndder} is an increasing function of $\pab[T]$, which itself is an increasing function of $T$ (recall \eqref{ptdefs}). Thus, $r^d \mapsto I_{1,T}(0,r)$ lies on the convex minorant of $x \mapsto I_{1,T}(0, x^{1/d})$, then the same is true for all $T'>T$. Theorem~\ref{LZlemtv} now yields the desired result, i.e., when $G_n(T)$ is SB also $G_n(T')$ is SB. The same argument applies when $u=1$, but in this case $\pbb[T]$ is a decreasing function of $T$.
\qed


\subsection{Application 2} 

$\mbox{}$

\medskip \noindent
\textbf{Proof of Theorem \ref{SPCond}.}
Recall that $\tilde H$ is the set of all paths in $\tilde{\mathscr{W}} \times [0,T]$ that start at $\tilde u$ and end at $\tilde r$. Since $
\tilde H$ is not compact, we first demonstrate that we can restrict our search for elements of $\tilde H^*$ to a compact set. To do this, we note that, by Lemma \ref{WTpath}, for any $u,r$ in the equivalence classes $\tilde u, \tilde r$, 
\begin{equation}
\begin{aligned}
I(h^*_{u \to r}) &= I_{1,T}(u,r) \leq \max \{ I_{1,T}(0,1), I_{1,T}(1,0) \} \\
&= \max \left\{ \log \frac{1}{p_{01,T}}, \log \frac{1}{1-p_{11,T}} \right\} =: K < \infty.
\end{aligned}
\end{equation}
Thus, no paths with rate strictly greater than $K$ can be an element of $\tilde H^*$.
Let 
\begin{equation}
\tilde H^{(K)}_\eta := \left\{ \tilde h \in \tilde{\mathscr{W}}\times [0,T] : \tilde I(\tilde h) 
\leq K, \; \tilde h_0=\tilde u, \; \delta_\square(\tilde h(T), \tilde r) \leq \eta \right\}.
\end{equation}
The fact that $\tilde H_\eta^{(K)}$ is compact follows from a similar line of reasoning as the one used in the proof of Lemma \ref{LExT}. Since $\tilde H^* \subseteq \tilde H^{(K)}_0$, $\tilde H^{(K)}_0$ is compact, and $\tilde I$ is lower semi-continuous, $\tilde I$ must attain its minimum on $\tilde H^{(K)}_0$. Thus, $\tilde H^*$ is non-empty. By the lower semi-continuity of $\tilde I$, $\tilde H^*$ is also closed (and hence compact). 

Fix $\varepsilon >0$ and let 
\begin{equation}
\tilde H^\geq_{\eta,\varepsilon} := \{ \tilde h \in \tilde H^{(K)}_\eta \colon\,
\delta^\infty_\square(\tilde h, \tilde H^*) \geq \varepsilon \}. 
\end{equation}
Then, for the same reasons as above, $\tilde H^\geq_{\eta,\varepsilon}$ is compact for all $\eta,\varepsilon \geq 0$. Define
\begin{equation}
I_1 :=\inf_{\tilde h \in \tilde H} I(\tilde h), \qquad I_2 := \inf_{\tilde h \in \tilde H^\geq_{0,\varepsilon}} I(\tilde h),
\end{equation}
By Theorem \ref{GSPLDP}, we have 
\begin{equation}
\begin{aligned}
\lim_{\eta \downarrow 0} \limsup_{n \to \infty} &\frac{1}{{n \choose 2}} \log
\mathbb{P}^{u_n} \Big( \delta^\infty_\square( \tilde f_n, \tilde H^*) 
\geq \varepsilon ~\Big|~ \tilde f_{n,t} \in \tilde H^{(K)}_{\eta} \Big) \\
&\leq \lim_{\eta \downarrow 0} \left[\inf_{\tilde h \in \tilde H^{(K)}_{\eta} }\tilde I(\tilde h) - \inf_{\tilde h \in \tilde H^\geq_{\eta,\varepsilon}} I(\tilde h) 
\right] = I_1 - I_2.
\end{aligned}
\end{equation}
The proof is complete once we are able show that $I_1 < I_2$. Now, clearly, $I_1 \leq I_2$. If $I_1 =I_2$, then the compactness of $\tilde H^\geq_{\eta,\varepsilon}$ implies that there exists a $\tilde h \in \tilde H^\geq_{\eta,\varepsilon}$ satisfying $\tilde I( \tilde h)=I_2$. However, this means that $\tilde h \in \tilde{H}^*$ and hence $\tilde H^\geq_{0,\varepsilon} \cap \tilde H^* \neq \emptyset$, which is a contradiction.
\qed

\medskip\noindent
\textbf{Proof of Lemma \ref{SPQOpt}.}
We first demonstrate that $f^*_{u \to r} \in [0,1] \times [0,T]$ is in the set of minimisers of $I$. By the contraction principle, it is enough to show that $I(f^*_{u \to r}) =I_{1,T}(u,r)$. Applying Theorem~\ref{tMultiLDP} and the contraction principle, we have 
\begin{equation}
I_{1,T}(u,r) = \min_{s \in [0,1]} \left[ I_{1,t}(u,s) - I_{1,T-t}(s,r) \right] 
= I_{1,t}(u,f^*_{u \to r}(t)) + I_{1,T-t}(f^*_{u \to r}(t),r).
\end{equation}
Consequently, 
\begin{equation}
I_{1,T}(u,r) = \sup_{k\in\mathbb{N}}\sup_{0=t_0 <t_1 \dots < t_k=T} 
\sum_{i=1}^k I_{1,t_{i}-t_{i-1}}\big(f^*_{u \to r}(t_{i-1}),f^*_{u \to r}(t_{i})\big)
= I(f^*_{u \to r}). 
\label{POQ}
\end{equation}
We next prove that $f^*_{u \to r}$ is the unique minimiser of $I(f)$ conditional on $f(0)=u$ and $f(T)=r$. Because $I(f^*_{u \to r})=I_{1,T}(u,r)$, it suffices to establish that, for any $u,r \in [0,1]$, $t <T$ and $s \neq f^*_{u \to r}(t)$,
\begin{equation}
\label{UniPQ}
I_{1,t}(u,s) + I_{1, T-t}(s,r) >I_{1,T}(u,r).
\end{equation}
To establish \eqref{UniPQ}, note that from the definition of $f^*_{u \to r}(t)$ and the fact that both $x \mapsto I_{1,t}(u,x)$ and $x \mapsto I_{1,T-t}(x,r)$ are continuously differentiable we get
\begin{equation}
\left. \frac{\partial I_{1,t}(u, x)}{\partial x} \right|_{x=f^*_{u \to r}(t)} 
= - \left. \frac{\partial I_{1,T-t}(x,r)}{\partial x} \right|_{x=f^*_{u \to r}(t)}.
\end{equation}
Combine this with the fact that, for any $x \in [0,1]$, $u \mapsto I_{1,T}(u,x)$ is convex and $r \mapsto I_{1,T}(x,r)$ is strictly convex (see Lemma \ref{Convexity}), to get
\begin{equation}
\begin{aligned}
I_{1,t}(u,s)+I_{1,T-t}(s,r) >\:& I_{1,t}(u,f^*_{u \to r}(t)) + (s-f_{u \to r}^*(t)) \left. 
\frac{\partial I_{1,t}(u, x)}{\partial x} \right|_{x=f^*_{u \to r}(t)}\\ 
&+I_{1,T-t}(f^*_{u \to r}(t),r) + (s-f_{u \to r}^*(t)) \left. \frac{\partial I_{1,T-t}(x,r)}{\partial x} \right|_{x=f^*_{u \to r}(t)}\\
&= I_{1,t}(u,f^*_{u \to r}(t))  + I_{1,T-t}(f^*_{u \to r}(t),r) = I_{1,T}(u,r),  
\label{PathQuot2}
\end{aligned}
\end{equation}
from which we conclude that indeed $f^*_{u \to r}$ is the unique minimiser of $I(f)$.
\qed

\medskip
The next lemma, whose proof is standard and is omitted, can be used to compute $f^*_{u \to r}$ (see \cite{M} for a related method). Let  $\tau_t(u,r)=v$ be the unique solution of 
\begin{equation}
0= r - u \frac{p_{11,t} e^v}{1-p_{11,t} + \eee^v p_{11,t}} 
- (1-u) \frac{\eee^v p_{01,t} }{1-p_{01,t} + \eee^v p_{01,t} },
\end{equation}
which amounts to solving a quadratic equation (recall \eqref{QuadSol}). 

\begin{lemma}
\label{MM1C}
$f^*_{u \to r}(t)=s$ is the unique solution of
\begin{equation}
\label{fscomp}
\eee^{\tau_t(u,s)}\big(1-p_{01, T-t} + \eee^{\tau_{T-t}(s,r)}p_{01, T-t}\big) = 1- p_{1 1,T-t} + \eee^{\tau(s,r)}p_{11,T-t}.
\end{equation}
\end{lemma}

\medskip\noindent
\textbf{Proof of Lemma \ref{WTpath}.}
Let 
\begin{equation}
(h^{*}_{u \to r})'(x,y,t)=\frac{\partial h^{*}_{u \to r}(x,y,t)}{\partial t}.
\end{equation}
We have
\begin{equation}
\begin{aligned}
&\int_{0}^T {\rm d}t \int_{[0,1]^2} \ddd x\, \ddd y \,\mathcal{L}\big(h^*_{u \to r}(x,y,t), (h^{*}_{u \to r})'(x,y,t)\big)\\
&=\int_{[0,1]^2} \ddd x\, \ddd y \int_{0}^T {\rm d}t\, \mathcal{L}\big(h^*_{u \to r}(x,y,t), (h^{*}_{u \to r})'(x,y,t)\big) 
= I_{1,T}(u,r), 
\label{WTOp}
\end{aligned}
\end{equation}
where we apply Lemma \ref{SPQOpt} to obtain the last equality. By the contraction principle, $h^*_{u \to r}$ is in the set of minimisers of $I(h)$ subject to the conditions $h_0=u$ and $h_T=r$. To show that $h^*_{u \to r}$ is the unique minimiser, note that if $h_0=u$, $h_T=r$ and $h \neq h^*_{u \to r}$, then, by \eqref{PathQuot} and Lemma \ref{SPQOpt}, there exist $\varepsilon, \beta >0$ such that
\begin{equation}
{\rm Leb} \left\{ (x,y) \in [0,1]^2\colon\, I(h(x,y)) \geq I_{1,T}(u(x,y), r(x,y))+\varepsilon \right\}> \beta.
\end{equation}
Consequently, $I(h) \geq I(h^*_{u \to r}) +  \varepsilon \beta$, which implies that indeed $h^*_{u \to r}$ is unique. 
\qed

\medskip\noindent
\textbf{Proof of Theorem \ref{Ap2Thm}.}
Suppose that $u \in \mathscr{W}^{(I)}$ and $r \in \mathscr{W}^{(J)}$ for some $I,J \in \mathbb{N}$. Let $0=a_0<a_1<\dots < a_I=1$ and $0=b_0<b_1<\dots < b_J=1$ be their block end points, and $A_i:=(a_{i-1}, a_i]$ and $B_i:=(b_{i-1},b_i]$ their block intervals. Recall the function $\alpha\colon\,\mathscr{M} \mapsto [0,1]^{I\times J}$ defined in \eqref{Aldeff} and the (compact) set $V$ defined in \eqref{Vdeff}. For any $\sigma, \phi \in \mathscr{M}$ with $\alpha(\sigma)=\alpha(\phi)$ we have $I_{1,T}(u,r^\phi) = I_{1,T}(u,r^\sigma)$, which implies
\begin{equation}
\label{CMPat}
\inf_{\sigma \in \mathscr{M}} I_{1,t}(u,r^\sigma) = \min_{{\boldsymbol v} \in V} I_{1,T}(u, r^{\sigma_{{\boldsymbol v}}}),
\end{equation}
where $\sigma_{\boldsymbol v}$ is any element of $\mathscr{M}$ with $\alpha(\sigma_{\boldsymbol v})={\boldsymbol v} \in V$. Because $\boldsymbol v \mapsto I_{1,T}(u, r^{\sigma_{\boldsymbol v}})$ is continuous and $V$ is compact, the set of minimisers $V^* \subseteq V$ of \eqref{CMPat} is also compact. Suppose that $\tilde h \in \tilde H^*$. Then, by Lemma \ref{WTpath} and the compactness of $V^*$, there exist ${\boldsymbol v}^* \in V^*$ and a sequence $(h^{[i]})_{i\in\mathbb{N}}$ of representatives of $\tilde h$ such that
\begin{equation}
\label{Hopt} 
\lim_{i \to \infty} I(h^{[i]}) = I_{1,T}(u, r^{\sigma_{{\boldsymbol v}^*}}), \qquad  
\lim_{i \to \infty} \alpha(\sigma^{[i]}) ={\boldsymbol v}^*,
\end{equation}
where $\sigma^{[i]}$ is any element of $\mathscr{M}$ with $h^{[i]}_{T}=r^{\sigma^{[i]}}$. 

Suppose that $\tilde h \notin \tilde F^*$. Due to the compactness of $V$, also $\tilde F^*$ is compact. Thus, there exist $\varepsilon>0$ and $t \in (0,T)$ such that
\begin{equation}
\label{eq:Lebb}
\delta_{\square}\left(\tilde h_t, \tilde h^*_{u \to r^{\sigma_{{\boldsymbol v}^*}}}(\cdot,\cdot,t)\right)\geq \varepsilon.
\end{equation}
By \eqref{Hopt}, 
\begin{equation}
\lim_{i\to \infty}\delta^\infty_\square(\tilde h^*_{h^{[i]}_0\to h^{[i]}_T},\tilde h^*_{u \to r^{\sigma_{\boldsymbol v^*}}}) \to 0,
\end{equation}
which implies that, for any $0<\varepsilon'<\varepsilon$, there exist $i^*$ such that if $i>i^*$ then 
\begin{equation}\label{eq:Lebbalt}
\delta_{\square}\left(\tilde h_t, \tilde h^*_{h^{[i]}_0\to h^{[i]}_T}(\cdot,\cdot,t)\right)\geq \varepsilon'.
\end{equation}
Now, using the fact that the cut distance is bounded above by the $L^1$ distance, for $i>i^*$ we obtain 
\begin{equation}
\label{eq:Leb}
{\rm Leb}\left\{ (x,y) \in [0,1]^2\colon\, \left| h^{[i]}_t(x,y) - f^*_{h^{[i]}_{0}(x,y) 
\to h^{[i]}_{T}(x,y)}(t) \right| > \varepsilon'  \right\} \geq \varepsilon'.
\end{equation} 
Let $u_{ij}$ and $r_{\ell m}$ denote the values of the block constants (i.e., if $(x,y) \in A_i \times A_j$, then $u(x,y)=u_{ij}$, and likewise for $r_{\ell m}$). Let 
\begin{equation}
\begin{aligned}
\kappa_t(\varepsilon) &:= \min_{i,j\in[ I],  \ell,m \in[ J], \, \beta: |\beta| \geq \varepsilon } 
\Big[I_{1,t}\big(u_{ij},f^*_{u_{ij} \to r_{\ell m}}(t)+\beta\big)+ I_{1,T-t}\big(f^*_{u_{ij} \to r_{\ell m}}(t)+\beta, r_{\ell m}\big) \\ 
&\qquad\qquad\qquad\qquad\qquad\qquad\qquad -I_{1,T}( u_{ij}, r_{\ell m}) \Big] > 0,\label{eq:kappa} 
\end{aligned}
\end{equation}
where the last inequality follows from Lemma \ref{SPQOpt}.  By the contraction principle, we know that 
\begin{equation}
I(h^{[i]}) \geq I_{1,t}\big(u, h^{[i]}_t\big) + I_{1,T-t}\big(h^{[i]}_t, h^{[i]}_T\big).
\end{equation}  
Now suppose $i>i^*$ and evaluate the integral 
\begin{equation}
\label{eq:LEB}
\int_{[0,1]^2} \ddd x\,\ddd y\, \left[I_{1,t}\big(u(x,y), h^{[i]}_t(x,y)\big) + I_{1,T-t}\big(h^{[i]}_t(x,y), h^{[i]}_T(x,y)\big)\right]  \end{equation}
by distinguishing two sets: the $(x,y)$ identified in \eqref{eq:Leb} and the rest. By \eqref{eq:Leb}, the first set has Lebesgue measure at least $\varepsilon'$, and by \eqref{eq:kappa} the integrand has value at least 
\begin{equation}
I_{1,T}\big(u(x,y), h^{[i]}_T(x,y)\big)+\kappa_t(\varepsilon'),
\end{equation}
while for the second set, the integrand has value at least 
\begin{equation}
I_{1,T}\big(u(x,y), h^{[i]}_T(x,y)\big).
\end{equation}
Combining the above observations, we find that the integral in \eqref{eq:LEB} satisfies the lower bound
\begin{equation}
I_{1,t}\big(u, h^{[i]}_t\big) + I_{1,T-t}\big(h^{[i]}_t, h^{[i]}_T\big) \geq I_{1,T}(u, r^{\sigma_{\boldsymbol v^*}})
+ \varepsilon' \kappa_t(\varepsilon').
\end{equation}
This is in contradiction with \eqref{Hopt}. Thus, we have established \eqref{HtSFt}. A similar argument yields \eqref{HtSHst}. 
\qed

\medskip\noindent
\textbf{Proof of Proposition \ref{TwoBlocks}.}
First, note that if $a<b$ and $c<d$, then
\begin{equation}
\label{HUC2}
I_{1,T}(a,c) + I_{1,T}(b,d) < I_{1,T}(a,d)+I_{1,T}(b,c).
\end{equation}
Combining \eqref{HUC} and \eqref{HUC2} with the observation that $I_{1,T}(u^\sigma, r^\phi)=I_{1,T}(u, r^{\phi\, \circ \,\sigma^{-1}})$, we see that $r^{\phi \,\circ\, \sigma^{-1}}$ is the unique representative of $\tilde r$ that minimises $I_{1,T}(u, \cdot)$. Equivalently, if $u \in \mathscr{W}^{(I)}$ and $r \in \mathscr{W}^{(J)}$ for some $I,J \in \mathbb{N}$, then $\boldsymbol v^* = \alpha(\phi \circ \sigma^{-1})$ is the unique element of $V$ that minimises $\boldsymbol v \mapsto I_{1,T}(u, r^{\sigma_{\boldsymbol v}})$, and the claim follows directly from Theorem \ref{Ap2Thm}. If not, then the same follows after we consider a sequence of block approximations.
\qed

\medskip\noindent
\textbf{Proof of Proposition \ref{ThreeBlocks}.}
Suppose that $u$, $r$, $\sigma$ are given by \eqref{PathExu}, \eqref{PathExr} and \eqref{PathExs}. Let $\boldsymbol v_1 = \alpha({\tt id})$ and $\boldsymbol v_2 = \alpha(\sigma)$, where ${\tt id} \in \mathscr{M}$ is the identity, i.e., $r=r^{\tt id}$. By Theorem~\ref{Ap2Thm}, the claim has been proven once we we have shown that: (i) $\boldsymbol v_1$ and $\boldsymbol v_2$ are the unique minimisers of $\boldsymbol v \mapsto I_{1,T}(u,r^{\sigma_{\boldsymbol v}})$ (recall that $\sigma_{\boldsymbol v}$ is any $\sigma \in \mathscr{M}$ with $\alpha(\sigma)=\boldsymbol v$); (ii) $\tilde h^*_{u \to r^{\sigma_{\boldsymbol v_1}},T} \neq \tilde h^*_{u \to r^{\sigma_{\boldsymbol v_2}},T}$.

\medskip\noindent
(i) Let $a=c=0$, $b=d=\varepsilon$ and $u_{11}=u_{23}=r_{11}=r_{23}=1$. Suppose $\varepsilon>0$ is small but fixed and $T << \varepsilon$. Following the same arguments as in the proof of Proposition \ref{Pts}, we find that, for $T \downarrow 0$,
\begin{equation}
\label{TSlim}
I_{1,T}(a',b')=|a'-b'| \log(1/T) + O(1), \qquad a',b' \in [0,1].
\end{equation}
For $a',b' \in [0,1]$, we then have 
\begin{align}
\begin{split}
&I_{1,T}(a',a') \sim I_{1,T}(b',b'), \qquad I_{1,T}(0,\varepsilon)\sim I_{1,T}(\varepsilon,0), \\
&I_{1,T}(\varepsilon,1)\sim I_{1,T}(1,\varepsilon), \qquad I_{1,T}(0,1)\sim I_{1,T}(1,0),
\end{split}
\end{align}
and 
\begin{equation}
\label{Eq:terms}
I_{1,T}(a',a') << I_{1,T}(0,\varepsilon) << I_{1,T}(\varepsilon,1) < I_{1,T}(0,1).
\end{equation}
Observe that
\begin{equation}
\label{fEF2}
 I_{1,T}(u, r^{\sigma_{\boldsymbol v_1}}) = I_{1,T}(u, r^{\sigma_{\boldsymbol v_2}}) 
 =  \tfrac{4}{25}I_{1,T}(0,\varepsilon) + \tfrac{4}{25}I_{1,T}(\varepsilon,0) + o(1),
\end{equation}
where here we apply $\lim_{T \to 0}I_{1,T}(a',a')=0$ which reflects the fact that edges are increasingly unlikely to change (i.e., go from active to inactive or vice-versa) between times $0$ and $T$ as $T\to 0$. Note that above expression for $I_{1,T}(u, r^{\sigma_{\boldsymbol v_1}})$ and $I_{1,T}(u, r^{\sigma_{\boldsymbol v_2}})$ contains none of the much larger terms $I_{1,T}(\varepsilon,1)$, $I_{1,T}(1,\varepsilon)$, $I_{1,T}(0,1)$, $I_{1,T}(1,0)$ listed in \eqref{Eq:terms}. Further note that $\boldsymbol v_1$ and $\boldsymbol v_2$ are the only elements of $V$ whose corresponding rate function does not incur any of these much larger terms. Below we use this fact to establish \emph{(i)}. In particular, we first show that if $\boldsymbol v$ is a minimiser of $\boldsymbol v \mapsto I_{1,T}(u, r^{\sigma_{\boldsymbol v}})$, then it must be close to either $\boldsymbol v_1$ or $\boldsymbol v_2$. Afterwards we show that if $\boldsymbol v$ is close to $\boldsymbol v_1$ ($\boldsymbol v_2$), then by moving still closer to $\boldsymbol v_1$ ($\boldsymbol v_2$) we strictly decrease the rate.

Recall that $V^*$ denotes the set of $\boldsymbol v \in V$ that minimise $\boldsymbol v \mapsto I_{1,T}(u,r^{\sigma_{\boldsymbol v}})$. Suppose that $\boldsymbol v \in V^*$. Let $z_1=\tfrac15-v_{11}$. Observing that $v_{12}+v_{13}=z_1$, $v_{22}+v_{32}\geq \frac15$ and $v_{23}+v_{33}\geq \tfrac15$, we obtain
\begin{equation}
\begin{aligned}
I_{1,T}(u,r^{\sigma_{\boldsymbol v}}) 
&> \int_{A_{2}\times A_{3} \cup A_{3}\times A_{2}} \ddd x\, \ddd y\, I_{1,T}\big(u(x,y),r^{\sigma_{\boldsymbol v}}(x,y)\big) 
\geq \tfrac{2 z_1}{5} I_{1,T}(1,\varepsilon),
\end{aligned}
\end{equation}
which, by \eqref{TSlim}, \eqref{fEF2} and the fact that $\boldsymbol v \in V^*$, implies that if $\varepsilon \leq 1/5$ then for $T$ sufficiently small,
\begin{equation}
\label{z1ineq}
z_1 \leq \frac{4 \varepsilon}{5 (1-\varepsilon)} \leq \varepsilon.
\end{equation}
Now suppose that $v_{11}\leq \varepsilon$ and $v_{22}>v_{32}$. Let $z_2= \frac{2}{5}-v_{22}$. Observing that $v_{33} > \frac{1}{5}-\varepsilon$, we have
\begin{equation} 
I_{1,T}\big(u,r^{\sigma_{\boldsymbol v}}\big) > \int_{A_{2}\times A_{3} \cup A_3 \times A_2} \ddd x\, \ddd y\, I_{1,T}\big(u(x,y),r^{\sigma_{\boldsymbol v}}(x,y)\big) \geq 2 z_2 \big(\tfrac15-\tfrac12\varepsilon\big) I_{1,T}(1,\varepsilon),
\end{equation}
which, by the same reasoning as above, implies that if $\varepsilon<1/40$ then for $T$ sufficiently small
\begin{equation}
z_2 \leq \frac{4\varepsilon}{5(1-5\varepsilon)(1-\varepsilon)} \leq \varepsilon.
\end{equation}
Similarly, if $v_{11}\leq \varepsilon$ and $v_{22}>v_{32}$ and $z_3= \frac{2}{5}-v_{33}$, then when $\varepsilon<1/40$ for $T$ sufficiently small $z_3  \leq \varepsilon$. Thus, we have shown that if $\boldsymbol v \in V^*$, $v_{22}>v_{32}$ and $\varepsilon < 1/40$, then 
\begin{equation}
\lVert \boldsymbol v - \boldsymbol v_1 \rVert_\infty \leq \varepsilon
\end{equation}
for $T$ sufficiently small. The same arguments can be used to show that if $\boldsymbol v \in V^*$, $v_{22}<v_{32}$, and $\varepsilon < 1/40$, then $\lVert \boldsymbol v - \boldsymbol v_2 \rVert_\infty \leq \varepsilon$ for $T$ sufficiently small. It can also be easily shown that if $v_{22}=v_{32}$, then $\boldsymbol v$ cannot be a minimiser of $\boldsymbol v \mapsto I_{1,T}(u,r^{\sigma_{\boldsymbol v}})$. Thus, we have shown that $\boldsymbol v$ must be close to either $\boldsymbol v_1$ or $\boldsymbol v_2$.

Next suppose that $\lVert \boldsymbol v - \boldsymbol v_1 \rVert_\infty \leq \varepsilon$. We show that by moving still closer to $\boldsymbol v_1$ we strictly decrease the value of the associated rate. To that end, first suppose that $v_{11} < \tfrac15$.  Then, because $v_{11}+v_{12}+v_{13}=\tfrac15$, either $v_{12}>0$ or $v_{13}>0$. In addition, because $v_{11}+v_{21}+v_{31}=\tfrac15$, either $v_{21}>0$ or $v_{31}>0$. Suppose, for instance, that $v_{12}>0$ and $v_{31}>0$, and let $\eta \leq \min \{ v_{12},v_{31}\}$. Define $\boldsymbol v'$ as
\begin{equation}
v'_{11}=v_{11}+\eta, \qquad v'_{32}=v_{32}+\eta, \qquad v'_{12}=v_{12}-\eta, \qquad v'_{31}=v_{31}-\eta,
\end{equation}
and $v'_{ij} = v_{ij}$ otherwise. Let 
\begin{equation}
D(\boldsymbol v, \boldsymbol v', x,y) := I_{1,T}\big(u(x,y), r^{\sigma_{\boldsymbol v}}(x,y)\big)
-I_{1,T}\big(u(x,y), r^{\sigma_{\boldsymbol v'}}(x,y)\big).
\end{equation}
From elementary considerations, we obtain
\begin{equation}
\begin{aligned}
\int_{A_1^2} \ddd x\,\ddd y\, D(\boldsymbol v, \boldsymbol v', x,y) 
&= \tfrac{2\eta \log(1/T) }{5} + O(\eta + \eta \varepsilon \log(1/T)), \\
\int_{A_1 \times A_2 \cup A_2 \times A_1} \ddd x\,\ddd y\, D(\boldsymbol v, \boldsymbol v', x,y) 
&= \tfrac{4\eta \log(1/T) }{5} +  O(\eta + \eta \varepsilon \log(1/T)), \\
\int_{A_1 \times A_3 \cup A_3 \times A_1} \ddd x\,\ddd y\, D(\boldsymbol v, \boldsymbol v', x,y) 
&= O(\eta + \eta \varepsilon \log(1/T)),\\
\int_{A_2^2} \ddd x\,\ddd y\, D(\boldsymbol v, \boldsymbol v', x,y) 
&=-\tfrac{4\eta \log(1/T) }{5} +  O(\eta + \eta \varepsilon \log(1/T)),\\
\int_{A_2 \times A_3 \cup A_3 \times A_2} \ddd x\,\ddd y\, D(\boldsymbol v, \boldsymbol v', x,y) 
&=O(\eta + \eta \varepsilon \log(1/T)),\\
\int_{A_3^2} \ddd x\,\ddd y\, D(\boldsymbol v, \boldsymbol v', x,y) 
&=0.
\end{aligned}
\end{equation}
Consequently, it is possible to choose $\varepsilon$ and $T$ sufficiently small to ensure that $I_{1,T}(u,r^{\sigma_{\boldsymbol{v}'}})<I_{1,T}(u,r^{\sigma_{\boldsymbol{v}}})$. Moreover, it is possible to select $\varepsilon$ and $T$ such that this inequality holds for any $\boldsymbol v \in V$ such that $\lVert \boldsymbol v - \boldsymbol v_1 \rVert_\infty \leq \varepsilon$, $v_{11}<1/5$, $v_{12}>0$, and $v_{31}>0$. Following the same arguments, we see that the same is true when $v_{11}<1/5$, $v_{12}>0$, $v_{21}>0$, $\eta \leq \min\{v_{12}, v_{21}\}$ and $\boldsymbol {v}'$ is defined as 
\begin{equation}
v'_{11}=v_{11}+\eta, \qquad v'_{22}=v_{22}+\eta, \qquad v'_{12}=v_{12}-\eta, \qquad v'_{21}=v_{21}-\eta
\end{equation}
or when $v_{11}<1/5$, $v_{13}>0$, $v_{31}>0$, $\eta \leq \min\{v_{13}, v_{31} \}$ and $\boldsymbol{v}'$ is defined as
\begin{equation}
v'_{11}=v_{11}+\eta, \qquad v'_{33}=v_{33}+\eta, \qquad v'_{13}=v_{13}-\eta, \qquad v'_{31}=v_{31}-\eta
\end{equation}
or when $v_{11}<1/5$, $v_{13}>0$, $v_{21}>0$, $\eta \leq \min\{v_{13}, v_{21} \}$, and $\boldsymbol {v}'$ is defined as 
\begin{equation}
v'_{11}=v_{11}+\eta, \qquad v'_{23}=v_{23}+\eta, \qquad v'_{13}=v_{13}-\eta, \qquad v'_{31}=v_{31}-\eta.
\end{equation}
Note that for each of these cases, after replacing $\boldsymbol v$ by $\boldsymbol v'$ we decrease $z_1=\tfrac15-v_{11}$ and strictly decrease the associated rate. We can hence conclude that if $\varepsilon$ and $T$ are sufficiently small, $\lVert \boldsymbol v - \boldsymbol v_1 \rVert_\infty \leq \varepsilon$ and $\boldsymbol v \in V^*$, then $v_{11}=\tfrac15$. Using similar arguments, we can verify that if $\lVert \boldsymbol v - \boldsymbol v_1 \rVert_\infty \leq \varepsilon$, $v_{11}=\tfrac15$ and $v_{22}<\tfrac25$, then $I_{1,T}(u,r^{\sigma_{\boldsymbol v_2}})<I_{1,T}(u,r^{\sigma_{\boldsymbol v}})$ for $\varepsilon$ and $T$ sufficiently small. Thus, if $\boldsymbol v \in V^*$ and $\lVert \boldsymbol v - \boldsymbol v_1 \rVert_\infty \leq \varepsilon$, then $\boldsymbol v=\boldsymbol v_1$ when $\varepsilon$ and $T$ are sufficiently small. Repeating these arguments when $\lVert \boldsymbol v - \boldsymbol v_2 \rVert_\infty \leq \varepsilon$, we have established (i).

\medskip\noindent
(ii) The proof focusses on the diagonal blocks $A_1^2$, $A_2^2$, $A_3^2$. We have
\begin{equation}
h^*_{u \to r^{\sigma_{\boldsymbol v_1}}}(x,y,t) = 
\begin{cases}
f^*_{1 \to 1}(t), \qquad &(x,y) \in A_1^2, \\
f^*_{0 \to 0}(t), \qquad &(x,y) \in A_2^2, \\
f^*_{\varepsilon \to \varepsilon}(t), \qquad &(x,y) \in A_3^2,
\end{cases}
\end{equation}
and
\begin{equation}
h^*_{u \to r^{\sigma_{\boldsymbol v_2}}}(x,y,t) = 
\begin{cases}
f^*_{1 \to 1}(t), \qquad &(x,y) \in A_1^2, \\
f^*_{0 \to \varepsilon}(t), \qquad &(x,y) \in A_2^2, \\
f^*_{\varepsilon \to 0}(t), \qquad &(x,y) \in A_3^2.
\end{cases}
\end{equation}
Observe that, for almost all values of $t \in (0,T)$, $f^*_{0 \to 0}(t)$ is different from $f^*_{1 \to 1}(t)$, $f^*_{0 \to \varepsilon}(t)$, and $f^*_{\varepsilon \to 0}(t)$. Fix one such value of $t$, and let 
\begin{equation}
 C^*:= \min \Big\{ | f^*_{0 \to 0}(t) - f^*_{1 \to 1}(t)|, \, |f^*_{0 \to 0}(t) - f^*_{0 \to \varepsilon}(t)|, 
 \, | f^*_{0 \to 0}(t) - f^*_{\varepsilon \to 0}(t) | \Big\}>0.
\end{equation}
For $\phi \in \mathscr{M}$ and $i=1,2,3$, let $L^\phi_i=\{ \phi(x) \in A_2\colon\,x \in A_i \}$. Since $u,r \in \mathscr{W}^{(3)}$ and ${\rm Leb}(A_2)=\tfrac25$, for any $\phi \in \mathscr{M}$ there exists an $i$ such that ${\rm Leb}(L_i^\phi)\geq \tfrac{2}{15}$. Consequently, 
\begin{equation}
\begin{aligned}
&d_\square\Big(h^*_{u \to r^{\sigma_{\boldsymbol v_1}}}(\cdot,\cdot,t), 
\big(h^*_{u \to r^{\sigma_{\boldsymbol v_2}}}(\cdot,\cdot,t)\big)^\phi\Big)\\
&\qquad \geq \left| 
\int_{L_i^\phi \times L_i^\phi} \ddd x\,\ddd y\, \left[ h^*_{u \to r^{\sigma_{\boldsymbol v_1}}}(x,y,t) 
-(h^*_{u \to r^{\sigma_{\boldsymbol v_2}}})^\phi(x,y,t) \right] \right| \geq C^* \left(\tfrac{2}{15}\right)^2.
\end{aligned}
\end{equation}
Since this bound is uniform in $\phi$, we have $\delta_\square(\tilde h^*_{u \to r^{\sigma_{\boldsymbol v_1}}}(\cdot,\cdot,t), \tilde h^*_{u \to r^{\sigma_{\boldsymbol v_2}}}(\cdot,\cdot,t)) >0$. 
\qed


\end{document}